\documentclass[twoside, 11pt]{amsart}

\usepackage[T1]{fontenc}
\usepackage[cal=boondoxo]{mathalfa} % mathcal from STIX, unslanted a bit

\usepackage[utf8]{inputenc}

\usepackage{latexsym}
\usepackage{amscd}
\usepackage{amsmath}
\usepackage{amssymb}
\usepackage{amsthm}
\usepackage{graphicx}
\usepackage{mathrsfs}
\usepackage{mathtools}
\usepackage{pict2e}
\usepackage{stackrel}
\usepackage{wasysym}
\usepackage[all]{xy}
\usepackage[dvipsnames]{xcolor}
\usepackage[colorlinks,final,hyperindex]{hyperref}
\usepackage[noabbrev,capitalize]{cleveref}
\usepackage{tikz}
\usepackage{tikz-cd}
\usepackage{wasysym}
\usepackage[all]{xy}
\usepackage{xcolor}
\usepackage{xy}
\usepackage{CJKutf8}
\usepackage{ipaex-type1}

\usepackage{pdfpages}

\usetikzlibrary{decorations.pathmorphing,decorations.markings,arrows,calc,shapes.geometric,arrows.meta,positioning}
\tikzset{math3d/.style=
    {x= {(-0.353cm,-0.353cm)}, z={(0cm,1cm)},y={(1cm,0cm)}}}
\tikzset{JLL3d/.style=
    {x= {(0.4cm,-0.2cm)}, z={(0cm,1cm)},y={(-1cm,0cm)}}}

\definecolor{Chocolat}{rgb}{0.36, 0.2, 0.09}
\definecolor{BleuTresFonce}{rgb}{0.215, 0.215, 0.36}
\definecolor{BleuMinuit}{RGB}{0, 51, 102}

\hypersetup{citecolor=BleuMinuit, linkcolor=Chocolat, filecolor=black, urlcolor=BleuMinuit}

\usepackage{palatino}
\usepackage{mathpazo} 
\allowdisplaybreaks

\newtheorem{lemma}{Lemma}[section]
\newtheorem{theorem}[lemma]{Theorem}
\newtheorem{corollary}[lemma]{Corollary}
\newtheorem{proposition}[lemma]{Proposition}

\newtheorem*{notation}{Notation}

\theoremstyle{definition}
\newtheorem{definition}[lemma]{Definition}
\newtheorem{defprop}[lemma]{Definition-Proposition}
\newtheorem{remark}[lemma]{Remark}
\newtheorem{example}[lemma]{Example}
\newtheorem{problem}{Problem}

\newcommand{\colim}[1]{\underset{#1}{\mathrm{colim}}~}

\newcommand{\kk}{\Bbbk}

\newcommand{\qi}{\xrightarrow{
   \,\smash{\raisebox{-0.65ex}{\ensuremath{\scriptstyle\sim}}}\,}}
  
\newcommand{\cob}{\mathcal{C}o\mathcal{B}}
\newcommand{\pab}{\mathcal{P}a\mathcal{B}}
\newcommand{\pacd}{\mathcal{P}a\mathcal{CD}}

\author{Damien Calaque and Victor Roca i Lucio}
\title{Associators from an operadic point of view}
\date{\today}

\address{Damien Calaque, IMAG, Universit\'e de Montpellier, CNRS, Montpellier, France}
\email{\href{mailto:damien.calaque@umontpellier.fr}{damien.calaque@umontpellier.fr}}
\address{Victor Roca i Lucio, Ecole Polytechnique Fédérale de Lausanne, EPFL,
CH-1015 Lausanne, Switzerland}
\email{\href{mailto:victor.rocalucio@epfl.ch}{victor.rocalucio@epfl.ch}}

\thanks{
\indent$\,\! $}


\begin{document}

\begin{abstract}
This is a survey on Drinfeld associators and their generalizations, where we focus on operadic aspects. 
\end{abstract}

\maketitle

\makeatletter
\def\@tocline#1#2#3#4#5#6#7{\relax
	\ifnum #1>\c@tocdepth % then omit
	\else
	\par \addpenalty\@secpenalty\addvspace{#2}%
	\begingroup \hyphenpenalty\@M
	\@ifempty{#4}{%
		\@tempdima\csname r@tocindent\number#1\endcsname\relax
	}{%
		\@tempdima#4\relax
	}%
	\parindent\z@ \leftskip#3\relax \advance\leftskip\@tempdima\relax
	\rightskip\@pnumwidth plus4em \parfillskip-\@pnumwidth
	#5\leavevmode\hskip-\@tempdima
	\ifcase #1
	\or\or \hskip 1em \or \hskip 2em \else \hskip 3em \fi%
	#6\nobreak\relax
	\hfill\hbox to\@pnumwidth{\@tocpagenum{#7}}\par% <---- \dotfill -> \hfill
	\nobreak
	\endgroup
	\fi}
\makeatother

\setcounter{tocdepth}{1}
\tableofcontents

%%%%% Introduction starts here %%%%%

\section{Introduction: Drinfeld associators}\label{sec-intro}

\subsection{A deformation quantization problem}

One of the original motivations for the introduction of associators by Drinfeld in \cite{Dri90} was a problem of deformation quantization, occuring in quantum group theory. 
In short terms, it asks for a universal quantization of infinitesimally braided monoidal categories; we are going to formulate the problem precisely. 

\medskip

Let $\mathfrak{g}$ be a Lie algebra over $\kk$, a field of characteristic zero. 
The category $\mathsf{Rep}(\mathfrak{g})$ of $\mathfrak{g}$-representations is a $\kk$-linear 
symmetric monoidal category, with monoidal product the tensor product of representations, and monoidal unit the trivial representation $\kk$. 

\medskip

Let now $t$ be a Casimir element, that is, an element in $S^2(\mathfrak{g})^\mathfrak{g}$. 
\begin{remark}\label{remark inf braiding}
Notice that $S^2(\mathfrak{g})^\mathfrak{g} \subset (\mathfrak{g} \otimes \mathfrak{g})^\mathfrak{g} \subset (\mathfrak{U}(\mathfrak{g})^{\otimes 2})^\mathfrak{g}~$, 
therefore $t$ acts naturally on any tensor product of two representations of $\mathfrak{g}$. The induced natural transformation from $\otimes$ to itself is called an \textit{infinitesimal braiding}, 
and turns $\mathsf{Rep}(\mathfrak{g})$ into an \textit{infinitesimally braided monoidal category}. Infinitesimally braided monoidal categories are braided monoidal $(\kk[\hbar]/\hbar^2)$-linear categories 
that are symmetric monoidal mod $\hbar$; i.e.~they are first order braided deformations of symmetric monoidal $\kk$-linear categories. 
They can be viewed as Gerstenhaber algebras (also known as $\mathbb{P}_2$-algebra) inside a suitable $2$-category of $\kk$-linear categories. 
\end{remark}

\begin{problem}\label{Problem 1}
Given a pair $(\mathfrak{g},t)$ as above, find a formal deformation of the symmetric monoidal category 
$(\mathsf{Rep}(\mathfrak{g}),\otimes, \kk)$ as a braided monoidal category over $\kk[[\hbar]]$ such that: 

\begin{enumerate}
\item The braiding $\sigma$ can be written as $\sigma=R\tau$, where $\tau$ is the natural symmetry isomorphism of $\mathsf{Rep}(\mathfrak{g})$ and 
\[
R = 1 \otimes 1 + \frac{t}{2}\hbar + O(\hbar^2) \quad \text{in} \quad (\mathfrak{U}(\mathfrak{g})^{\otimes 2})^\mathfrak{g}[[\hbar]] ~. 
\]
\item The associator $\Phi$ can be written as\footnote{Strictly speaking, the associator should be $\Phi a$, where $a$ is the natural associativity isomorphism of $\mathsf{Rep}(\mathfrak{g})$. 
But thanks to the Mac Lane coherence theorem, we can pretend that $a$ is the identity. \label{footnote_1}} 
\[
\Phi = 1 \otimes 1 \otimes 1 + O(\hbar) \quad \text{in} \quad (\mathfrak{U}(\mathfrak{g})^{\otimes 3})^\mathfrak{g}[[\hbar]] ~. 
\]
\end{enumerate}
\end{problem}

\begin{remark}
In view of \cref{remark inf braiding}, \cref{Problem 1} asks for the existence of a formal extension to all orders of the first order braided deformation of 
$\mathsf{Rep}(\mathfrak{g})$ given by $t$. 
Note that braided monoidal categories can be viewed as $\mathbb{E}_2$-algebras in the $2$-category of categories. Moreover, the Gerstenhaber operad being the 
homology of the $\mathbb{E}_2$ operad, one can state the problem of quantizing a Gerstenhaber algebra into an $\mathbb{E}_2$-algebra. 
In other words, \cref{Problem 1} amounts to searching for a deformation quantization of the Gerstenhaber algebra $\mathsf{Rep}(\mathfrak{g})$ as an 
$\mathbb{E}_2$-algebra in a suitable $2$-category of $\kk$-linear categories. 
\end{remark}

\subsection{Universal reformulation}\label{ssec-1.2-universal}

Drinfeld noticed in \cite{Dri90} that one could restrict to the case where $R = e^{\frac{\hbar t}{2}}$ and where $\Phi$ is given by a universal formula involving only 
copies of $t$ and iterated Lie brackets. This is analogous to what happens with the Baker--Campbell--Hausdorff formula, that only involves iterated Lie brackets 
of two elements of a given Lie algebra. 
In order to reformulate this formal deformation problem in a universal manner, one first introduces the family $(\mathfrak{t}_n)_{n\geq 0}$ of \textit{Drinfeld--Kohno} 
Lie algebras. In analogy with the Baker--Campbell--Hausdorff universal formula living in the free complete Lie algebra on two generators, a universal formula for 
$\Phi$ shall live in a completion of the universal enveloping algebra of $\mathfrak{t}_3$. 

\begin{definition}[Drinfeld--Kohno Lie algebras]\label{Kohno-Drinfeld algebras}
The $n$-th \textit{Drinfeld--Kohno} Lie algebra $\mathfrak{t}_n$ is given by the following presentation:
\[
\mathfrak{t}_n \coloneqq \frac{\mathcal{L}ie\{t_{ij}\,|\,1\leq i,j \leq n,i \neq j\}}{\left(t_{ij} = t_{ji}~, ~ [t_{ij},t_{ik}+t_{kj}] = 0~,~ [t_{ij},t_{kl}] = 0\right)}~.
\]
Assigning degree $1$ to generators turns $\mathfrak{t}_n$ into a graded Lie algebra. 
\end{definition}
Note that $\mathfrak{t}_0=\mathfrak{t}_1=\{0\}$. 

The $n$-th Drinfeld--Kohno Lie algebra is indeed universal in the following sense: 
\begin{lemma}\label{lemma: morphisms pushing forward}
Let $\mathfrak{g}$ be a Lie algebra together with a Casimir element
\[
t = \sum_{\alpha \in \mathrm{I}} a_{\alpha} \otimes b_{\alpha}\in S^2(\mathfrak{g})^\mathfrak{g}~.
\]
The assignment 
\[
\begin{tikzcd}[column sep=3pc,row sep=0pc]
\varphi_t^n: \mathfrak{t}_n \arrow[r]
& (\mathfrak{U}(\mathfrak{g})^{\otimes n})^{\mathfrak{g}} \\
t_{ij} \arrow[r,mapsto]
& \displaystyle \sum_{\alpha \in \mathrm{I}} a^{(i)}_{\alpha}b^{(j)}_{\alpha}
\end{tikzcd}
\] 
defines a morphism of Lie algebras. Here, $x^{(i)}:= 1^{\otimes i -1 } \otimes x \otimes 1^{\otimes n-i}$, and $\mathfrak{U}(\mathfrak{g})^{\otimes n}$ is 
endowed with the Lie bracket given by the commutator of the associative product. 
\end{lemma}

\begin{proof}
It is enough to check that the images of the generators $t_{ij}$ satisfy the relations defining Drinfeld--Kohno Lie algebra. 
Since $t$ is symmetric, $\varphi_t^n(t_{ij}) = \varphi_t^n(t_{ji})$. 
One can check that since $t$ is $\mathfrak{g}$-invariant, we have 

\[
[\varphi_t^n(t_{ij}),\varphi_t^n(t_{ik})+\varphi_t^n(t_{kj})] 
= \sum_{\alpha, \beta \in \mathrm{I}} \left( \left[ a_{\alpha}^{(i)}b_{\alpha}^{(j)}, a_{\beta}^{(i)}b_{\beta}^{(k)} \right] 
+ \left[ a_{\alpha}^{(i)}b_{\alpha}^{(j)}, a_{\beta}^{(j)}b_{\beta}^{(k)} \right] \right) = 0~.
\]

The relation $[\varphi_t^n(t_{ij}),\varphi_t^n(t_{kl})] = 0$ follows from the definition of $\varphi_t^n$.
\end{proof}

Since the universal enveloping algebra is left adjoint, we get a universal morphism 
$\varphi_t^n: \mathfrak{U}(\mathfrak{t}_n) \longrightarrow (\mathfrak{U}(\mathfrak{g})^{\otimes n})^{\mathfrak{g}}$ of associative algebras 
for any pair $(\mathfrak{g},t)$. This indicates that a universal reformulation of \cref{Problem 1} can be made using the Drinfeld--Kohno Lie algebras. 

\begin{remark}\label{completed variant}
A useful variant of the above, where one sends $t_{ij}$ to 
\[
\hbar  \sum_{\alpha \in \mathrm{I}} a^{(i)}_{\alpha}b^{(j)}_{\alpha}
\]
gives an algebra morphism $\varphi_{\hbar t}^n:\hat{\mathfrak{U}}(\mathfrak{t}_n)\to (\mathfrak{U}(\mathfrak{g})^{\otimes n})^{\mathfrak{g}}[[\hbar]]$, 
where $\hat{\mathfrak{U}}(\mathfrak{t}_n)$ is the degree completion with respect to the grading from Definition \ref{Kohno-Drinfeld algebras} (given by assigning 
degree $1$ to the generators $t_{ij}$). 
\end{remark}

The family of Drinfeld--Kohno Lie algebras $(\mathfrak{t}_n)_{n\geq0}$ comes equipped with an extra structure, often called \textit{insertion-coproduct morphisms}. 
\begin{defprop}[Insertion-coproduct morphisms]\label{def: insertion-coproduct morphisms}
For every partially defined map $f:\{1,\dots,m\}\to\{1,\dots,n\}$, that one can define as a pointed map $\{*,1,\dots,m\}\to \{*,1,\dots,m\}$, there is a Lie algebra morphism $(-)^f:\mathfrak{t}_n \longrightarrow \mathfrak{t}_m$
that sends $t_{ij}$ to 
\[
t_{ij}^f:=\sum_{\substack{k\in f^{-1}(i) \\ l\in f^{-1}(j)}}t_{kl}~.
\]
\end{defprop}
Insertion-coproduct morphisms are compatible with composition, making $n\mapsto \mathfrak{t}_n$ a presheaf of graded Lie algebras on (the skeleton of) 
the category $\mathrm{Fin}_*$ of finite pointed sets
\begin{notation}
Since a partially defined map $f:\{1,\dots,m\}\to\{1,\dots,n\}$ is completely determined by the familly of level sets $(f^{-1}(i))_{1\leq i\leq n}$, one sometimes 
write $(-)^{f^{-1}(1),\dots,f^{-1}(n)}=(-)^f$. 
For instance, if $x\in\mathfrak{t}_{3}$, then $x^{13,2,5}\in \mathfrak{t}_{6}$ corresponds to $x^f$ for $f:\{1,\dots,6\}\to\{1,2,3\}$ defined by 
$f^{-1}(1)=\{1,3\}$, $f^{-1}(2)=\{2\}$, and $f^{-1}(3)=\{5\}$. 
\end{notation}
Observe that $(\mathfrak{U}(\mathfrak{g})^{\otimes n})^{\mathfrak{g}}$ is isomorphic to the algebra $\mathsf{End}(\otimes^n)$ of natural endomorphisms of 
the $n$-fold tensor product functor $\otimes^n:\mathsf{Rep}(\mathfrak{g})^n\to \mathsf{Rep}(\mathfrak{g})$.\footnote{Strictly speaking, one should fix a choice 
of parenthesization on $\otimes^n$. \label{footnote_2}} 
For every partially defined map $f:\{1,\dots,m\}\to\{1,\dots,n\}$ there is a functor $f_*:\mathsf{Rep}(\mathfrak{g})^m\to \mathsf{Rep}(\mathfrak{g})^n$ 
taking $(V_1,\dots,V_m)$ to $(W_1,\dots, W_n)$ with $W_i:=\otimes_{j\in f^{-1}(i)}V_j$. Precomposing with $f_*$ gives an algebra map 
$(-)^f:\mathsf{End}(\otimes^n)\to \mathsf{End}(\otimes^m)$, making 
$n \mapsto \mathsf{End}(\otimes^n)\cong(\mathfrak{U}(\mathfrak{g})^{\otimes n})^{\mathfrak{g}}$ a presheaf of associative algebras on $\mathrm{Fin}_*$. 

We let the reader check the following properties: 
\begin{itemize}
\item The morphisms $(-)^f:(\mathfrak{U}(\mathfrak{g})^{\otimes n})^{\mathfrak{g}}\to (\mathfrak{U}(\mathfrak{g})^{\otimes m})^{\mathfrak{g}}$ can 
be expressed using coproducts, insertion of $1$'s, and counits (that one can see as degenerate iterated coproducts), whence the name 
\textit{insertion-coproduct morphisms};
\item The morphisms $\varphi_t^n: \mathfrak{U}(\mathfrak{t}_n) \longrightarrow (\mathfrak{U}(\mathfrak{g})^{\otimes n})^{\mathfrak{g}}$ from 
Lemma \ref{lemma: morphisms pushing forward} commute with insertion-coproduct morphisms, thus defining a morphism of presheaves of algebras on 
$\mathrm{Fin}_*$. The same remains true with the variant $\varphi_{\hbar t}^n$ from Remark \ref{completed variant}. 
\end{itemize}
Going back to \cref{Problem 1}, the conditions that $R\in \mathsf{Aut}(\otimes^2)$ and $\Phi\in \mathsf{Aut}(\otimes^3)$ must satisfy\footnote{Following footnotes \ref{footnote_1} and \ref{footnote_2}, 
stricly speaking, one shall have $\Phi\in \mathsf{Aut}((-\otimes-)\otimes-)$. } to define a braided monoidal structure on $\mathsf{Rep}(\mathfrak{g})$, where the monoidal product 
functor $\otimes$ and the monoidal unit remain the same, can be rephrased in terms of insertion-coproduct morphisms. We refer to \cite{joyalbraided} for the original conditions that we now state in this ``new'' way: 
\begin{enumerate}
\item Unit condition: $\Phi^{1,\emptyset,2}=\mathrm{id}$, in $\mathsf{End}(\otimes^2)$. 
\item Inverse condition: $\Phi^{-1}=\Phi^{3,2,1}$, in $\mathsf{End}(\otimes^3)$. 
\item Pentagon equation: $\Phi^{2,3,4}\Phi^{1,23,4} \Phi^{1,2,3} = \Phi^{1,2,34}\Phi^{12,3,4}$, in $\mathsf{End}(\otimes^4)$. 
\item Hexagon equations: $\Phi^{2,3,1}R^{1,23}\Phi^{1,2,3}=R^{1,3}\Phi^{2,1,3}R^{1,2}$, and the same with $(R^{2,1})^{-1}$ instead of $R$, in $\mathsf{End}(\otimes^3)$.  
\end{enumerate}

\medskip 

This naturally leads to the following universal version of \cref{Problem 1}: 
\begin{problem}\label{Problem 2}
Find an invertible element $\Phi\in 1+\hat{\mathfrak{U}}(\mathfrak{t}_3)^{\geq1}$ such that 
\begin{enumerate}
\item Unit condition: $\Phi^{1,\emptyset,2}=1$, in $\hat{\mathfrak{U}}(\mathfrak{t}_2)$. 
\item Inverse condition: $\Phi^{-1}=\Phi^{3,2,1}$, in $\hat{\mathfrak{U}}(\mathfrak{t}_3)$. 
\item Pentagon equation: $\Phi^{2,3,4}\Phi^{1,23,4} \Phi^{1,2,3} = \Phi^{1,2,34}\Phi^{12,3,4}$, in $\hat{\mathfrak{U}}(\mathfrak{t}_4)$. 
\item Hexagon equations: $\Phi^{2,3,1}e^{\pm\frac{t_{12}+t_{13}}{2}}\Phi^{1,2,3}=e^{\pm\frac{t_{13}}{2}}\Phi^{2,1,3}e^{\pm\frac{t_{12}}{2}}$ in $\hat{\mathfrak{U}}(\mathfrak{t}_3)$.  
\end{enumerate}
\end{problem}
It is clear from the above discussion that every solution to \cref{Problem 2} is sent, using the morphism $\varphi_{\hbar t}^\bullet$ of presheaves of algebras on 
$\mathrm{Fin}_*$, to a solution of Problem \ref{Problem 1} such that $R=e^{\frac{\hbar t}{2}}$. 
%Having this dictionary at hand, it is easy to see that the \textit{pentagon equation} of an associator given by 
%\[
%\begin{tikzcd}[column sep=3pc,row sep=2.5pc]
%&((1 \otimes 2) \otimes 3) \otimes 4 \arrow[ld,"\Phi^{12,3,4}",swap] \arrow[rd,"\Phi^{1,2,3}"]
%& \\
%(1 \otimes 2) \otimes (3 \otimes 4) \arrow[d,"\Phi^{1,2,34}",swap]
%&
%&(1 \otimes (2 \otimes 3)) \otimes 4 \arrow[d,"\Phi^{1,23,4}"]\\
%1 \otimes (2 \otimes (3 \otimes 4)) 
%&
%&1 \otimes ((2 \otimes 3)) \otimes 4) \arrow[ll,"\Phi^{2,3,4}"]
%\end{tikzcd}
%\]
%translates into the following equation 
%\[
%d_0\Phi.d_2\Phi. d_4\Phi = d_3\Phi. d_1\Phi
%\]
%satisfied inside $\mathfrak{U}(\mathfrak{t}_4)~.$
%\begin{remark}
%The pentagon equation is often simply written as 
%\[
%\Phi^{1,2,3}.\Phi^{1,23,4}.\Phi^{2,3,4} = \Phi^{1,2,34}. \Phi^{12,3,4}
%\]
%in the literature, regardless of the context. 
%\end{remark}

\subsection{Drinfeld associators} \label{subsection: drinfeld associator}
Observe that the center of $\mathfrak{t}_3$ is one dimensional and generated by $c:=t_{12} + t_{13} + t_{23}$. Hence we have a Lie algebra isomorphism 
\[
\mathfrak{t}_3\cong \mathfrak{f}_2\oplus \kk c\,,
\]
where $\mathfrak{f}_2$ is the free Lie algebra on two generators $x=t_{12}$ and $y=t_{23}$. 
Through this identification, the Lie algebra morphism $(-)^{1,\emptyset,2}:\mathfrak{t}_3\to \mathfrak{t}_2$ sends $x$ and $y$ to $0$ and $c$ to $t_{12}$.

Given a solution $\Phi$ of \cref{Problem 2}, the unit condition (1) is equivalent to requiring that 
\[
\Phi=\Phi(x,y)\in \hat{\mathfrak{U}}(\mathfrak{f}_2)\subset \hat{\mathfrak{U}}(\mathfrak{t}_3)\,.
\]
Then the inverse condition (2) reads as $\Phi(x,y)^{-1}=\Phi(y,x)$, and the pentagon equation (3) as 
\begin{align}\label{eq-pentagon}\tag{\pentagon}
\Phi(t_{23},t_{34})\Phi(t_{12}+t_{13},t_{24}+t_{34})\Phi(t_{12},t_{23}) \\ 
= \Phi(t_{12}, t_{23}+t_{24})\Phi(t_{13}+t_{23},t_{34})\,. \nonumber
\end{align}
Finally, the hexagon equations (4) become equivalent to the single equation 
\begin{equation}\label{eq-hexagon}\tag{\hexagon}
e^{x/2}\Phi(y,x)e^{y/2}\Phi(z,y)e^{z/2}\Phi(x,z)=1
\end{equation}
in the complete associative algebra $\kk\langle\!\langle x,y,z\rangle\!\rangle/(x+y+z)$. 

\begin{definition}[Drinfeld associator]\label{def: classical associator}
A Drinfeld $1$-associator (with coefficients in $\kk$) is a group-like element $\Phi(x,y)\in \exp(\widehat{\mathfrak{f}}_2)$  such that $\Phi(x,y)^{-1}=\Phi(y,x)$ 
and sastisfying \eqref{eq-pentagon} and \eqref{eq-hexagon}. The set of Drinfeld $1$-associators (with coefficients in $\kk$) is denoted $\mathrm{Assoc}_1(\kk)$.
\end{definition}
From the above discussion, one sees that Drinfeld $1$-associators are in bijection with group-like solutions to \cref{Problem 2}. 

\begin{remark}
One can more generally define Drinfeld $\lambda$-associators for any $\lambda\in\kk$: all conditions remain the same except for equation \eqref{eq-hexagon}, 
that becomes 
\[
e^{\lambda x/2}\Phi(y,x)e^{\lambda y/2}\Phi(z,y)e^{\lambda z/2}\Phi(x,z)=1\,. 
\]
For every $\lambda\neq0$, rescaling $x$ and $y$ (to $\lambda x$ and $\lambda y$) gives a bijection between $\mathrm{Assoc}_1(\kk)$ and the set 
$\mathrm{Assoc}_\lambda(\kk)$ of Drinfeld $\lambda$-associators. 
\end{remark}

\begin{theorem}[{\cite{Dri90}}]
Let $\kk$ be a field of characteristic zero. The set of Drinfeld $1$-associators with coefficients in $\kk$ is non empty. 
\end{theorem}
In \cite{Dri90}, Drinfeld first constructs a $2\pi\mathrm{i}$-associator $\Phi_{KZ}$ with coefficients in $\mathbb{C}$, given as the renormalized holonomy from $0$ 
to $1$ of a differential equation known as the Knizhnik--Zamolodchikov equation. Then, one can rescale it to get a $1$-associator with coefficients in $\mathbb{C}$. 
One can finally use descent methods to prove existence over $\mathbb{Q}$ (this uses the fact that associators actually form a torsor). 

\vspace{2pc}

\subsection{Motivations and perspectives} 

\subsubsection*{The operadic approach}

The insertion-coproduct morphisms defined above can be used to endow the family $(\mathfrak{t}_n)_{n\geq0}$ with an operad structure in the symmetric 
monoidal category of Lie algebras with the direct sum $(\mathsf{Lie},\oplus,0)$. This will allow us to restate the above definition in an operadic fashion 
(see \cref{Section: Drinfeld associators}). Indeed, using this operad structure, one can write the set of associators as a set of isomorphisms between two 
different operads. This set is thus a torsor over the respective automorphisms groups, which 
are none other than the \textit{Grothendieck--Teichmüller} group and the \textit{graded Grothendieck--Teichmüller} group. It is Tamarkin who, inspired by Bar-Natan's work 
\cite{BarNatan99}, first came up with this operadic picture in his proof \cite{Tamarkin03} of the rationnal formality of the little disks operad se also \cite{Tamarkin02}). 

The work of Tamarkin \cite{Tamarkin03,Tamarkin02} (see also \cite{Kontsevich99}) actually suggests a homotopical version of this operadic approach to the torsor of associators 
and the \textit{Grothendieck--Teichmüller} groups, that has been achieved in \cite{FresseGT2}. 

\subsubsection*{Motivic aspects}

Drinfeld's proof \cite{Dri90} of the existence of an associator with complex coefficients is of motivic nature, in the 
following sense. An associator, called the Knizhnik--Zamolodchikov (KZ) associator, is constructed as the regularized holonomy of an algebraic 
flat connection defined on $\mathbb{P}^1-\{0,1,+\infty\}$ (the KZ connection). The coefficients of the KZ associator $\Phi_{KZ}$ are therefore periods (in the 
sense of \cite{Kontsevich-Zagier}); in fact, $\Phi_{KZ}$ is a generating series for very specific periods known as multiple zeta values. 
We refer to \cite{Le-Murakami} for a proof, and \cite{Zagier} for generalities on multiple zeta values (MZV). 
The proof that $\Phi_{KZ}$ satisfies the defining equations of a Drinfeld ($2\pi i$-)associator only involves ``natural'' relations between these 
periods (because they are all proven by saying that the holonomy of a flat connection along a contractible loop is trivial); 
natural (actually, \textit{motivic}) relations are linearity, change of variables and Stokes formula (see e.g.~\cite{Kontsevich-motives}). 

\medskip

Brown \cite{Brown} proved that MZV linearly span (over $\mathbb{Q}[\frac{1}{2\pi i}]$) all periods of mixed Tate motives that are 
unramified over $\mathbb{Z}$. This implies that there is an injective morphism 
$G_{\mathcal M_T(\mathbb{Z})_{\mathbb{Q}}}\hookrightarrow \mathrm{GRT}(\mathbb{Q})$, where 
\begin{itemize}
\item $G_{\mathcal M_T(\mathbb{Z})_{\mathbb{Q}}}$ is the Galois group of the tannakian category $\mathcal M_T(\mathbb{Z})_{\mathbb{Q}}$ of mixed Tate motives 
that are unramified over $\mathbb{Z}$ (see \cite{goncharov2001multiple,Deligne-Goncharov}); 
\item $\mathrm{GRT}(\mathbb{Q})$ is the graded Grothendieck--Teichmüller group already mentionned above (see also \S\ref{ssec-1.5} below). 
\end{itemize}
It turns out that both $G_{\mathcal M_T(\mathbb{Z})_{\mathbb{Q}}}$ and $\mathrm{GRT}(\mathbb{Q})$ are semi-direct products of the multiplicative group 
$\mathbb{G}_m$ with a pro-unipotent $\mathbb{Q}$-group. More importantly, Goncharov \cite{goncharov2001multiple} proved that in the case of $G_{\mathcal M_T(\mathbb{Z})_{\mathbb{Q}}}$ the graded Lie algebra of the associated pro-unipotent 
$\mathbb{Q}$-group is free with one generator in each odd degree $\geq3$. 

It is conjectured that the morphism $G_{\mathcal M_T(\mathbb{Z})_{\mathbb{Q}}}\to \mathrm{GRT}(\mathbb{Q})$ is an isomorphism. 
This would imply that the relations 
among MZV given by the defining equations of Drinfeld associators imply all the relations of motivic nature between MZV. 

\medskip

There is another set of relations between MZV, called (regularized) double shuffle relations. 
They are essentially combinatorial relations obtained from formally manipulating the two standard ways of defining MZV: 
the one \textit{via} infinite sums and the one \textit{via} iterated integrals. Indeed, the product of two Euler sums is a 
linear combination of Euler sums (indexed by shuffles) and the product of two iterated integrals is a linear combination 
of iterated integrals (indexed by quasi-shuffles). These relations were systematically studied by Racinet \cite{Racinet}, 
who proved that they are implied by the motivic relations. Furusho \cite{Furusho} proved they are also implied by the associator relations. 

\subsubsection*{Known examples of Drinfeld associators}

Apart from the KZ associator $\Phi_{KZ}$ (and its ``complex conjugate'' $\Phi_{\overline{KZ}}$), another associator $\Phi_{AT}$ 
has been constructed by Alekseev and Torossian \cite{Alekseev-Torossian,Severa-Willwacher} using integrals over compactified configuration spaces 
of points in the plane. Using similar techniques, a whole family $\{\Phi^t\}_{t\in[0,1]}$ of Drinfeld associators such that $\Phi^0=\Phi_{KZ}$, 
$\Phi^{1/2}=\Phi_{AT}$ and $\Phi^1=\Phi_{\overline{KZ}}$ was constructed by Rossi and Willwacher \cite{Rossi-Willwacher}. 

\subsubsection*{Generalizations}

The whole story of Drinfleld associators is very much related to the geometry of configuration spaces of points on a genus zero surface. 
Variants of the KZ connection for a higher genus surface $\Sigma$ have been constructed, as well as twisted versions by a finite group $\Gamma$
acting on the surface. In several cases, one can guess the correct definition of an associator for such a data by inspecting the natural 
relations satisfied by holonomies of the corresponding connection. A more systematic approach for defining and studying associators uses operads 
(operads and their siblings provide a consistent way of organizing abstract relations satisfied by these holonomies). 
The table below summarizes the state of the art, leaving aside the motivic aspects (that are not our main focus in this survey).  

\vspace{1pc}
\hspace{-2.5pc}
\begin{tabular}{|c|c|c|c|c|} 
\hline  
Genus & Group & KZ connect. & Associator type & Operadic def.\\  
\hline 
0 & 1 & \cite{Dri90} & Drinfeld \cite{Dri90} & \cite{Tamarkin03,FresseGT1} \\  
\hline
0 & $\mathbb{Z}/n\mathbb{Z}$ & \cite{Enriquez07} & cyclotomic \cite{Enriquez07} & \cite{calaque20moperadic} \\
\hline
0 & $\Gamma \subset \mathrm{PSU}(2)$ & \cite{Maassarani} & \textit{unknown} & \textit{unknown}  \\
\hline
1 & 1 & \cite{UniversalKZB09} & elliptic \cite{Enriquez14} & \cite{calaque2020ellipsitomic}  \\
\hline 
1 & $\mathbb{Z}/n\mathbb{Z} \times \mathbb{Z}/m\mathbb{Z}$ & \cite{EllipsitomicKZB} & ellipsitomic \cite{calaque2020ellipsitomic} & \cite{calaque2020ellipsitomic}  \\
\hline
$\geq1$ & 1 & \cite{Enriquez14b} & higher genus \cite{GonzalezHigherGenus} & \cite{GonzalezHigherGenus,campos2020configuration}  \\ 
\hline 
\end{tabular}
\vspace{1pc}

Note that in genus $g>1$, since the surface doesn't have a framing, one needs to add the framing data to configuration spaces of points 
(each point now has a unit tangent vector attached to it). To this day, it is still an open problem to prove the existence of an associator 
(in the sense of \cite{GonzalezHigherGenus} in this case) using the connection from \cite{Enriquez14b} (or, better, a framed version). It is nevertheless expected that 
the data of a genuine Drinfeld associator should be enough to construct higher genus associators ; this is the case in genus $1$ (see \cite{UniversalKZB09,Enriquez14}). 
For higher genus associators, this willl appear in \cite{BrochierHigherGenus}. 
A similar result is proven in \cite{campos2020configuration} for a more homotopical definition of higher genus associators that should be compared with the one from 
\cite{GonzalezHigherGenus}.

\bigskip

\begin{center}\fbox{
\begin{minipage}{0.75\textwidth}
\begin{center}
\textbf{The goal of this survey is to review the operadic approach, in genus $0$ and in genus $1$ }
\end{center}
\end{minipage}}
\end{center}

\medskip

\subsection{Plan of the paper}
\cref{Section: Drinfeld associators} is devoted to the first line of the above table. We start with some recollections on braid groups, configurations 
spaces, operads and groupoids, in \S\ref{subsec-1.1} and \S\ref{subsec-1.2}. We then introduce the two main players in our story: 
\begin{enumerate}
\item the (pro-unipotent completion of the) operad of parenthezised braids (\S\ref{subsection: pab}). 
\item the (degree completion of the) operad of parenthezised chord diagrams (\S\ref{ssec-1.4}). 
\end{enumerate}
We move on in \S\ref{ssec-1.5} with the operadic definition of an associator, which is simply an isomorphism from the first player to the second one. 
The set of associators naturally becomes a bitorsor, with the two acting groups being the automorphisms groups of both players. 
The presentation by generators and relations of the operad of parenthezised braids given in \S\ref{ssec-1.6} allows to give an explicit description of the 
bitorsor of associators, and in particular to show that it is consistent with the Definition \ref{def: classical associator} motivated by \cref{Problem 2}. 
Finally, \S\ref{Subsection: Description topo} provides a more topological description of the operad of parenthezised braids. 

\medskip

The aim of \cref{section-cyclotomic} is to tell a parallel story for the second line of the above table: Enriquez's cyclotomic associators \cite{Enriquez07}. 
We start in \S\ref{ssec-cyclotomic-motivation} with some motivations; the reader can view it as a condensed version of our introductory \cref{sec-intro} 
for the cyclotomic case. We then review moperads in \S\ref{ssec-Moperads}, which were introduced by Willwacher in \cite{Willwachermoperads}, and which play an analogue role to operads in the cyclotomic picture. Furthermore, we prove that this structure naturally appears when one considers the derivative (in the sense of species) of an operad. This construction recovers many of the examples that appear in this text. In \S\ref{ssec-3.3} we construct the moperad of parenthesized cyclotomic braids, which will be the first player of this story. And in \S\ref{subsec: inf cyclo braids}, we construct the moperad of infinitesimal cyclotomic braids, which is our second player. This allows us to define cyclotomic associators purely in terms of isomorphisms of moperads in \S\ref{ssec-Cyclo associator and gts}. In Subsection \ref{ssec-more concrete description of cyclos}, we give an explicit description of these cyclotomic associators and finally in \S\ref{subsec: topological cyclotomic} we give a topological interpretation of some of the constructions performed. 

\medskip

Finally, in Section \ref{Section: Elliptic}, we tell the story of the fourth line of the table and give a quick overview of the fifth line. The structure remains the same. We start with some motivation in \S\ref{ssec: Motivations elliptic}. We introduce the first player, the right module of parenthesized elliptic braids in \S\ref{ssec: parenthesized elliptic braids}; then we introduce the second player, the right module of infinitesimal elliptic braids in \S\ref{ssec-4.3-ellipticCD}. We then define elliptic associators as right module isomorphisms of the two in \S\ref{ssec: elliptic associators}. Using a presentation of the right module of parenthesized elliptic braids, we give a more concrete description of elliptic associators in \S\ref{ssec: more concrete elliptic}. We then present the topological point of view on these objects in \S\ref{ssec: topological description elliptic}. Finally, we give a quick overview on the ellipsitomic case in \S\ref{ssec: overview of the ellipsitomic case}: this case can be essentially thought as a combination of the elliptic case with the cyclotomic case. 

\subsection{Acknowledgements}
This survey emerged from a series of lectures given in June 2021 by DC at the ``workshop on higher structures and operadic calculus''. 
Both authors thank the organizers and the participants for their enthusiasm, as well as Adrien Brochier for his comments and suggestions. 

DC has received funding from the European Research Council (ERC) under the European Union’s Horizon 2020 research and innovation programme 
(grant agreement No 768679).

\subsection{Conventions}\label{ssec-conventions}
Let us state here some of the conventions we will adopt throughout this survey.

\medskip

\begin{enumerate}

\item \textit{Convention for the generators of the pure braid groups:} We choose the following generators for the pure braid groups $\mathrm{PB}_n$: $x_{ij}$ is the pure braid that goes \textit{from} the $i$-th strand \textit{in front} of the others, does a loop around the $j$-th strand, and comes back \textit{in front} of the other strands:  
\[
\includegraphics[width=50mm,scale=1]{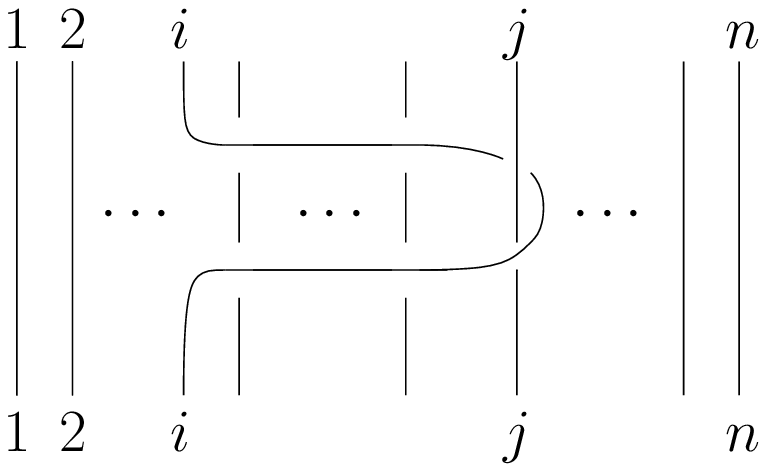}
\]

\medskip

\item \textit{Topological convention for categories:} We view braid groups as homotopy groups, where the multiplication is given by the concatenation of paths. 

\medskip

This entails that, when working with categories, we \textit{concatenate} arrows instead of composing them. If $a,b,c$ are objects in a category $\mathcal{C}$, we denote 
\[
\begin{tikzcd}[column sep=1.5pc,row sep=0pc]
a \arrow[r,"f"]
&b \arrow[r,"g"]
&c 
\end{tikzcd}
\]
by $f~g$ instead of $g~f$. In other words, $f~ g:=g\circ f$. 

\medskip

\item \textit{Symmetric group actions:} If $M$ is a symmetric sequence, every $M(n)$ is endowed with a \textit{right} $\mathbb{S}_n$-action, where $\mathbb{S}_n$ stands for the $n$-th symmetric group. Nevertheless, we will often consider a \textit{left} $\mathbb{S}_n$-action, as it simplifies diagrams. In order to pass pass from one to the other, one merely needs to pre-compose the action by the involution which sends an element to its inverse.
\end{enumerate}

\newpage

%%%%%%%%%%%%% SECTION 2 %%%%%%%%%%%%%%%

\section{Operadic approach to Drinfeld associators}\label{Section: Drinfeld associators}

\subsection{Braid groups and configuration spaces}\label{subsec-1.1}

The space of configurations of $n$ ordered points in the complex plane $\mathrm{Conf}_n(\mathbb{C})$ is given by 
\[
\mathrm{Conf}_n(\mathbb{C}) \coloneqq \left\{ (x_1, \cdots, x_n) \in \mathbb{C}^n~~|~~x_i \neq x_j~~\text{if}~~i \neq j \right\}~.
\]
for $n \geq 0$. It is a path-connected topological space, equipped with an action of the symmetric group $\mathfrak{S}_n$ 
(which permutes the indices of the points) for all $n$ in $\mathbb{N}$.

\begin{definition}[Pure braid groups]
The \textit{pure braid group} $\mathrm{PB}_n$ on $n$ strands is given by the fundamental group
\[
\mathrm{PB}_n \coloneqq \pi_1(\mathrm{Conf}_n(\mathbb{C}))~.
\]
\end{definition}

Elements of the pure braid group $\mathrm{PB}_n$ can be represented as braids. One can choose a point $(x_1, \cdots, x_n)$ in $\mathrm{Conf}_n(\mathbb{C})$ 
(since it is path-connected, two choices yield isomorphic groups), and represent elements in 
$\mathrm{PB}_n$ as braids with $n$-strands. Notice that the $i$-th strand of a braid in the pure braid group starts at $i$-th spot and must also finish at $i$-th spot. 
For instance, the following braid 

\begin{center}
\includegraphics[width=20mm,scale=1]{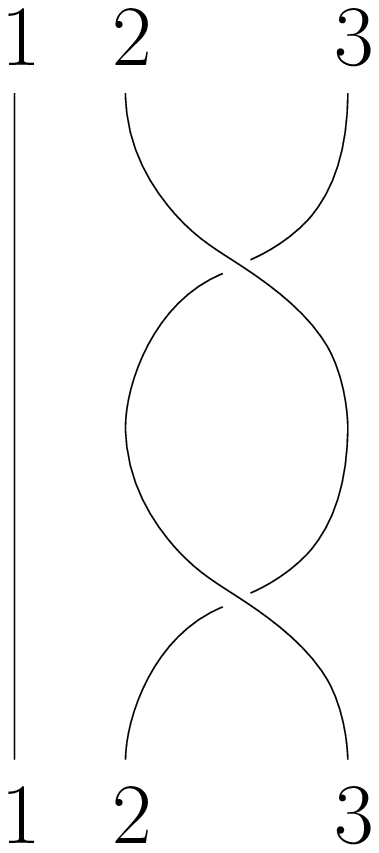}
\end{center}

is an element of $\mathrm{PB}_3$.

\begin{theorem}\label{thm: Artin's presentation}
The pure braid group $\mathrm{PB}_n$ on $n$ strands admits the following presentation
\begin{enumerate}
\item It is generated by $\{x_{ij}\}$ for $1 \leq i < j \leq n$. These \textit{elementary pure braids} can be represented as

\[
\includegraphics[width=50mm,scale=1]{elementxij.eps}
\]

\item The generators are subject to the following relations 

\begin{enumerate}
\item $[x_{ij},x_{kl}] = 1~$ for $i < j < k < l$,

\item $[x_{il},x_{jk}] = 1~$ for $i < j < k < l$,

\item $x_{ik}x_{jk}x_{ij} = x_{jk}x_{ij}x_{ik} = x_{ij}x_{ik}x_{jk}~$ for $i < j < k$,

\item $[x_{kl}x_{ik}x_{kl}^{-1},x_{jl}] = 1~$ for $i < j < k < l$.
\end{enumerate}
\end{enumerate}
\end{theorem}

There is a fibration of topological spaces 
\[
\mathbb{C} - \{x_1,\cdots,x_n\} \hookrightarrow \mathrm{Conf}_{n+1}(\mathbb{C}) \twoheadrightarrow \mathrm{Conf}_n(\mathbb{C})~,
\]
given by forgetting the last point of the configuration. It is split by the morphism $(x_1,\dots,x_n) \mapsto (x_1,\dots,x_n, 1 + \sum_{i = 1}^n |x_i|)$. 
This induces a split short exact sequence of homotopy groups 
\[
\begin{tikzcd}[column sep=1.5pc,row sep=0pc]
1 \arrow[r]
&\mathbb{F}_n \arrow[r]
&\mathrm{PB}_{n+1} \arrow[r]
&\mathrm{PB}_n \arrow[r]
&1~,
\end{tikzcd}
\]
since all the higher homotopy groups are trivial. Here and below, $\mathbb{F}_n$ denotes the free group on $n$ generators. 
In particular, there is the following decomposition
\[
\mathrm{PB}_n \cong \mathbb{F}_n \rtimes (\mathbb{F}_{n-1} \rtimes \cdots (\mathbb{F}_2 \rtimes \mathbb{F}_1)))~,
\]
which follows inductively from the previous split short exact sequence.

\begin{lemma}\label{lemma: PB_3 iso au produit de libres}
There is an isomorphism of groups 
\[
\mathrm{PB}_3 \cong \mathbb{F}_2 \times \mathbb{F}_1~,
\]
where the generators of $\mathbb{F}_2$ are sent to $x_{12}$ and $x_{23}$ and where the generator of $\mathbb{F}_1$ is sent to $x_{12}x_{13}x_{23}$. 
\end{lemma}

\begin{proof}
In this particular case, one can choose another split fibration 
\[
\mathbb{C} - \{x_1,x_3\} \hookrightarrow \mathrm{Conf}_{3}(\mathbb{C}) \twoheadrightarrow \mathrm{Conf}_2(\mathbb{C})~.
\]
Here we forget the second point, and the splitting is given by sending a configuration in $\mathrm{Conf}_2(\mathbb{C})$ to the configuration given 
by keeping the two original points and adding a point in the middle of the two. Given this splitting, one can directly check that the corresponding semi-direct product 
\[
\mathrm{PB}_3 \cong \mathbb{F}_2 \rtimes \mathbb{F}_1
\]
is in fact a direct product (this is because the full twist $x_{12}x_{13}x_{23}$ is central). 
\end{proof}

\subsection{Conventions on operads and groupoids}\label{subsec-1.2}

Let $(\mathcal{E},\otimes,\mathbb{1})$ be symmetric monoidal category where the tensor product commutes with colimits. The category of symmetric sequences $\mathbb{S}\text{-}\mathsf{mod}_{\mathcal{E}}$ is given by collections of objects $(e(n))_{n \geq 0}$ for $n \geq 0$, where $e(n)$ is endowed with a right action of $\mathfrak{S}_n$. 

\medskip

The category $\mathbb{S}\text{-}\mathsf{mod}_{\mathcal{E}}$ can be endowed with a monoidal product $\circ$ called the plethysm product. The category $(\mathbb{S}\text{-}\mathsf{mod}_{\mathcal{E}},\circ,\mathbb{1}_0)$ forms a monoidal category, where $\mathbb{1}_\circ$ is the symmetric sequence given by $\mathbb{1}$ in arity one and the initial object $\emptyset$ elsewhere. \textit{Operads} are defined to be unital monoids in $(\mathbb{S}\text{-}\mathsf{mod}_{\mathcal{E}},\circ,\mathbb{1}_\circ)$.

\medskip

One can also consider the category of sequences in $\mathcal{E}$, indexed by $\mathbb{N}$. This category, denoted by $\mathbb{N}\text{-}\mathsf{mod}_{\mathcal{E}}$, also has a plethysm product which endows it with a monoidal structure. Monoids in this category are called \textit{non-symmetric operads}.

\begin{example}
The category of small sets $\mathsf{Sets}$ together with the cartesian product of sets satisfies these hypothesis. Similarly, the categories of small groups 
$\mathsf{Grp}$, small groupoids $\mathsf{Grpd}$, or topological spaces $\mathsf{Top}$ all satisfy them with the cartesian product again. \hfill$\triangle$
\end{example}

\begin{remark}
Notice that the hypothesis that the tensor product commutes with colimits is too strong in order to include the symmetric monoidal category $(\mathsf{Lie},\oplus,0)$. Nevertheless, since it is cocomplete, one can still define operads as monoids in the normal oplax monoidal category $(\mathbb{S}\text{-}\mathsf{mod}_{\mathsf{Lie}},\circ,\mathbb{1}_0)$. We refer to \cite{chingcomposition} for more details.
\end{remark}

From now on, we restrict to operads $\mathcal{P}$ which are \textit{pointed reduced}, meaning $\mathcal{P}(1) = \mathcal{P}(0) = \mathbb{1}$. Let $\mathcal{U}nit$ be the symmetric sequence given by $\mathcal{U}nit(0) = \mathcal{U}nit(1) = \mathbb{1}$ and $\emptyset$ elsewhere. It has a unique operad structure which is given by the obvious compositions maps. An operad $\mathcal{P}$ is pointed reduced if and only if it admits a morphism of operads $f: \mathcal{U}nit \longrightarrow \mathcal{P}$ 
which is an isomorphism in arities zero and one.

\begin{example}\leavevmode

\begin{enumerate}
\item The collection of the pure braid groups $(\mathrm{PB}_n)_{n\geq0}$ forms a non symmetric operad in the category of $(\mathsf{Grp},\times,\{*\})$, 
where the structure is given by the insertion of braids. Notice that the insertion of the empty braid in $\mathrm{PB}_0 \cong \{*\}$ suppresses the strand into which 
it is inserted. One can check that the natural action of $\mathfrak{S}_n$ is \textit{not} compatible with the insertion of braids.

\medskip

\item The \textit{little disks operad} $\mathbb{E}_2$ forms an operad in the category of topological spaces $(\mathsf{Top},\times,\{*\})$ (see \cite{May72}). 
It does not form an operad in the category of \textit{pointed} topological spaces 
$(\mathsf{Top}_*,\times,\{*\})$. Indeed, if it did, since we know that for all $n$ there is a weak equivalence $\mathbb{E}_2(n) \qi \mathrm{Conf}_n(\mathbb{C})$, 
then one could apply the strong monoidal functor $\pi_1(-)$ and obtain a symmetric operad structure on the collection $(\mathrm{PB}_n)_{n\geq0}$. 
\hfill$\triangle$
\end{enumerate}
\end{example}

One can "try to symmetrize" the non-symmetric operad of pure braid groups $(\mathrm{PB}_n)_{n\geq0}$ by passing to the larger category of groupoids and 
defining the following operad in it.

\begin{example}\label{example: cob}
Let $\cob$ be the operad in groupoids given by the following description:

\medskip

\begin{enumerate}
\item The objects of $\cob(n)$ are elements $\sigma$ of the symmetric group $\mathbb{S}_n$. For example, $(321)$ is in $\mathrm{Ob}(\cob(3))$. 

\medskip

\item The set of morphisms between two permutations $\sigma$ and $\tau$ in $\mathbb{S}_n$ is given by the set of braids going from $\sigma$ to $\tau$. Pictorially, we have: 

\begin{center}
\includegraphics[width=55mm,scale=1]{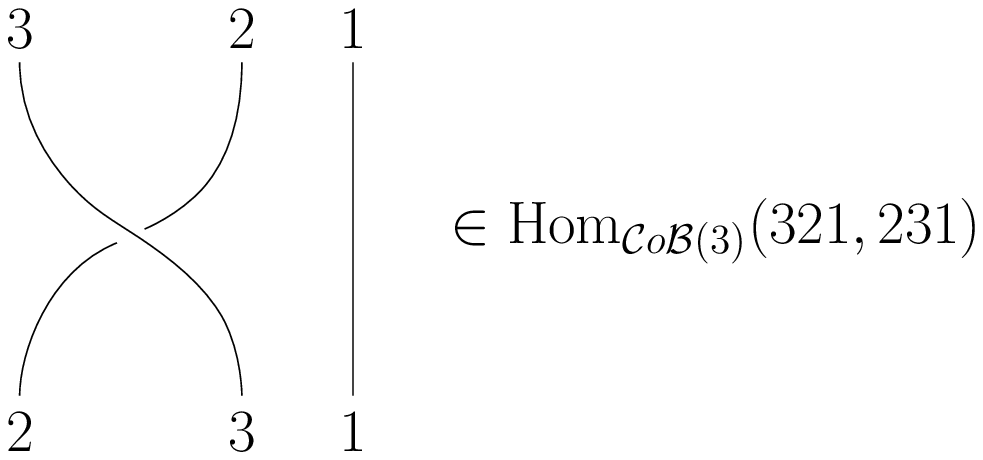}
\end{center}

\item The operadic composition is given by the insertion of braids together with the relabeling of strands according to the insertion spot. Pictorially, we have: 
\medskip

\hspace{3.5pc}
\includegraphics[width=60mm,scale=1]{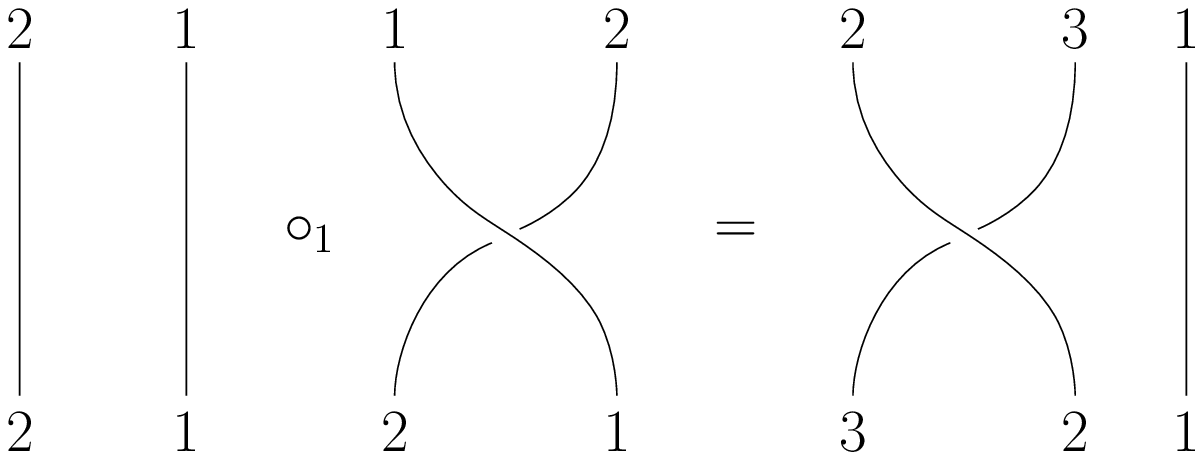}~.
\hfill$\triangle$
\end{enumerate}
\end{example}

\begin{remark}
Since the operad $\mathbb{E}_2$ is not pointed, one could take the \textit{fundamental groupoid} $\Pi_1(\mathbb{E}_2)$ instead of the fundamental group as before. Since the functor $\Pi_1(-)$ is strong monoidal, 
this defines an operad in the category of groupoids. Nevertheless, the operad of objects of $\Pi_1(\mathbb{E}_2)$ is huge. 

\medskip

One can ask oneself if the operad $\cob$ defined above is \textit{a model} for $\Pi_1(\mathbb{E}_2)$, that is, if these two operads are equivalent as operads in the category of groupoids. It is indeed a model for $\Pi_1(\mathbb{E}_2)$. 
Nevertheless, it is not a \textit{cofibrant} operad in the category of groupoids, as the operad of objects of $\cob$ is not free. 
\end{remark}

\subsection{The operad of parenthesized braids}\label{subsection: pab}

\begin{definition}[Fake pull-back]\label{def: fake-pullback}
Let $\mathcal{O},\mathcal{P}$ be two operads in $(\mathsf{Grpd},\times,\{*\})$ and suppose there is a morphism 
$f: \mathrm{Ob}(\mathcal{O}) \longrightarrow \mathrm{Ob}(\mathcal{P})$ of operads in the category $(\mathsf{Sets},\times,\{*\})$. 
The \textit{fake pull-back} of $\mathcal{P}$ along $f$, denoted $f^*\mathcal{P}$, is the operad in the category of groupoids defined as follows:

\begin{enumerate}
\medskip

\item The objects of $f^*\mathcal{P}(n)$ are given by $\mathrm{Ob}(f^*\mathcal{P}(n)) \coloneqq \mathrm{Ob}(\mathcal{O}(n))~.$ 

\medskip

\item The morphisms between $a,b$ in $\mathrm{Ob}(f^*\mathcal{P}(n))$ are given by 
\[
\mathrm{Hom}_{f^*\mathcal{P}(n)}(a,b) \coloneqq \mathrm{Hom}_{\mathcal{P}(n)}(f(a),f(b))~.
\]
\end{enumerate}
\end{definition}

\begin{remark}
It is straightforward to check that $f^*\mathcal{P}$ forms an operad in the category of groupoids endowed with the composition of the operad $\mathcal{O}$ 
on objects and the composition of the operad $\mathcal{P}$ on morphisms. 
\end{remark}

Let $\mathcal{P}a$ be the free operad in sets generated by an arity two operation. The set $\mathcal{P}a(n)$ is the set of maximally parenthesized permutations of 
$\mathfrak{S}_n$. One can view this operad in sets inside the category of operads in groupoids by declaring that all sets of morphisms are trivial (the empty set 
between two different objects, the identity as the unique endomorphism). There is an obvious morphism of operads in sets 
$\varphi: \mathrm{Ob}(\mathcal{P}a) \longrightarrow \mathrm{Ob}(\cob)$ given by forgetting the parenthesis on permutations. 

\begin{definition}[Operad of parenthesized braids]
The operad in the category of groupoids of parenthesized braids $\pab$ is defined to be the fake pull-back of $\cob$ along the morphism $\varphi$; 
$\pab:=\varphi^*\cob$. 
\end{definition}

\begin{example}
Objects in $\pab(n)$ are fully parenthesized permutations $\sigma$ in $\mathbb{S}_n$. Morphisms are again given by braids going from one permutation to another. 
Composition is given by the insertion of braids together with the composition of parenthesized permutations. Pictorially, we have for instance

\begin{center}
\includegraphics[width=55mm,scale=1]{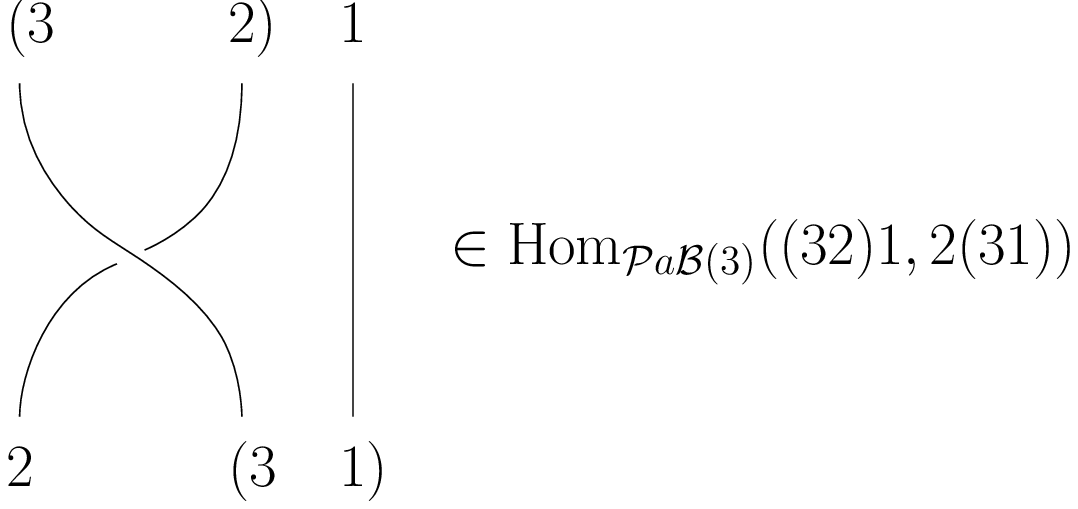}~.
\end{center}

The operadic composition is given by the insertion of braids together with the insertion of parenthesis

\medskip

\hspace{7pc}
\includegraphics[width=60mm,scale=1]{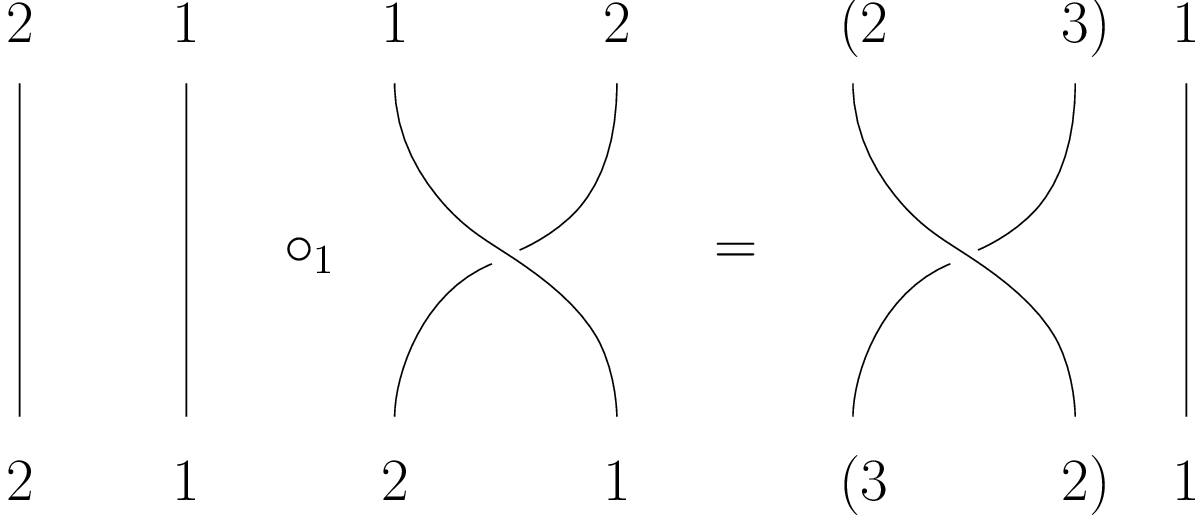}~.\hfill$\triangle$
\end{example}

\begin{remark}
The operad of objects of $\pab$ is free. It provides us with a cofibrant model for $\Pi_1(\mathbb{E}_2)$ in the category of operads in groupoids. 
\end{remark}

Since the Malcev completion $\widehat{(-)}(\kk)$ is a strong monoidal functor, we can complete the operad $\pab$ and still get an operad. See the Appendix \ref{Completion pro-unipotente des groupoides} for more details on the Malcev completions and pro-unipotent $\kk$-groupoids. 

\medskip

The resulting operad $\widehat{\pab}(\kk)$ in the category of pro-unipotent $\kk$-groupoids will be the \textit{first player} in our story. 

\subsection{The operad of chord diagrams}\label{ssec-1.4}

Recall the Drinfeld--Kohno family of Lie algebras $(\mathfrak{t}_n)_{n \geq 0}$ from Definition \ref{Kohno-Drinfeld algebras}. The insertion-coproduct maps will 
endow this family of Lie algebras with an operad structure. This operad can be understood as the holonomy Lie operad (see \cite[\S5.3]{ManinVallette19}) of the little disks operad $\mathbb{E}_2$. 

\begin{proposition}\label{prop: operad structure on Kohno-Drinfeld}
The family of Lie algebras $(\mathfrak{t}_n)_{n\geq0}$ can be endowed with the following operad structure in the symmetric 
monoidal category $(\mathsf{Lie},\oplus,0)$: 

\begin{enumerate}
\item The action of $\sigma$ in $\mathbb{S}_n$ on $\mathfrak{t}_n$ is given by $\sigma \bullet t_{ij}^n \coloneqq t^n_{\sigma(i)\sigma(j)}~.$

\item The partial composition maps 
\[
\{\circ_p: \mathfrak{t}_n \oplus \mathfrak{t}_m \longrightarrow \mathfrak{t}_{n+m-1}\}_{1 \leq p \leq n}
\]
only needs to be specified on each component of the direct sum. For $i<j$, they are given as 
\begin{eqnarray*}
\mathfrak{t}_n \ni t_{ij} & \longmapsto & 
\begin{cases*}
t_{i+m-1\hspace{1pt}j+m-1} & if~~p < i~.\\
{\displaystyle \sum_{k=i}^{i+m-1}t_{k\hspace{1pt}j+m-1}} & if~~p = i~.\\
t_{i\hspace{1pt}j+m-1} & if~~ i < p < j~.\\
{\displaystyle \sum_{k=j}^{j+m-1}t_{i\hspace{1pt}k}} &  if~~ p = j~.\\
t_{i\hspace{1pt}j} & if~~ j < p~.
\end{cases*}\\[0.35cm]
\mathfrak{t}_m \ni t_{ij} & \longmapsto & t_{i+p-1 \hspace{1pt} j+p-1}
\end{eqnarray*}

%\[
%\begin{tikzcd}[column sep=1.5pc,row sep=0pc]
%t_{ij} ~~\text{in}~~ \mathfrak{t}_n\arrow[r,mapsto]
%& \left\{
%    \arraycolsep=1.4pt\def\arraystretch{2.2}\begin{array}{lllll}
%        t_{i+m-1\hspace{1pt}j+m-1} ~~\mathrm{if} ~~p<i,j~.\\
%        t_{i\hspace{1pt}j+m-1}+t_{i+1\hspace{1pt}j+m-1}+\cdots+ t_{i+m-1\hspace{1pt}j+m-1}~~\mathrm{if}~~ p = i~.\\
%        t_{i\hspace{1pt}j+m-1}~~\mathrm{if}~~ i < p < j~.\\
%        t_{i\hspace{1pt}j}+t_{i\hspace{1pt}j+1}+\cdots+t_{i\hspace{1pt}j+m-1}~~\mathrm{if}~~ p = j~.\\
%        t_{i\hspace{1pt}j}~~\mathrm{if}~~ p> i,j~.
%    \end{array}
%\right. \\
%t_{ij}  ~~\text{in}~~ \mathfrak{t}_m \arrow[r,mapsto]
%& t_{i+p-1 \hspace{1pt} j+p-1}~.
%\end{tikzcd}
%\]
\end{enumerate}
\end{proposition}

\begin{proof}
Let us construct the restriction of $\circ_p$ to each of the sources components using the insertion-coproduct morphisms of Definition 
\ref{def: insertion-coproduct morphisms}. There is a ``coproduct type'' Lie algebra morphism 
\[
\circ_p(1): \mathfrak{t}_n \longrightarrow \mathfrak{t}_{n+m-1}
\]
associated to the well-defined map $\{1,...,n+m-1\} \to \{1,...,n\}$ which sends $\{p,...,p+m-1\}$ to $p$, and there is a ``insertion type'' Lie algebra morphism 
\[
\circ_p(2): \mathfrak{t}_m \longrightarrow \mathfrak{t}_{n+m-1}
\]
associated to the partially defined map $\{1,...,n+m-1\} \to \{1,...m\}$ which sends $k\in\{p,...,p+m-1\}$ to $k-p+1$. 
This induces by universal property a Lie algebra morphism
\[
\circ_p(1) \amalg \circ_p(2): \mathfrak{t}_n \amalg \mathfrak{t}_m \longrightarrow \mathfrak{t}_{n+m-1}~.
\]
This induces an operad structure on $(\mathfrak{t}_n)_{n\geq0}$ for the coproduct of Lie algebras, see Appendix \ref{Appendix B1: operads}. One can notice that the images of these two maps commute, hence this map descends to the direct sum and gives the partial composition map $\circ_p$ defined in the proposition, which satisfy the axioms of an operad.
\end{proof}

\begin{remark}
Notice that in the above proposition, we have written the \textit{left} $\mathbb{S}_n$-action on $\mathfrak{t}_n$. The right $\mathbb{S}_n$-action is given by 
\[
t_{ij}^n \bullet \sigma \coloneqq t^n_{\sigma^{-1}(i)\sigma^{-1}(j)}~, 
\]
and in general one can pass from a left to a right action by applying the involution $(-)^{-1}$. We have chosen to do so since it simplifies the diagrams in Theorem \ref{thm: presentation pab}. The same will apply for Proposition \ref{cyclotomic moperad structure on Kohno-Drinfeld} and Theorem \ref{thm: presentation moperad pab Gamma}; and for Proposition \ref{prop: right mod structure} and Theorem \ref{thm: presentation pab elliptic}.
\end{remark}

\begin{remark}
The above proof is a particular instance of a general phenomenon that we describe in \cref{appendixB}. 
\end{remark}

Since the universal enveloping algebra functor $\mathfrak{U}(-)$ is a strong monoidal functor, it sends operads in $(\mathsf{Lie},\oplus,0)$ to operads in 
$(\mathsf{Hopf}\text{-}\mathrm{alg},\otimes,\kk)$.

\begin{definition}[Chord diagrams operad]
The operad of \textit{chord diagrams} $\mathcal{CD}$ is the operad given by applying the universal enveloping algebra functor to the Drinfeld--Kohno operad.
\end{definition}

Unwrapping this definition, the space of arity $n$ operations of $\mathcal{CD}$ is given by the universal enveloping algebra of $\mathfrak{t}_n$, the $n$-th 
Drinfeld--Kohno algebra. Since the latter admits an explicit presentation, the former also admits the following presentation
\[
\mathfrak{U}(\mathfrak{t}_n) =  \frac{\kk\langle t_{ij}\,|\,1\leq i,j \leq n,i \neq j\rangle}{\left( t_{ij} = t_{ji}~,~[t_{ij},t_{ik}+t_{kj}] = 0~,~ [t_{ij},t_{kl}] = 0\right)}~.
%\frac{u\mathcal{A}ss\{t_{ij}\,|\,1\leq i,j \leq n,i \neq j\}}{\left( t_{ij} = t_{ji}~,~[t_{ij},t_{ik}+t_{kj}] = 0~,~ [t_{ij},t_{kl}] = 0\right)}~.
\]

The algebras $\big(\mathfrak{U}(\mathfrak{t}_n)\big)_{n\geq0}$ are called the \textit{chord diagram algebras} since pictorially one can represent the generators 
$t_{ij}$ as chords (or infinitesimal braids) in the following way

\[
\includegraphics[width=50mm,scale=1]{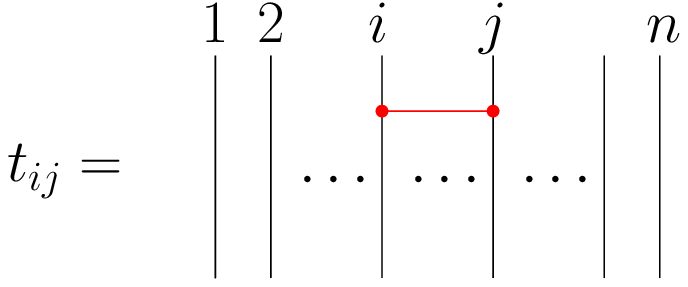}~.
\]

This way, one can represent the relations between the generators as the following relations between chord diagrams.

\begin{enumerate}
\item The first relation can be depicted as
\[
\includegraphics[width=65mm,scale=1]{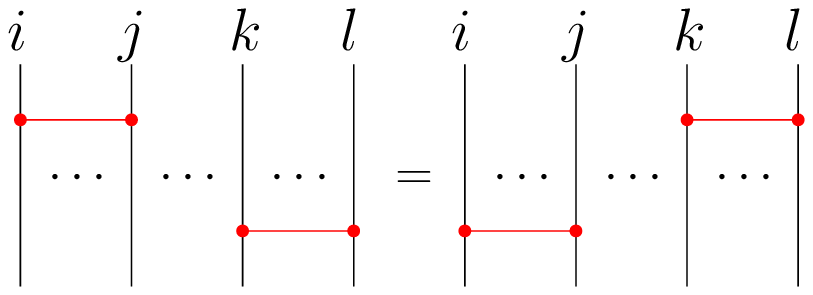}~.
\]

\item The second relation can be depicted as

\[
\includegraphics[width=115mm,scale=1]{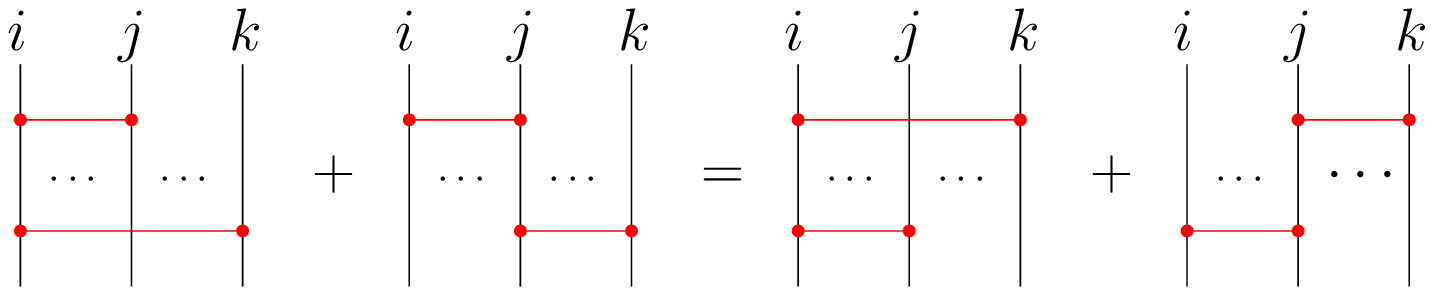}~.
\]

\end{enumerate}

\begin{remark}[Weight-filtration]\label{rmk: weight filtration}
The relations that define the Lie algebras $\mathfrak{t}_n$ are homogeneous for all $n \geq 2$. Hence the algebras $\mathfrak{U}(\mathfrak{t}_n)$ are endowed with a weight filtration where the generators $t_{ij}$ are in weight one.
\end{remark}

Let $\mathsf{Cat}(\mathsf{Coalg}_\kk)$ denote the category of small categories enriched over cocommutative $\kk$-coalgebras. It comes equipped with a symmetric monoidal structure given as follows: let $\mathcal{E}$ and $\mathcal{F}$ be two such categories, their product $\mathcal{E} \otimes \mathcal{F}$ is the category given by the following presentation
\begin{enumerate}
\medskip

\item The objects of $\mathcal{E} \otimes \mathcal{F}$ are given by pairs in $\mathrm{Ob}(\mathcal{E})\times \mathrm{Ob}(\mathcal{F})~.$

\medskip

\item The set of morphisms between two pairs $(e,f)$ and $(e',f')$ is
\[
\mathrm{Hom}_{\mathcal{E} \otimes \mathcal{F}}((e,f),(e',f')) \coloneqq \mathrm{Hom}_{\mathcal{E}}(e,e') \otimes \mathrm{Hom}_{\mathcal{F}}(f,f')~,
\]
where the last tensor denotes simply the tensor product of two cocommutative $\kk$-coalgebras. 
\end{enumerate}

One checks that it endows $\mathsf{Cat}(\mathsf{Coalg}_\kk)$ with a symmetric monoidal structure. The unit is given by the category with one object whose 
endomorphisms are given by $\kk$, which we denote by $\{*\}$. Thus one can define operads in the symmetric monoidal category 
$(\mathsf{Cat}(\mathsf{Coalg}_\kk),\otimes,\{*\})~.$

\begin{example}
The chord diagram operad $\mathcal{CD}$ is an operad in $\mathsf{Cat}(\mathsf{Coalg}_\kk)$ where the category $\mathcal{CD}(n)$ only has one object whose 
algebra of endomorphisms is given by $\mathfrak{U}(\mathfrak{t}_n)$. \hfill$\triangle$
\end{example}

\begin{remark}
Let $\mathcal{E}$ be in $\mathsf{Cat}(\mathsf{Coalg}_\kk)$. Notice that for any object $e$ in $\mathcal{E}$, its endomorphisms $\mathrm{Hom}_{\mathcal{E}}(e,e)$ 
form a cocommutative bialgebra, where the multiplication is given by the composition of endomorphisms.
\end{remark}

\begin{definition}[$\kk$-linear extension]\label{def: k-extension}
The $\kk$-\textit{extension} is a strong monoidal functor 
\[
\begin{tikzcd}[column sep=1.5pc,row sep=0pc]
(\mathsf{Set},\times,\{*\}) \arrow[r]
&(\mathsf{Cat}(\mathsf{Coalg}_\kk),\otimes,\{*\}) \\
S \arrow[r,mapsto] 
&S_\kk~.
\end{tikzcd}
\]
It associates to a set $S$ the category $S_\kk$ which is defined as follows:
\begin{enumerate}
\medskip
\item The set of objects of $S_\kk$ is given by $S$.
\medskip
\item Let $s,s'$ be in $S$, the space of morphisms is given by 
\[
\mathrm{Hom}_{S_\kk}(s,s') \coloneqq \kk~.
\]
\end{enumerate}
\end{definition}

The set-theoretical operad $\mathcal{P}a$ can therefore be promoted into an operad in $\mathsf{Cat}(\mathsf{Coalg}_\kk)$ by considering its $\kk$-linear extension 
$\mathcal{P}a_\kk$.

\begin{definition}[Parenthesized chord diagrams]
The \textit{parenthesized chord diagrams} operad $\pacd$ is given by Hadamard product of $\mathcal{P}a_\kk$ and $\mathcal{CD}$ in the category of operads defined 
above $\mathsf{Cat}(\mathsf{Coalg}_\kk)$. That is, we have:
\[
\pacd(n) \coloneqq \mathcal{P}a_\kk(n) \otimes \mathcal{CD}(n)~,
\]
where $\otimes$ denotes the tensor product in $\mathsf{Cat}(\mathsf{Coalg}_\kk)$.
\end{definition} 

\begin{example}\label{example: elements X H A}
The object of $\pacd(n)$ are therefore given by parenthesized permutations in $\mathfrak{S}_n$. The morphisms in $\mathrm{Hom}_{\pacd(n)}(\sigma,\tau)$ between 
two parenthesized permutations can be depicted as chord diagrams on the unique strand that goes from $\sigma$ to $\tau$. For instance, the element 
$1 \otimes (t_{24}.t_{34})$ in $\kk \otimes \mathfrak{U}(\mathfrak{t}_4) \cong \mathrm{Hom}_{\pacd(4)}\left(((31)2)4,(4(13))2\right)$ can be depicted as:

\[
\includegraphics[width=80mm,scale=1]{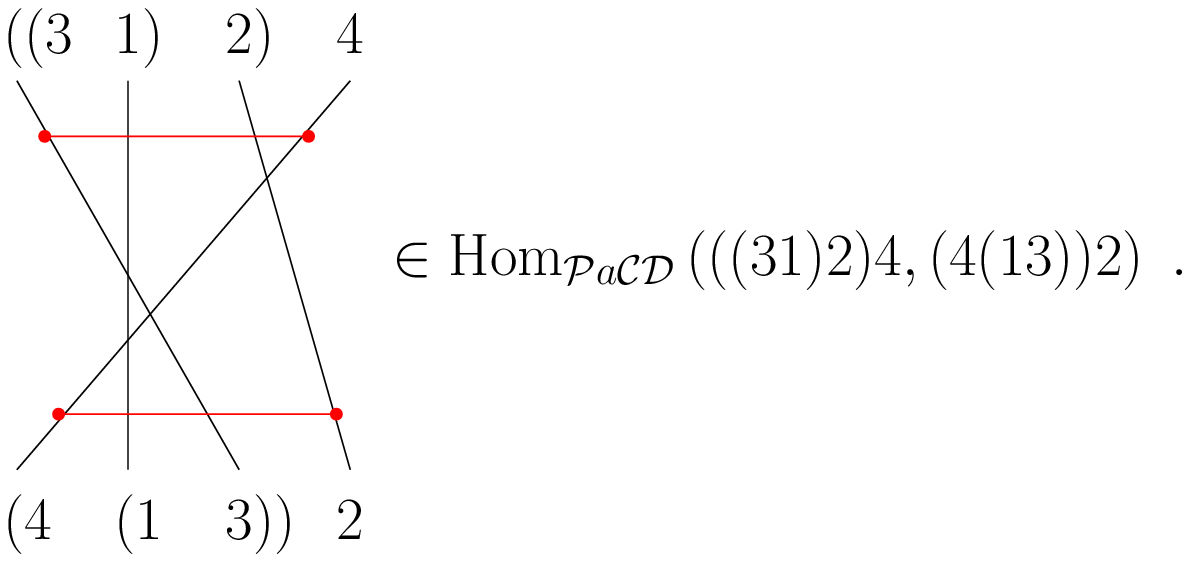}
\]

Composition operations in the operad 

\[
\circ_i: \pacd(n) \otimes \pacd(m) \longrightarrow \pacd(n+m-1)~
\]

are defined as follows. On objects, the composition of two parenthesized permutations is given by the insertion the $i$-th spot. On morphisms, that is, on chord diagrams, 
its given by the sum of possible distributions of the existing chords along the inserted strands. For instance, the composition of 
$1 \otimes t_{12} \circ_2 1 \otimes \mathrm{id}$ gives 

\[
\includegraphics[width=80mm,scale=1]{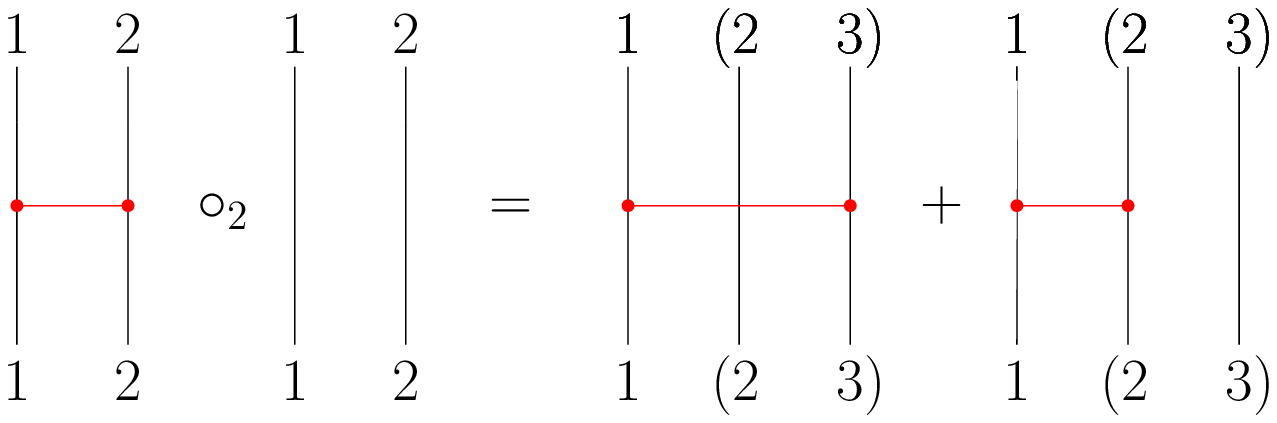}
\]

which is an endomorphism of $1(23)$ in $\pacd(3)$. Here are some examples of distinguished morphisms in the operad $\pacd$: 

\begin{enumerate}
\item The element $H \coloneqq 1 \otimes t_{12}$ in $\mathrm{Hom}_{\pacd(2)}(12,12)$ depicted as

\[
\includegraphics[width=20mm,scale=1]{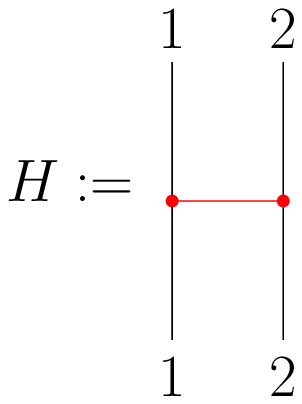}
\]

\item The element $X \coloneqq 1 \otimes 1$ in $\mathrm{Hom}_{\pacd(2)}(12,21)$ depicted as

\[
\includegraphics[width=20mm,scale=1]{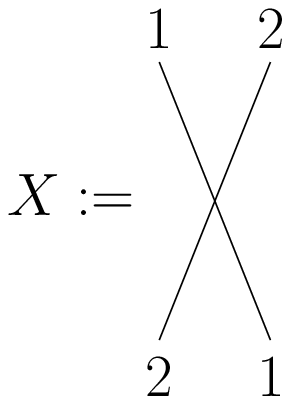}
\]

\item The element $\alpha \coloneqq 1 \otimes 1$ in $\mathrm{Hom}_{\pacd(3)}((12)3,1(23))$ depicted as

\[
\includegraphics[width=30mm,scale=1]{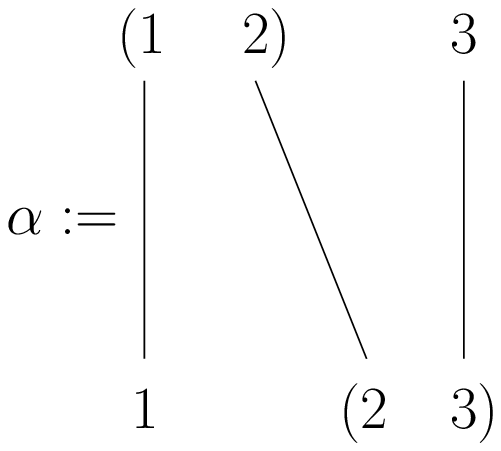}
\]
\hfill$\triangle$
\end{enumerate}
\end{example}

\begin{remark}\label{rmk: presentation of PaCD}
The operad $\pacd$ is in fact generated, as an operad in the category $\mathsf{Cat}(\mathsf{Coalg}_\kk)$, by the object $12$ in $\pacd(2)$ and by the morphisms 
$\alpha$, $X$ and $H$. They satisfy some relations but, to the best of our knowledge, there is no explicit presentation for $\pacd$. See \cite[Section 10.2]{FresseGT1} 
for more details.
\end{remark}

We define $\widehat{\pacd}$ as the operad in categories enriched in complete cocommutative coalgebras obtained by completing $\pacd$ with respect to the filtration 
defined in \cref{rmk: weight filtration}. Indeed, the weigh filtration on each algebra $\mathfrak{U}(\mathfrak{t}_n)$ induces a weight filtration on each 
hom-coalgebra $\mathrm{Hom}_{\pacd(n)}(\sigma,\tau)$. This operad forms an operad in complete Hopf groupoids. See Appendix \ref{Completion pro-unipotente des groupoides} for more details on these notions.

\medskip

The operad  $\mathrm{Grp}(\widehat{\pacd})$, obtained by applying the group-like element functor to each complete Hopf groupoid, forms an operad in the category of 
pro-unipotent $\kk$-groupoids. This will be our \textit{second player}.

\subsection{Operadic definition of associators}\label{ssec-1.5}

Now we may give a fully operadic definition of an associator.

\begin{definition}[Operadic associators] \label{Def: operadic associator}
The set of \textit{operadic associators} is given by 
\[
\mathrm{OpAssoc}(\kk) \coloneqq \mathrm{Iso}^+_{\mathsf{Op}(\mathsf{p.u}\text{-}\mathsf{Grpd}_{\kk})}\left(\widehat{\pab}(\kk),\mathrm{Grp}(\widehat{\pacd})\right)~,
\]
where we consider only isomorphisms of operads in the category of groupoids which are given by the identity morphism on the objects of these operads.
\end{definition}

This definition entails that the set of operadic associators is a bitorsor over the automorphisms groups of each operad that are the identity morphism on objects. These 
automorphisms groups are of primordial importance.

\begin{definition}[Grothendieck--Teichmüller group]
The \textit{Grothendieck--Teichmüller group} over $\kk$ is given by 
\[
\mathrm{GT}(\kk) \coloneqq \mathrm{Aut}_{\mathsf{Op}(\mathsf{p.u}\text{-}\mathsf{Grpd}_{\kk})}^+\left(\widehat{\pab}(\kk)\right)~,
\]
where we consider only automorphisms of operads in the category of groupoids which are given by the identity morphism on the objects.
\end{definition}

\begin{remark}[Fifty shades of Grothendieck--Teichmüller]
One can define many other versions of the Grothendieck--Teichmüller group. There is a \textit{pro-finite} version, denoted $\widehat{\mathrm{GT}}$, 
where one considers the automorphisms of operads that fix the objects of the pro-finite completion of $\pab$. It was shown in \cite{Ihara} that the absolute 
Galois group $\mathrm{Gal}(\overline{\mathbb{Q}}/\mathbb{Q})$ embeds into $\widehat{\mathrm{GT}}$. These two groups are conjecturally isomorphic. 
For more on this, see \cite{Marcy}. In the same spirit, one can defined a \textit{pro-$\ell$} version of the Grothendieck--Teichmüller group. 
%There is also a \textit{modulo $p$} version given by considering the pro-unipotent completion of $\pab$ with respect to a field of characteristic $p>0$.
Notice that if one does not complete in any way $\pab$, its automorphisms group is simply $\mathbb{Z}/2\mathbb{Z}$. This corresponds heuristically to the 
complex conjugation in $\mathrm{Gal}(\overline{\mathbb{Q}}/\mathbb{Q})$. 
\end{remark}

\begin{definition}[Graded Grothendieck--Teichmüller group]
The \textit{graded Grothendieck--Teichmüller group} over $\kk$ is given by 
\[
\mathrm{GRT}(\kk) \coloneqq \mathrm{Aut}_{\mathsf{Op}(\mathsf{p.u}\text{-}\mathsf{Grpd}_{\kk})}^+\left(\mathrm{Grp}(\widehat{\pacd})\right) \cong \mathrm{Aut}_{\mathsf{OpCat}(\mathsf{Coalg}_\kk)}^+\left(\widehat{\pacd}\right)~,
\]
where we consider only automorphisms of operads in the category of groupoids (or equivalently, of operad in the category of categories enriched in 
counital cocommutative coalgebras) which are given by the identity morphism on the objects.
\end{definition}

\begin{remark}
In the three above definitions, one could replace isomorphisms that are the identity on objects by isomorphisms in the homotopy category of a suitable 
model category of operads in pro-unipotent groupoids. It appears that this gives the same result, as is very well-explained (for the pro-finite case) in \cite{Horel}. 
It roughly relies on the fact that any morphism between these operads is homotopic to one being the identity on objects. 
\end{remark}

\subsection{More concrete descriptions}\label{ssec-1.6}

\begin{theorem}\label{thm: presentation pab}
The operad $\pab$, as an operad in groupoids admits the following presentation. It is generated by the object $12$ in $\pab(2)$ and by the following morphisms:
\begin{enumerate}
\item The \textit{braiding} $R$ which pictorially is given by the following braid

\[
\includegraphics[width=50mm,scale=1]{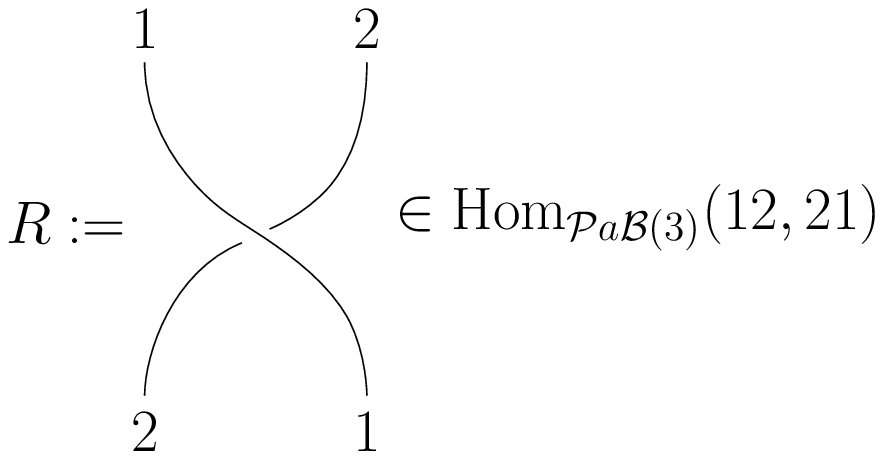}
\]

\item The \textit{associator} $\Phi$ which pictorially is given by the following braid

\[
\includegraphics[width=70mm,scale=1]{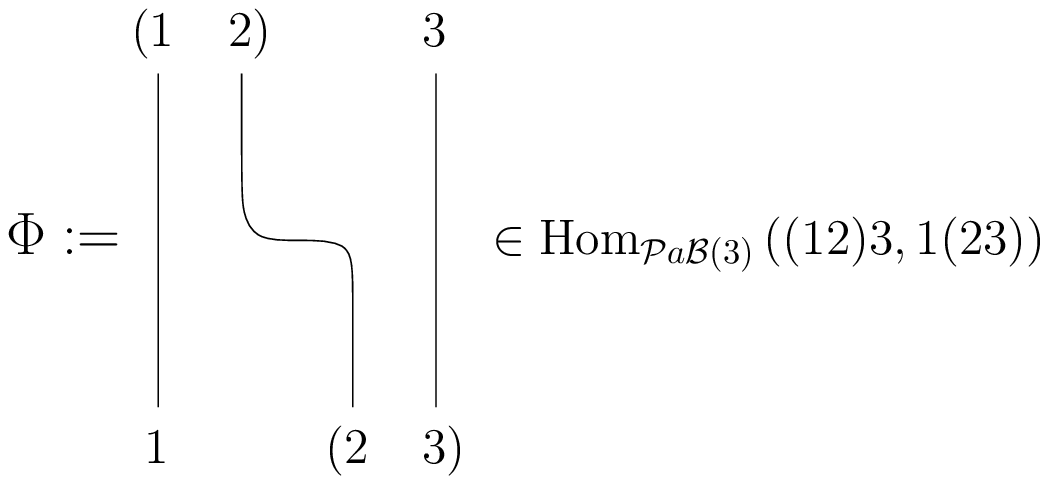}
\]

\end{enumerate}

They satisfy the following relations

\begin{enumerate}
\item Recall that $\pab(0)$ is the trivial groupoid $\{*\}$. The \textit{unital relation} states that
\[
\Phi \circ_2 \mathrm{id}_{\{*\}} = \mathrm{id}_{12}~. 
\]

\item The \textit{pentagon relation} amounts to the commutativity of the following diagram
\[
\begin{tikzcd}[column sep=2.5pc,row sep=2pc]
&((12)3)4   \arrow[dl,"\mathrm{id}_{12} ~ \circ_1 ~\Phi",swap] \arrow[dr,"\Phi ~\circ_1 ~\mathrm{id}_{12}"]
& \\
(1(23))4 \arrow[d,"\Phi ~\circ_2 ~\mathrm{id}_{12}",swap] 
&
&(12)(34) \arrow[d,"\Phi ~\circ_3 ~\mathrm{id}_{12}"] \\
1((23)4)  \arrow[rr,"\mathrm{id}_{12} ~\circ_2 ~\Phi"]
&
&1(2(34))
\end{tikzcd}
\]

\item The \textit{first hexagon relation} amounts to the commutativity of the following diagram
\[
\begin{tikzcd}[column sep=5pc,row sep=3pc]
&(12)3 \arrow[r,"\Phi"] \arrow[ld,"\mathrm{id}_{12} ~\circ_1 ~ R",swap]
&1(23) \arrow[rd,"\mathrm{id}_{12} ~\circ_2 ~R"] 
& \\
(21)3 \arrow[rd,"(213) ~\bullet ~ \Phi" swap] 
&
&
&1(32) \arrow[dl,"(132)~ \bullet ~\Phi^{-1}"]\\
&2(13) \arrow[r,"(132) ~\bullet ~ (R ~ \circ_2 ~ \mathrm{id}_{12})"]
&(13)2
&
\end{tikzcd}
\]

\item The \textit{second hexagon relation} amounts to the commutativity of the above diagram when replacing $R$ by $\tilde{R} \coloneqq ((21) \bullet R)^{-1}~.$
\end{enumerate} 
\end{theorem}
The proof can be found in \cite[Theorem 6.2.4]{FresseGT1} (although originally proven in \cite{BarNatan99} using a different language). 
\begin{remark}
Notice that the unitality relation, together with the pentagon relation, implies the two other unitality relations 
$\Phi \circ_1 \mathrm{id}_{\{*\}} = \Phi \circ_3 \mathrm{id}_{\{*\}} = \mathrm{id}_{12}$ as well. 
Indeed, if one applies twice $-\circ_1\mathrm{id}_{\{*\}}$ to the pentagon relation, one finds $(\Phi \circ_1 \mathrm{id}_{\{*\}})^2=\Phi \circ_1 \mathrm{id}_{\{*\}}$; 
this implies that $\Phi \circ_1 \mathrm{id}_{\{*\}}=\mathrm{id}_{12}$. One proves along the same lines that $\Phi \circ_3 \mathrm{id}_{\{*\}} = \mathrm{id}_{12}$. 
\end{remark}
\begin{corollary}
There is a one to one correspondence between classical associators as in \cref{def: classical associator} and operadic associators as in \cref{Def: operadic associator}.
\end{corollary}

\begin{proof}[Sketch of proof]
Let $F: \widehat{\pab}(\kk) \qi \mathrm{Grp}(\widehat{\pacd})$ be an isomorphism which is the identity on objects. By the above theorem, it is completely determined by the image of the generators $R$ and $\Phi$. 

\medskip

The image of $R$ by $F$ has to be of the form
\[
F(R) = e^{\lambda.t_{12}/2}X~, \quad \mathrm{in} \quad \mathrm{Hom}_{\pacd(2)}(12,21) ~,
\]
where $e^{\lambda.t_{12}/2}$ is in $\mathrm{exp}(\widehat{\mathfrak{t}}_2) \cong \kk$ and where $X$ is the element of $\pacd$ given in \cref{example: elements X H A}. Furthermore, $F$ is an isomorphism if and only if $\lambda$ is in $\kk^\times$. 

\medskip

The image of $\Phi$ by $F$ has to be of the form
\[
F(\Phi) = f(\Phi)\alpha~, \quad \mathrm{in} \quad \mathrm{Hom}_{\pacd(3)}((12)3,1(23))~,
\]
where $f(\Phi)$ is in $\mathrm{exp}(\widehat{\mathfrak{t}}_3)$ and where $\alpha$ is the element of $\widehat{\pacd}$ given in \cref{example: elements X H A}. Recall from Subsection \ref{subsection: drinfeld associator} that there is a canonical isomorphism of Lie algebras
\[
\mathfrak{t}_3 \cong \kk c \oplus \mathfrak{f}_2 \,,
\]
where $c = t_{12} + t_{23} + t_{13}$ and where $\mathfrak{f}_2$ is the free Lie algebra generated by $t_{12}$ and $t_{23}$.

\medskip

This gives a canonical isomorphism $\mathrm{exp}(\widehat{\mathfrak{t}}_3) \cong \kk \times \widehat{\mathbb{F}}_2(\kk)$. Therefore we can decompose $f(\Phi)$ as 
\[
f(\Phi) = (\Xi(c),\Phi(t_{12},t_{23})) \quad \mathrm{in} \quad \mathrm{exp}(\widehat{\mathfrak{t}}_3) \cong \kk \times \widehat{\mathbb{F}}_2(\kk)~.
\]
Now one can check that the morphism
\[
- \circ_2 \mathrm{id}_{\{*\}}: \mathrm{Hom}_{\pacd(3)}((12)3,1(23)) \longrightarrow \mathrm{Hom}_{\pacd(2)}(12,12)
\]
is the one determined by the morphism $\mathfrak{t}_3 \twoheadrightarrow \mathfrak{t}_2$ which sends $t_{13}$ to $t_{12}$ and $t_{12},t_{23}$ to zero. By pre-composing it with the isomorphism $\kk c \oplus \mathfrak{f}_2 \cong \mathfrak{t}_3$, the resulting morphism sends $t_{12},t_{23}$ to zero and $c$ to $t_{12}$. The unital relation 
\[
\Phi \circ_2 \mathrm{id}_{\{*\}} = \mathrm{id}_{12}
\]

therefore gives that $\Xi(c) = 1$, and thus that $f(\Phi)$ is given by $\Phi(t_{12},t_{23})$ in $\mathrm{exp}(\widehat{\mathfrak{f}}_2)$.

\medskip

So far, we have identified the data present in \cref{def: classical associator} and in \cref{Def: operadic associator}. Let us briefly illustrate how the 
relations in both definitions are sent to each other for the particular case of the pentagon equation. Notice that there is a dictionary between the notation used in the 
introduction and the operadic notation. The morphism  $\Phi ~\circ_2 ~\mathrm{id}_{12}$ rewrites as $\Phi^{1,23,4}$ since it is the associator where $2$ and $3$ are 
considered as "a single object". The morphism $\mathrm{id}_{12} ~\circ_2 ~\Phi$ rewrites as $\Phi^{2,3,4}$ since the object $1$ "plays no role". The pentagon relation 
imposed on $\Phi$ then rewrites in this particular notation as
\[
\Phi^{1,2,3}\Phi^{1,23,4} \Phi^{2,3,4} = \Phi^{12,3,4}\Phi^{1,2,34}~,
\]
which is the same equation as in the introduction, but where the products are in the opposite order. 
More generally, one can check (see \cite[Theorem 10.2.9]{FresseGT1} for a detailled proof) that this dictionary identifies the relations imposed on Drinfeld associators 
with the relations imposed on the operadic associators (with reversed order). 
The one to one correspondence between classical associators as in \cref{def: classical associator} and operadic associators as in \cref{Def: operadic associator} 
is therefore given by $\Phi\mapsto \Phi^{-1}$. 
\end{proof}

\begin{remark}[Algebraic and topological conventions]\label{rmk: algebraic vs topological conventions}
As observed in the above proof, the pentagon equation for a Drinfeld associator is written as 
\[
\Phi^{2,3,4}\Phi^{1,23,4} \Phi^{1,2,3} = \Phi^{1,2,34}\Phi^{12,3,4}
\]
in the \cref{Problem 2}, whereas the pentagon equation of \cref{thm: presentation pab}, translated into the same notation, are 
in the opposite order
\[
\Phi^{1,2,3}\Phi^{1,23,4} \Phi^{2,3,4} = \Phi^{12,3,4}\Phi^{1,2,34}~.
\]
This is due to our unusual choice of convention, that leads us to view the product $ab$ of two elements of an algebra as a concatenation of arrows, 
and thus as the \textit{opposite} of their composition (when one sees the algebra as a category with only one object). 
\end{remark}

Similarly, this presentation of $\pab$ allows us to give an explicit description of the Grothendieck--Teichmüller group.

\begin{theorem}\label{presentation of GT}
Elements of the Grothendieck--Teichmüller group $\mathrm{GT}$ are in bijection with pairs $(\mu,f)$ in $\kk^\times \times \widehat{\mathbb{F}}_2(\kk)$, where 
$\widehat{\mathbb{F}}_2(\kk)$ denotes the Malcev completion of the free group on two generators, which satisfy the following conditions
\begin{enumerate}
\medskip

\item $f(x,y)^{-1} = f(y,x)~$ in $\widehat{\mathbb{F}}_2(\kk)~.$

\medskip

\item $x^\nu f(x,y) y^\nu f(y,z) z^\nu f(z,x) = 1$ if $xyz = 1$ in $\widehat{\mathbb{F}}_3(\kk)$, where $\nu = \frac{\mu - 1}{2}~.$

\medskip

\item Let $x_{ij}$ denote the standard generator of $\widehat{PB}_4(\kk)$, the Malcev completion of the pure braid group on four strands, defined in 
\cref{thm: Artin's presentation}. Then the element $f$ satisfies 
\[
f(x_{12},x_{23})f(x_{12}x_{13},x_{24}x_{34})f(x_{23},x_{34}) = f(x_{13}x_{23},x_{34})f(x_{12},x_{23}x_{24})
\]
in $\widehat{PB}_4(\kk)$.
\end{enumerate}

\medskip

Under this bijection, the group structure of the Grothendieck--Teichmüller group is given by: 
\[
(\mu_1,f_1) \star (\mu_2,f_2) = \Big(\mu_1\mu_2, f_1\big(x^{\mu_2},f_2(x,y)y^{\mu_2}f_2(y,x)\big)f_2(x,y)\Big)~.
\]
Under this bijection, its action on the set of associators is given by:
\[
(\mu,f) \bullet (\lambda,\Phi) = \left(\mu \lambda, f\left(e^{\mu t_{12}}, \Phi(t_{12},t_{23})e^{\mu t_{23}}\Phi(t_{23}, t_{12})\right)\Phi(t_{12},t_{23})\right)~.
\]
\end{theorem}

\begin{proof}
An explicit proof of this can be found in \cite[Theorem 11.1.7]{FresseGT1}. 
\end{proof}

\begin{remark}
Beware that \cite{FresseGT1} uses a different presentation of the pure braid group. In \textit{op.cit.}, $x_{ij}$ denotes the pure braid that goes 
\textit{from} the $j$-th strand, passing \textit{behind} other strands, does a loop around the $i$-th strand, and comes back passing \textit{behind} again. 
Hence the formula given in condition (3) of Theorem \ref{presentation of GT} does not coincide with the one appearing in \cite[Theorem 11.1.7]{FresseGT1}. 
%\[
%f(x_{12},x_{23})f(x_{13}x_{12},x_{34}x_{23})f(x_{23},x_{34}) = f(x_{23}x_{13},x_{34})f(x_{12},x_{24}x_{23})~,
%\]
%and does not coincide with the one given above. See Subsection \ref{ssec-conventions}.
\end{remark}

\begin{remark}
Using \cref{rmk: presentation of PaCD}, one can also give an explicit description of the \textit{graded Grothendieck--Teichmüller group} $\mathrm{GRT}(\kk)$, 
which coincides with the original definition of Drinfeld. 
\end{remark}

\subsection{Topological description of $\pab$}\label{Subsection: Description topo}

The operad in the category of groupoids $\pab$ can be obtained via a topological construction. For this, we consider the configuration spaces 

\[
\mathrm{Conf}(\mathbb{C},\mathrm{I}) \coloneqq \left\{ (x_1, \cdots, x_n) \in \mathbb{C}^\mathrm{I}~~|~~x_i \neq x_j~~\text{if}~~i \neq j \right\}~,
\]
\vspace{0.1pc}

for any finite set $\mathrm{I}$. We then consider the space

\[
\mathrm{C}(\mathbb{C},\mathrm{I}) \coloneqq \mathrm{Conf}(\mathbb{C},\mathrm{I})/\mathbb{C} \ltimes \mathbb{R}_{> 0}~,
\]
\vspace{0.1pc}

of configurations of (pairwise distinct) points modulo translations and dilations. Their Fulton--MacPherson compactifications $\overline{\mathrm{C}}(\mathbb{C},\mathrm{I})$ can be endowed with a canonical operad structure. Indeed, one can compute that the boundary of 

\[
\partial~\overline{\mathrm{C}}(\mathbb{C},\mathrm{I}) \cong \bigcup_{ k \geq 0} \bigcup_{\mathrm{J_1} \sqcup \cdots \sqcup\mathrm{J_k} = \mathrm{I}}  \overline{\mathrm{C}}(\mathbb{C},[k]) \times \left(\prod_{j =1}^k \overline{\mathrm{C}}(\mathbb{C},\mathrm{J}_j) \right)
\]

where $[k] = \{1, \cdots, k \}$. Therefore the operad structure is simply given by the inclusions 

\[
\gamma_{\mathrm{J_1}, \cdots, \mathrm{J_k}}:  \overline{\mathrm{C}}(\mathbb{C},[k]) \times \left(\prod_{j =1}^k \overline{\mathrm{C}}(\mathbb{C},\mathrm{J}_j) \right) \hookrightarrow \partial~\overline{\mathrm{C}}(\mathbb{C},\mathrm{J_1} \sqcup \cdots \sqcup \mathrm{J_k}) \hookrightarrow \overline{\mathrm{C}}(\mathbb{C},\mathrm{J_1} \sqcup \cdots \sqcup \mathrm{J_k})~.
\]

We denote this operad in topological spaces by $\overline{\mathrm{C}}(\mathbb{C})$.

\begin{remark}
There exists no choice of a family of points in $\overline{\mathrm{C}}(\mathbb{C})$ compatible with the operad structure. In other words, this operad in topological spaces can not be promoted into an operad in pointed topological spaces.
\end{remark}

\begin{remark}
There is a direct weak-equivalence of operads $\overline{\mathrm{C}}(\mathbb{C}) \qi \mathbb{E}_2$, where $\mathbb{E}_2$ is the little disks operad. See \cite{hoefel2011explicit}.
\end{remark}

Recall $\mathcal{P}a$, the operad in the category of sets generated by a single arity two operation. Elements in $\mathcal{P}a(n)$ are in bijection with maximally parenthesized permutations of $\mathfrak{S}_n$. Now lets view this operad as a topological operad where we impose the discrete topology on the sets $\mathcal{P}a(n)$.

\begin{lemma}\label{lemme: morphisme topo operades}
There is an inclusion of topological operads 

\[
\iota: \mathcal{P}a \hookrightarrow \overline{\mathrm{C}}(\mathbb{C})~. 
\]

\end{lemma}

Pictorially, the morphism $\iota$ establishes the following correspondence 

\[
\includegraphics[width=85mm,scale=1]{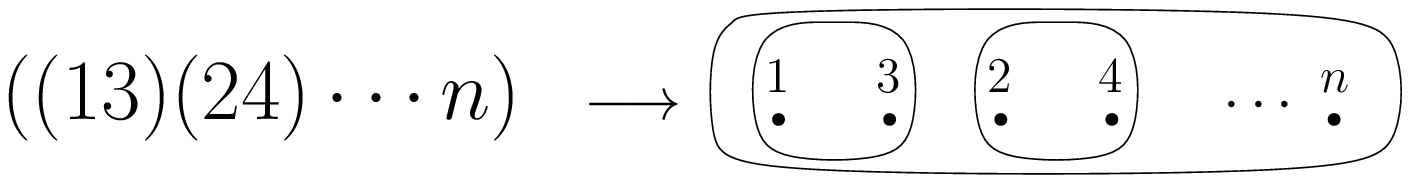}
\]

between maximally parenthesized permutations and configurations of labeled $n$ points inside the real line $\mathbb{R}\subset\mathbb{C}$, 
where the parenthesis on the right denote points which are infinitesimally close in the Fulton-MacPherson compactification.

\begin{theorem}\label{thm: pab iso topo}
There is an isomorphism of operads in the category of groupoids

\[
\pab \cong \Pi_1(\overline{\mathrm{C}}(\mathbb{C}), \mathcal{P}a)~.
\]

\end{theorem}

\begin{proof}[Sketch of proof]
Both of these operads in groupoids have the same underlying operad of objects, so we can consider the identity morphism on objects. Therefore what is left is to 
identify the automorphism groups on each side. Take the leftmost parenthesization of the trivial permutation $(((12)3) \cdots n)$ as an object, one can construct 
an explicit isomorphism 

\[
\mathrm{PB}_n \cong \mathrm{Aut}_{\pab(n)}((((12)3) \cdots n)) \cong \pi_1(\overline{\mathrm{C}}(\mathbb{C}),\iota(((12)3) \cdots n))
\]
\vspace{0.1pc}

by sending the generator $x_{ij}$ to the loop where the $i$-th point travels around only the $j$-th point and then goes back to its original position. 
\end{proof}

\newpage

%%%%%%%%%%%%% SECTION 3 %%%%%%%%%%%%%%%

\section{Cyclotomic associators}\label{section-cyclotomic}

\subsection{Motivation}\label{ssec-cyclotomic-motivation}

\subsubsection*{A variation on the quantization problem}

Let $\mathfrak{g}$ be a Lie algebra and let $\sigma: \mathfrak{g} \to \mathfrak{g}$ be an automorphism of order $N$. It defines a Lie subalgebra 
of fixed points $\mathfrak{h} \coloneqq \mathfrak{g}^\sigma$. One can always get a decomposition $\mathfrak{g} \cong \mathfrak{h} \oplus \mathfrak{m}$ 
as an $\mathfrak{h}$-module, where $\mathfrak{m}$ is the direct sum of the other eigenspaces of $\sigma$. Let $t$ be a $\sigma$-invariant Casimir element of 
$\mathfrak{g}$; it decomposes as $t = t_\mathfrak{h} + t_\mathfrak{m}~,$ where $t_\mathfrak{h}\in S^2(\mathfrak{h})$ and 
$t_\mathfrak{m}\in S^2(\mathfrak{m})$. 

\medskip

The category $\mathsf{Rep}\left(\mathfrak{U}(\mathfrak{g})\rtimes \Gamma\right)$, where $\Gamma:=\mathbb{Z}/N\mathbb{Z}$, is a symmetric monoidal category, 
thus a braided monoidal category. One can show that the category $\mathsf{Rep}(\mathfrak{h})$ is in fact a \textit{braided module category} 
over the braided monoidal category $\mathsf{Rep}\left(\mathfrak{U}(\mathfrak{g})\rtimes \Gamma\right)$. 
See \cite{Enriquez07} and \cite{Brochier13} for more on this notion. 
Equivalently, one can say that $\mathsf{Rep}(\mathfrak{h})$ is a $\sigma$-braided module category over $\mathsf{Rep}(\mathfrak{g})$ in the sense of 
\cite{deCommer}. 

\medskip

We have already seen that $t$ can be used to define a first order deformation of 
$\mathsf{Rep}(\mathfrak{g})$ as a braided module category. 
Similarly, following \cite{Enriquez07,Brochier13,deCommer}, $t_\mathfrak{h}$ 
can be used to define a first order deformation of $\mathsf{Rep}(\mathfrak{h})$ as a $\sigma$-braided module category over (the above mentionned first order 
deformation of) $\mathsf{Rep}(\mathfrak{g})$. 
Thus, \cref{Problem 1} of the introduction can be extended to this setting, and stated very shortly as follows: can one extend this first order deformation 
into a formal deformation over $\kk[[\hbar]]$?

\medskip

\textit{Cyclotomic associators} provide a universal answer to this type of deformation quantization problems analogously 
to classical associators before.

\subsubsection*{Universal reformulation}

Let us be a bit more specific about the meaning of ``a universal answer'' in the previous paragraph. Recall the morphisms 
\[
\varphi_t^n:\mathfrak{U}(\mathfrak{t}_n)\to \big(\mathfrak{U}(\mathfrak{g})^{\otimes n}\big)^{\mathfrak{g}}\cong \mathsf{End}(\otimes^n)\,,
\]
that are compatible with insertion-coproduct morphisms. There is a cyclotomic version $\mathfrak{t}_n^\Gamma$, $\Gamma:=\mathbb{Z}/N\mathbb{Z}$, 
of the Drinfeld--Kohno graded Lie algebra (see Definition \ref{def: cyclotomic Kohno-Drinfeld} below). It comes along with algebra morphisms 
\[
\varphi_{t,\sigma}^n:\mathfrak{U}(\mathfrak{t}_n^\Gamma)\to 
\Big(\mathfrak{U}(\mathfrak{h})\otimes\big(\mathfrak{U}(\mathfrak{g})\rtimes\Gamma\big)^{\otimes n}\Big)^{\mathfrak{h}}
\cong \mathsf{End}(\odot^n)~,
\]
where the functor 
$\odot^n:\mathsf{Rep}(\mathfrak{h})\times\mathsf{Rep}\left(\mathfrak{U}(\mathfrak{g})\rtimes \mathbb{Z}/N\mathbb{Z}\right)^n\to \mathsf{Rep}(\mathfrak{h})$ 
is the $n$-fold module structure of $\mathsf{Rep}(\mathfrak{h})$. 

It turns out that the natural transformations and relations defining a braided module category sit in $\mathsf{End}(\odot^n)$ and can be expressed using 
insertion-coproduct type morphisms. Such insertion-coproduct morphisms also exist on the cyclotomic Drinfeld--Kohno Lie agebras $\mathfrak{t}_n^\Gamma$, and 
the morphisms $\varphi_{t,\sigma}^n$ are compatible with them. Therefore, a universal solution to the quantization problem is a solution that lives at the level 
of $\mathfrak{U}(\mathfrak{t}_n^\Gamma)$ (in fact, it lives in its degree completion, see \cref{completed variant}). 

\subsubsection*{Number theoretic aspects}

As mentioned in the previous Section, the absolute Galois group $\mathrm{Gal}(\overline{\mathbb{Q}}/\mathbb{Q})$ embeds into the pro-finite version 
$\widehat{\mathrm{GT}}$ 
of the Grothendieck--Teichmüller group. Cyclotomic associators provide new versions of the pro-finite Grothendieck--Teichmüller group, and the absolute Galois groups 
$\mathrm{Gal}(\overline{\mathbb{Q}}/\mathbb{Q}(\mu_N))$ embed into these. Here $\mathbb{Q}(\mu_N)$ stands for the $N$-th cyclotomic extension of 
$\mathbb{Q}$, and $\mu_N$ is the group of $N$-th roots of $1$. 

\subsubsection*{Motivic aspects}

Enriquez's proof \cite{Enriquez07} of the existence of a cyclotmic associator is also motivic, as an explicit cyclotomic associator $\Psi_{KZ}$ is constructed as 
the regularized holonomy of an algebraic flat connection defined on $\mathbb{P}^1-(\{0,+\infty\}\cup \mu_N)$ that generalizes the KZ connection. 
The cyclotomic associator $\Psi_{KZ}$ is in fact a generating series for special values of multiple polylogarithms at $N$-th roots of $1$ (see \textit{loc.cit}). 
These periods, sometimes called $N$-coloured MZV, have very interesting combinatorial \cite{Racinet} and motivic 
\cite{Goncharov-cyclotomy,Goncharov-dihedral,goncharov2001multiple,Goncharov-ICM,Deligne-Goncharov} features.   

\subsection{Moperads}\label{ssec-Moperads}

Let $(\mathcal{E},\otimes,\mathbb{1})$ be a cocomplete symmetric monoidal category such that the monoidal product $\otimes$ commutes with small colimits 
in each variable. The category of symmetric sequences $\mathbb{S}\text{-}\mathsf{mod}_{\mathcal{E}}$ has a monoidal structure other than the pleythism, which 
is symmetric and is given by
\[
 (M \circledast N)(n) \coloneqq \displaystyle \amalg_{p + q = n} \left(M(p) \otimes N(q)\right)_{\mathbb{S}_p \times \mathbb{S}_q}^{\mathbb{S}_n}~,
\]
for $M,N$ two symmetric sequence in $\mathcal{E}$. The unit for this symmetric monoidal structure is given by the symmetric sequence $\mathbb{1}_\circledast$ which 
is $\mathbb{1}$ in arity zero and zero elsewhere. This product is compatible with the $\circ$-product in the following way: there is a natural transformation
\[
\varepsilon_{M,N,L}: M \circledast (N \circ L) \longrightarrow (M \circledast N) \circ L~,
\]
for any triple of symmetric sequences $(M,N,L)$ which is given by assembling the evident maps 
\[
\begin{tikzcd}
M(p) \otimes (N(q) \otimes L(i_1) \otimes \cdots \otimes L(i_q)) \arrow[d] \\
(M(p) \otimes N(q)) \otimes L(i_1) \otimes \cdots \otimes L(i_q) \otimes L(1) \otimes \cdots \otimes L(1) 
\end{tikzcd}
\]
\begin{definition}[Moperad] 
Let $\mathcal{P}$ be an operad. A \textit{moperad} $(M, \gamma_M, \eta, \gamma_\mathcal{P})$ over $\mathcal{P}$ is the data of a symmetric sequence $M$ such that
\begin{enumerate}
\medskip

\item $(M, \gamma_M, \eta)$ forms a unital monoid for the $\circledast$ product.

\medskip

\item $(M,\gamma_\mathcal{P},\eta)$ is a right module over the operad $\mathcal{P}$.

\medskip

\end{enumerate}
These structures make the following diagram commutes
\[
\begin{tikzcd}
M \circledast (M \circ \mathcal{P}) \arrow[rr,"\gamma_\mathcal{P}"] \arrow[d,"\varepsilon_{M,M,\mathcal{P}}",swap]
&
&M \circledast M \arrow[d,"\gamma_M"] \\
(M \circledast M) \circ \mathcal{P} \arrow[r,"\gamma_M"]
&M \circ \mathcal{P} \arrow[r,"\gamma_\mathcal{P}"]
&M~.
\end{tikzcd}
\]
\end{definition}

\begin{remark}
This structure was introduced for the first time in \cite{Willwachermoperads}. See \textit{loc.cit} or \cite{calaque20moperadic}.
\end{remark}

This structure can be rephrased in terms of partial composition maps as the data of
\[
\begin{tikzcd}[column sep=1.5pc,row sep=0pc]
\left\{
    \arraycolsep=1.4pt\def\arraystretch{2.2}\begin{array}{ll}
        \circ_0: M(k) \otimes M(m) \longrightarrow M(k+m)~, \\
        \circ_i: M(k) \otimes \mathcal{P}(m) \longrightarrow M(k+m-1)~~ \text{for}~~ 1 \leq i \leq k, \\
    \end{array}
\right.
\end{tikzcd}
\]
which satisfy analogous compatibility conditions. One can visualize these compositions as follows: operations in $M(k)$ are represented as rooted trees with $k+1$-leaves, where the up most left leaf is fixed and not acted upon by the symmetric group $\mathfrak{S}_k$. This leaf is called the $0$-th leaf. The compositions $\circ_0$ amount to inserting operations in $M$ into the $0$-th leaf, and the compositions $\circ_i$ amount to inserting operations in $\mathcal{P}$ into the other, non-fixed, leaves. Pictorially it gives 

\[
\includegraphics[width=120mm,scale=1]{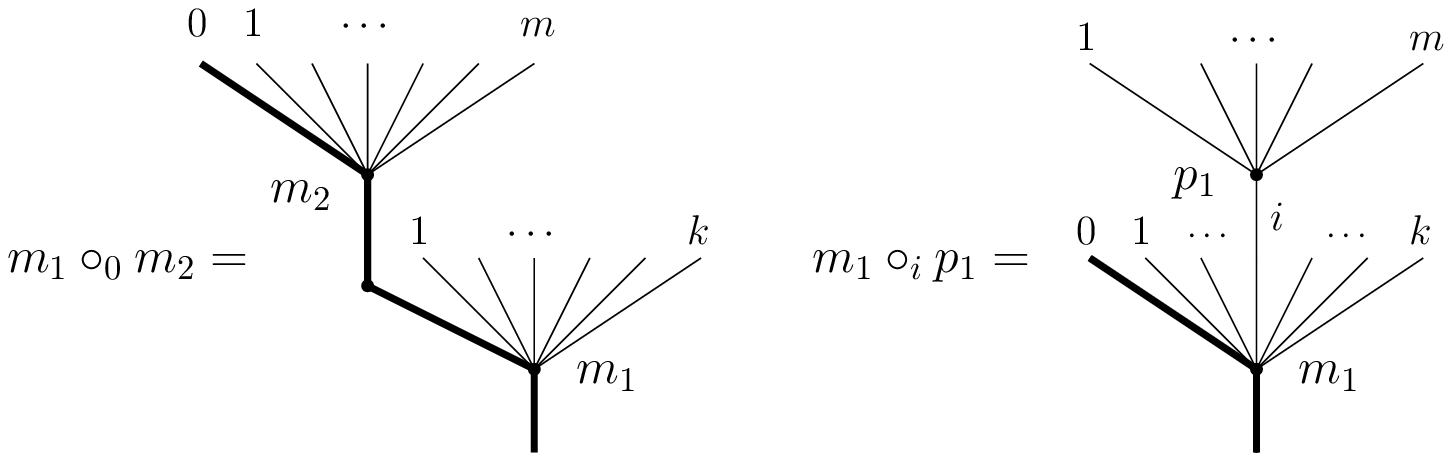}~,
\]

for $m_1$ an operation in $M(k)$, $m_2$ an operation in $M(m)$ and $p_1$ an operation in $\mathcal{P}(m)$.

\subsubsection*{Moperads as the derivatives of operads}
Symmetric sequences in $\mathcal{E}$ can equivalently be described as \textit{species of structures}, that is, as functors 
\[
M: \mathrm{Fin}_{\mathrm{bij}}^{\mathrm{op}} \longrightarrow \mathcal{E}
\]
from the category of finite sets with bijections to the category $\mathcal{E}$. The theory of species of structures was introduced by Joyal in \cite{Joyalspecies}. We refer to \cite{Bergeronspecies} for more details.  

\begin{remark}
In order to extend a symmetric sequence to a specie of structures, there is a canonical choice given by 
\[
M(I) \coloneqq \left(\coprod_{f: I \longrightarrow [n]} M(n)\right)_{\mathfrak{S}_n}~,
\]
where $I$ is a finite set with $n$ elements, $[n] = \{1,\cdots,n\}$ and the sum is over all bijections $f$. The notion of an operad and a moperad can be extended to species of structures in a straightforward way.
\end{remark}

\begin{definition}[Derivative of a specie of structures]
Let 
\[
M: \mathrm{Fin}_{\mathrm{bij}}^{\mathrm{op}} \longrightarrow \mathcal{E}
\]
be a specie of structure in $\mathcal{E}$. Its \textit{derivative} $M'$ is the specie of structures whose image is given by 
\[
M'(I) \coloneqq M(I \sqcup \{*\})
\]
for any finite set $I$.
\end{definition}

The derivatives are a classical operation on species of structures: if the specie happens to be an operad, then its derivative is in fact a moperad over it.

\begin{proposition}\label{prop: derived moperad of an operad}
Let $\mathcal{P}$ be a specie of structure in $\mathcal{E}$ endowed with an operad structure. Then $\mathcal{P}'$ is canonically endowed with a moperad structure over $\mathcal{P}$ given as follows.

\medskip

\begin{enumerate}
\item The monoid structure
\[
\circ_0: \mathcal{P}'(I) \otimes \mathcal{P}'(J) \longrightarrow \mathcal{P}'(I \sqcup J)
\]
is given by the operadic composition 
\[
\circ_{\{*\}}: \mathcal{P}(I \sqcup \{*\}) \otimes \mathcal{P}(J \sqcup \{*\}) \longrightarrow \mathcal{P}(I \sqcup J \sqcup  \{*\})
\]
at the added point $\{*\}$.

\medskip

\item The right module structure 
\[
\circ_i: \mathcal{P}'(I) \otimes \mathcal{P}(J) \longrightarrow \mathcal{P}'((I-\{i\}) \sqcup J)
\]
is given by the operadic composition 
\[
\circ_{i}: \mathcal{P}(I \sqcup \{*\}) \otimes \mathcal{P}(J) \longrightarrow \mathcal{P}((I-\{i\}) \sqcup J \sqcup \{*\})
\]
at the point $i$ in $I$.
\end{enumerate}
\end{proposition}

\begin{proof}
It is straightforward to check that the associativity conditions on the operad structure of $\mathcal{P}$ induce those of the moperad structure on $\mathcal{P}'$.
\end{proof}

\begin{example}
Let $\mathcal{M}\mathcal{U}nit$ be the symmetric sequence given for all $n$ by $\mathcal{M}\mathcal{U}nit(n) \coloneqq \mathbb{1}$, then $\mathcal{M}\mathcal{U}nit$ forms a moperad over $\mathcal{U}nit$ with the obvious compositions maps. One can easily see that $\mathcal{M}\mathcal{U}nit$ is the derivative of the operad 
$u\mathcal{C}om$, defined by $u\mathcal{C}om(n) \coloneqq \mathbb{1}$, encoding unital commutative algebras in $\mathcal E$. \hfill$\triangle$
\end{example}

Recall our convention that all operads $\mathcal{P}$ come equipped with a morphism of operads $f: \mathcal{U}nit \longrightarrow \mathcal{P}$ which is an 
isomorphism in arities zero and one, that is, our operads are \textit{pointed reduced}. We will use the same convention for moperads: any moperad $M$ over an 
operad $\mathcal{P}$ will come with a morphism of moperads $\varphi_f: \mathcal{M}\mathcal{U}nit \longrightarrow M$ over the above mentioned morphism $f$. 
This implies that there is a distinguished morphism of symmetric sequences $\mathcal{P} \longrightarrow M$ for all the moperads considered from now on. 

\begin{example}\label{example: moperad Pa0}
The moperad $\mathcal{P}a_0$ over the operad of parenthesized permutations $\mathcal{P}a$ in the category of sets is defined as follows: the set 
$\mathcal{P}a_0(n)$ is given by the set of maximal parenthesization of $0\hspace{1pt}\sigma_1 \cdots \sigma_n$, where $\sigma_1 \cdots \sigma_n$ is a 
permutation of $\mathfrak{S}_n$. The right module structure of $\mathcal{P}a_0$ over $\mathcal{P}a$ is given by the insertion of parenthesized permutations 
on the right of the zero element: $(0\hspace{1pt}\sigma) \circ_i \tau \coloneqq 0\hspace{1pt }\sigma \circ_i \tau~ ,$ for any parenthesized permutation $\tau$. 
For example,
\[
(01)2 \circ_1 (13)2 = (0(13)2)4~.
\]
The moperad structure is given by
\[
0\hspace{1pt}\sigma \circ_0 0\hspace{1pt}\tau \coloneqq 0\hspace{1pt}\tau \sigma~,
\]
where one inserts $0\hspace{1pt}\tau$ together with its maximal parenthesization into the zero spot of the fist maximally parenthesized word 
$0\hspace{1pt}\sigma$, giving another maximally parenthesized word. For example,
\[
(01)2 \circ_0 0(3(21)) = (0(3(21))4)5~.
\]
The inclusion of $\mathcal{M}\mathcal{U}nit$ into $\mathcal{P}a_0$ sends $\mathcal{M}\mathcal{U}nit(n) = \{*\}$ into the leftmost parenthesization 
of $0\hspace{1pt}\mathrm{id}_n=01\cdots n$, where $\mathrm{id}_n$ is the trivial permutation in $\mathfrak{S}_n$. One can directly see that it forms a moperad by applying Proposition \ref{prop: derived moperad of an operad} to $\mathcal{P}a$ (indeed, $\mathcal{P}a_0=\mathcal{P}a'$). \hfill$\triangle$
\end{example}

\begin{example}
The moperad $\mathcal{C}o\mathcal{B}^{(1)}$ over the operad $\mathcal{C}o\mathcal{B}$ from \cref{example: cob}, is defined as follows: 
\begin{enumerate}
\item The objects of $\mathcal{C}o\mathcal{B}^{(1)}(n)$ are given by words $0\hspace{1pt}\sigma$, where $\sigma$ is a permutation of $\mathfrak{S}_n$. 
For instance, $0321$ is in $\mathrm{Ob}(\mathcal{C}o\mathcal{B}^{(1)}(3))$.

\medskip

\item The morphisms of $\mathcal{C}o\mathcal{B}^{(1)}(n)$ are given, for two words $0\hspace{1pt}\sigma$ and $0\hspace{1pt}\tau$, by the set of 
braids going from $0\hspace{1pt}\sigma$ to $0\hspace{1pt}\tau$, where the zeroth strand is frozen. Pictorially, we have: 

\[
\includegraphics[width=80mm,scale=1]{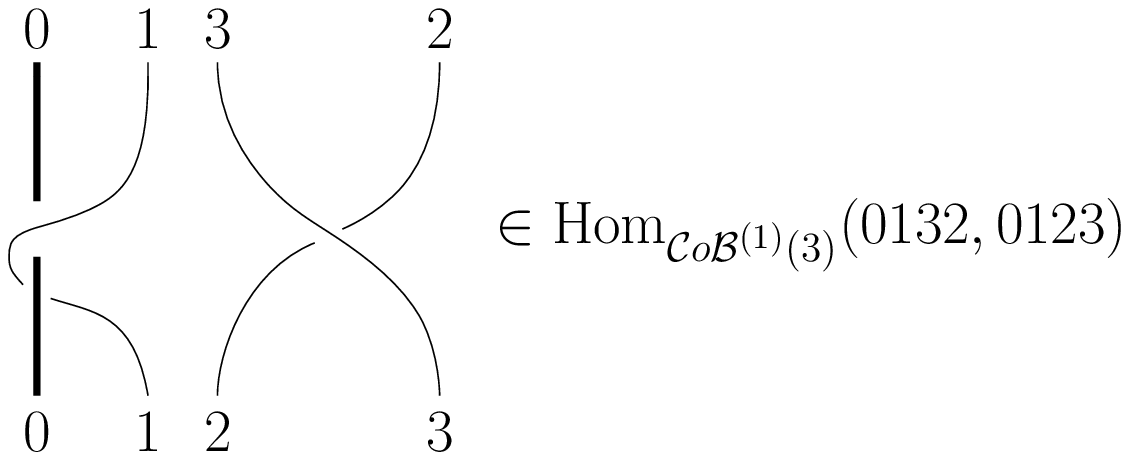}
\]

\item The right module structure of $\mathcal{C}o\mathcal{B}^{(1)}$ over $\mathcal{C}o\mathcal{B}$ is given by the composition of braids on the last $n$ spots. 
The monoid structure of $\mathcal{C}o\mathcal{B}^{(1)}$ is given by the insertion of the second braid into the zeroth spot of the first braid. The morphism from 
$\mathcal{M}\mathcal{U}nit$ to $\mathcal{C}o\mathcal{B}^{(1)}$ sends $\{*\}$ to $0 \hspace{1pt} \mathrm{id}_n$. 
\end{enumerate}
One can directly see that it forms a moperad by applying Proposition \ref{prop: derived moperad of an operad} to $\mathcal{C}o\mathcal{B}$ (indeed, $\mathcal{C}o\mathcal{B}^{(1)}=\mathcal{C}o\mathcal{B}'$). \hfill$\triangle$
\end{example}

Consider $\mathcal{P}a_0$ as a moperad over $\mathcal{P}a$ in groupoids where there are no non-trivial morphisms between objects. There is a morphism of moperads in sets $\psi: \mathrm{Ob}(\mathcal{P}a_0) \longrightarrow \mathrm{Ob}(\mathcal{C}o\mathcal{B}^{(1)})$ which lies above the morphism of operads in sets $\varphi: \mathrm{Ob}(\mathcal{P}a) \longrightarrow \mathrm{Ob}(\mathcal{C}o\mathcal{B})$. Both are given by forgetting the parenthesization. 

\begin{definition}[Moperad of parenthesized braids]
The \textit{moperad of parenthesized braids} $\pab^{(1)}$ over the operad in groupoids $\pab$ is defined to be the fake pull-back of $\mathcal{C}o\mathcal{B}^{(1)}$ along the morphism $\psi$.
\end{definition}

\begin{remark}
It is straightforward to check that the fake pull-back of \cref{def: fake-pullback} generalizes to moperads. 
\end{remark}

\begin{remark}
The moperad $\pab^{(1)}$ can also directly be obtained by applying Proposition \ref{prop: derived moperad of an operad} to $\pab$: $\pab^{(1)}=\pab'$. 
\end{remark}

\subsection{Moperads of parenthesized cyclotomic braids}\label{ssec-3.3}

\subsubsection*{Group actions on moperads}

Consider the terminal operad  $u\mathcal{C}om$ in the category of sets, given by $u\mathcal{C}om(n) = \{*\}$ for all $n \geq 0$ together with the obvious 
compositions. For any group $G$, one can defined a moperad $\underline{G}$ over $u\mathcal{C}om$ as follows
\begin{enumerate}
\medskip

\item The set $\underline{G}(n) \coloneqq G^{\times n}$, together with the evident $\mathfrak{S}_n$ action.

\medskip

\item The right module structure over $u\mathcal{C}om$ is given by 
\[
\begin{tikzcd}[column sep=1.5pc,row sep=0pc]
G^{\times n} \cong G^{\times n} \times u\mathcal{C}om(m) \arrow[r,"\circ_i^G"]
&G^{\times n + m -1} \\
(g_1,\cdots, g_n) \arrow[r,mapsto] 
&(g_1, \cdots, g_{i-1}, g_i, \cdots, g_i, g_{i+1}, \cdots, g_n)~,
\end{tikzcd}
\]
where we have $m$ copies of the element $g_i$, for $1 \leq i \leq n$.

\medskip

\item The monoid structure maps $\circ_0^G: G^{\times n} \times G^{\times m} \longrightarrow G^{\times (n+m)}$ are given by the evident isomorphism. 

\medskip
\end{enumerate}

The data of $\underline{G}$ over $u\mathcal{C}om$ actually forms a moperad in the category of groups, as all the above maps are compatible with the group structures.

\begin{definition}[$G$-action on moperads]\label{def: G-action moperad}
Let $M$ be a moperad over an operad $\mathcal{P}$ in the category of sets. The data of a $G$-action on the moperad $M$ is the data of a morphism of moperads over $\mathcal{P}$
\[
\gamma_G: \underline{G}\times M \longrightarrow M~,
\]
where $\underline{G}\times M$ is a moperad over $u\mathcal{C}om \times \mathcal{P} \cong \mathcal{P}$ under the natural identification. The structural morphism $\gamma_G$ satisfies the condition imposed by the commutativity of the following diagram
\[
\begin{tikzcd}[column sep=3pc,row sep=3pc]
\underline{G} \times \underline{G} \times M \arrow[r,"\mathrm{id} \times \gamma_G"]  \arrow[d,swap,"\mathrm{Mult}_{\underline{G}} \times \mathrm{id}_n"] 
&\underline{G} \times M \arrow[d,"\gamma_G"]\\
\underline{G} \times M \arrow[r,"\gamma_G"]
&M~,
\end{tikzcd}
\]
where the morphism $\mathrm{Mult}_{\underline{G}}$ is given in arity $n$ by the group multiplication of $G^{\times n}$.
\end{definition}

\begin{remark}
The data of a $G$-action on a moperad $M$ over $\mathcal{P}$ amounts to the data of an $\mathfrak{S}_n$-equivariant left $G^{\times n}$-action on $M(n)$ such that the right module structure maps $\{\circ_i\}$ are $\{\circ_i^G\}$-equivariant and such that the monoid maps $\{\circ_0\}$ is $G^{\times n+m}$-equivariant.
\end{remark}

\subsubsection*{Semi-direct products.}

Let $G$ be a group. A $G$-groupoid is the data of a groupoid $\mathcal{D}$ together with a morphism of groups from $G$ to the automorphisms group 
$\mathrm{Aut}_{\mathsf{Grpd}}(\mathcal{D})$ of the groupoid $\mathcal{D}$; hence each element $g$ of $G$ induces an automorphism of groupoids of 
$\mathcal{D}$ which will be denoted by $g \star - $. Together with $G$-equivariant morphisms, $G$-groupoids form a category, denoted by 
$G\text{-}\mathsf{Grpd}$. We denote $\mathrm{B}G$ the classifying groupoid of $G$, which has only one object $\{*\}$, whose automorphism group is precisely $G$.

\begin{definition}[Semi-direct product]
The \textit{semi-direct product} is a functor
\[
\begin{tikzcd}[column sep=1.5pc,row sep=0pc]
G\text{-}\mathsf{Grpd} \arrow[r]
&\mathsf{Grpd}/\mathrm{B}G \\
\mathcal{D} \arrow[r,mapsto] 
&\mathcal{D} \rtimes G~.
\end{tikzcd}
\]
The semi-direct product $\mathcal{D} \rtimes G$, which is a groupoid above the groupoid $\mathrm{B}G$, is the groupoid defined as follows:
\begin{enumerate}
\medskip

\item The object of $\mathcal{D} \rtimes G$ are the objects of $\mathcal{D}$.

\medskip

\item Let $d$ be in $\mathrm{Ob}(\mathcal{D} \rtimes G)$ and $g$ be in $G$, then 
\[
\mathrm{Hom}_{\mathcal{D} \rtimes G}(g \star d, d) \coloneqq \mathrm{Hom}_{\mathcal{D}}(g \star d, d) \amalg \{\varepsilon_g\}~,
\]
that is, there is a free added arrow $\varepsilon_g: g \star d \longrightarrow d$ for all $d$ and all $d$. The set of free added arrows satisfies the following conditions:
\begin{enumerate}
\medskip

\item For any $g,h$ in $G$ and object $d$:
\[
\left(\begin{tikzcd}[column sep=1.5pc,row sep=0pc]
g \star (h \star d) \arrow[r,"\varepsilon_g"]
&h \star d \arrow[r,"\varepsilon_h"]
&d
\end{tikzcd}\right)
~=~ %\arrow[r,equal]
\left(\begin{tikzcd}[column sep=1.5pc,row sep=0pc]
(g.h) \star d \arrow[r,"\varepsilon_{g.h}"]
&d~,
\end{tikzcd}\right)
\]
where $g \star d$ denotes the image of the object $d$ via the functor $g \star -$. 

\medskip

\item For any morphism $f:d \longrightarrow e$ in $\mathcal{D}$ and any $g$ in $G$:
\[
\left(\begin{tikzcd}[column sep=1.5pc,row sep=0pc]
g \star d \arrow[r,"\varepsilon_g"]
&d \arrow[r,"f"]
&e \arrow[r,"(\varepsilon_{g})^{-1}"]
&g \star e
\end{tikzcd}\right)
~=~% \arrow[r,equal]
\left(\begin{tikzcd}[column sep=1.5pc,row sep=0pc]
g \star d \arrow[r,"g \star f"]
&g \star e~,
\end{tikzcd}\right)
\]
where $g \star f$ denotes the image of the arrow $f$ via the functor $g \star -$. 
\end{enumerate}

\medskip

\item The structural morphism $\mathcal{D} \rtimes G \longrightarrow \mathrm{B}G$ sends every object to $\{*\}$ and every arrow to $\mathrm{id}_*$ except the arrows $\{\varepsilon_g\}$, which are sent to $g$, for every $g$ in $G$. 
\end{enumerate}
\end{definition}

There is another functor, which associates to any groupoid above $\mathrm{B}G$ a groupoid endowed with a $G$-action.

\begin{definition}[Pseudo-free $G$-groupoid]
The \textit{pseudo-free} $G$-\textit{groupoid} is a functor
\[
\begin{tikzcd}[column sep=1.5pc,row sep=0pc]
\mathsf{Grpd}/\mathrm{B}G \arrow[r]
& G\text{-}\mathsf{Grpd}\\
\mathcal{Q} \arrow[r,mapsto] 
&\mathcal{G}(\mathcal{Q})~.
\end{tikzcd}
\]
Let $p: \mathcal{Q} \longrightarrow \mathrm{B}G$ be a groupoid above $\mathrm{B}G$, the free $G$-groupoid generated by $p: \mathcal{Q} \longrightarrow \mathrm{B}G$, denoted by $\mathcal{G}(\mathcal{Q})$, is defined as follows:
\begin{enumerate}
\medskip

\item The set of objects of $\mathcal{G}(\mathcal{Q})$ is given by $\mathrm{Ob}(\mathcal{Q}) \times G$. 

\medskip

\item The set of morphisms 
\[
\mathrm{Hom}_{\mathcal{G}(\mathcal{Q})}((x,g),(y,h)) \coloneqq \{f \in \mathrm{Hom}_{\mathcal{Q}}(x,y)~|~ g \bullet p(f) = h \}~,
\]
for $x,y$ in $\mathrm{Ob}(\mathcal{Q})$ and $g,h$ in $G$.

\medskip

\item The $G$-action $g \star -$ sends $(x,h)$ to $(x,gh)$ and sends any arrow $f: (x_1,h_1) \longrightarrow (x_2,h_2)$ to $f:(x_1,gh_1) \longrightarrow (x_2,gh_2)$. 
\end{enumerate}
\end{definition}

\begin{remark}
The pseudo-free $G$-construction and the semi-direct product construction do \textit{not} define adjoint functors. 
\end{remark}

\begin{remark}
Let $(x,g)$ be an object of $\mathcal{G}(\mathcal{Q})$, the automorphisms group $\mathrm{Aut}_{\mathcal{G}(\mathcal{Q})}((x,g))$ given by the automorphisms 
$f: x \longrightarrow x$ in $\mathcal{Q}$ such that $p(f) = \mathrm{id}_{p(x)}~.$
\end{remark}

\begin{example}
Let $\mathrm{B}_n$ be the braid group on $n$ strands. The morphism $\mathrm{B}_n \twoheadrightarrow \mathfrak{S}_n$ gives a morphism of groupoids 
$p:\mathrm{BB}_n \twoheadrightarrow \mathrm{B}\mathfrak{S}_n$. The free $\mathfrak{S}_n$-groupoid generated by $p$ is given by the groupoid 
$\mathcal{C}o\mathcal{B}(n)$, with its corresponding $\mathfrak{S}_n$-action. \hfill$\triangle$
\end{example}

We denote again $u\mathcal{C}om$ the terminal operad in the category of groupoids, given by the trivial groupoid in each arity. Since the classifying groupoid functor 
$\mathrm{B}(-)$ is strong monoidal, the moperad $\underline{G}$ in sets is sent to a moperad $\mathrm{B}\underline{G}$ over $u\mathcal{C}om$ in the category of 
groupoids. 

\medskip

Let $\mathcal{P}$ be an operad in groupoids. Since the above-mentioned adjunction is strong monoidal, any moperad $M$ over $\mathcal{P}$ together with a morphism 
of moperads $\phi: M \longrightarrow \mathrm{B}\underline{G}$ above the trivial morphism $\mathcal{P} \twoheadrightarrow u\mathcal{C}om$ will give, under the 
functor $\mathcal{G}(-)$, a moperad $\mathcal{G}(M)$ over $\mathcal{P}$ which comes equipped with a $\mathrm{B}\underline{G}$-action (in the sense of Definition 
\ref{def: G-action moperad} extended to the groupoid case). 

\begin{example}\label{example: cobN}
There is a morphism of moperads $\iota: \mathcal{C}o\mathcal{B}^{(1)} \longrightarrow \mathrm{B}\underline{\mathbb{Z}/N\mathbb{Z}}$ above the trivial morphism $\mathcal{C}o\mathcal{B} \twoheadrightarrow u\mathcal{C}om~.$ On objects, it sends any word $0 \hspace{1pt} \sigma$ in $\mathcal{C}o\mathcal{B}^{(1)}(n)$ to the single object $\{*\}$. On morphisms, it sends a braid between two words $0 \hspace{1pt} \sigma$ and $0 \hspace{1pt} \tau$ in $\mathcal{C}o\mathcal{B}^{(1)}(n)$ to the element $(\gamma_1, \cdots, \gamma_n)$ in $(\mathbb{Z}/N\mathbb{Z})^{\times n}$. The value of $\gamma_i$ is given by the number of loops that the $i$-th strand of the braid makes around the zeroth strand, counted module $N$. For instance, 
\[
\quad \quad \quad \quad \quad \mathcal{C}o\mathcal{B}(3) \longrightarrow \mathrm{B}(\mathbb{Z}/N\mathbb{Z})^3
\]
\vspace{0.1pc}
\[
\includegraphics[width=55mm,scale=1]{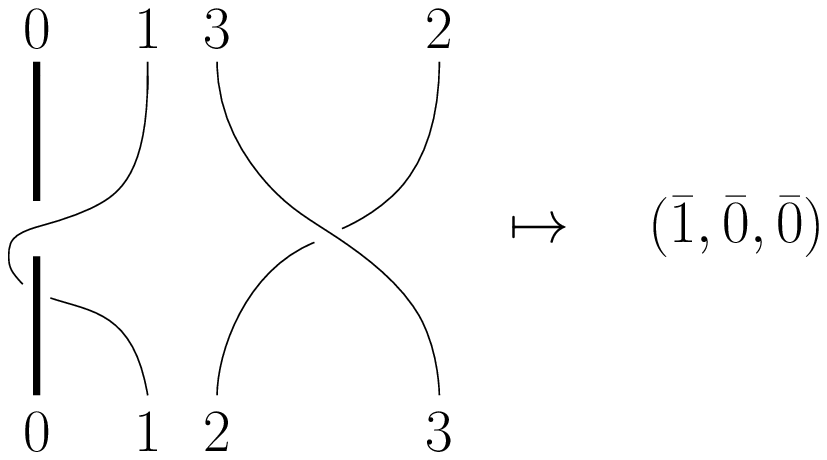}
\]
This allows us to define $\mathcal{C}o\mathcal{B}^{(N)}$ as the moperad over $\mathcal{C}o\mathcal{B}$ given, in arity $n$, by the free 
$(\mathbb{Z}/N\mathbb{Z})^n$-groupoid on the morphism $\iota(n): \mathcal{C}o\mathcal{B}^{(1)}(n) \longrightarrow \mathrm{B}((\mathbb{Z}/N\mathbb{Z})^n)$. 
Objects of $\mathcal{C}o\mathcal{B}^{(N)}(n)$ are given by words $0 \hspace{1pt} \sigma_1 \cdots \sigma_n$ labeled by $(\gamma_1, \cdots, \gamma_n)$ in 
$(\mathbb{Z}/N\mathbb{Z})^n$. We will denote them $0 \hspace{1pt} (\sigma_1)_{\gamma_1} \cdots (\sigma_n)_{\gamma_n}$. Morphisms between two labeled 
words $0 \hspace{1pt} (\sigma_1)_{\gamma_1} \cdots (\sigma_n)_{\gamma_n}$ and $0 \hspace{1pt} (\tau_1)_{\lambda_1} \cdots (\tau_n)_{\lambda_n}$ are given 
by braids which satisfy the following condition: the number of loops that the $i$-th strand does around the zeroth strand is $\lambda_i - \gamma_i$ modulo $N$ for all 
$1 \leq i \leq n$.  \hfill$\triangle$
\end{example}

Similarly to the previous example, there is also a morphism of moperads 
$\jmath: \mathcal{P}a\mathcal{B}^{(1)} \longrightarrow \mathrm{B}\underline{\mathbb{Z}/N\mathbb{Z}}$ above the 
trivial morphism $\pab \twoheadrightarrow u\mathcal{C}om$. It again sends parenthesized permutations to $\{*\}$ and braids between them to the numbers 
of loops around the zeroth strand modulo $N$.

\begin{definition}[Moperad of parenthesized cyclotomic braids]
The \textit{moperad of parenthesized braids} $\pab^{(N)}$ is the moperad over $\pab$ defined in each arity by
\[
\pab^{(N)}(n) \coloneqq \mathcal{G}\left(\jmath(n): \mathcal{P}a\mathcal{B}^{(1)}(n) \longrightarrow \mathrm{B}((\mathbb{Z}/N\mathbb{Z})^n)\right)~.
\]
Its moperad structure comes from the fact that $\jmath: \mathcal{P}a\mathcal{B}^{(1)} \longrightarrow \mathrm{B}\underline{\mathbb{Z}/N\mathbb{Z}}$ is a morphism of moperads.
\end{definition}

\begin{remark}
The description of $\mathcal{C}o\mathcal{B}^{(N)}$ given in Example \ref{example: cobN} applies \textit{mutatis mutandis} to $\pab^{(N)}$ as well. In this case, the only real difference is that objects of $\mathcal{C}o\mathcal{B}^{(N)}(n)$ are words $0 \hspace{1pt} \sigma_1 \cdots \sigma_n$ while objects in $\pab^{(N)}(n)$ are words $0 \hspace{1pt} \sigma_1 \cdots \sigma_n$ together with a \textit{maximal parenthesization}.
\end{remark}

The moperad $\widehat{\pab}^{(N)}(\kk)$ over $\widehat{\pab}(\kk)$ is obtained by applying the Malcev completion functor to the aforementioned moperad. It gives a moperad in the category of pro-unipotent $\kk$-groupoids. This will be \textit{our first cyclotomic player}. 

\subsection{Infinitesimal cyclotomic braids}\label{subsec: inf cyclo braids}
In the last section, we defined the moperad $\pab^{(N)}$ over $\pab$, which will be the first player necessary to define cyclotomic associators. Now we need the "infinitesimal version" of this moperad. From now on, we fix $\Gamma \coloneqq \mathbb{Z}/N\mathbb{Z}$.

\begin{definition}[Cyclotomic Drinfeld--Kohno algebras]\label{def: cyclotomic Kohno-Drinfeld}
The $n$-\textit{cyclotomic Drinfeld--Kohno} Lie algebra $\mathfrak{t}_n^{\Gamma}$ is given by the following presentation:
\begin{enumerate}
\item It is generated by elements $t_{0i}$ for $1\leq i \leq n$ and by elements $t_{ij}^{\alpha}$ for $1 \leq i,j \leq n$, where $i \neq j$, and for $\alpha$ in $\Gamma$. 

\medskip

\item These generators are subject to the following relations:
\begin{enumerate}
\item $t_{ij}^{\alpha} = t_{ji}^{-\alpha}~.$
\item $[t_{0i},t_{jk}^{\alpha}] = 0$ and $[t_{ij}^{\alpha},t_{kl}^{\beta}] = 0~.$
\item $[t_{ij}^{\alpha}, t_{ik}^{\alpha + \beta} + t_{jk}^{\beta}]= 0~.$
\item $[t_{0i}, t_{0j} + \sum_{\alpha \in \Gamma} t_{ij}^{\alpha}] =0~.$
\item $[t_{0i} + t_{0j} + \sum_{\beta \in \Gamma} t_{ij}^{\beta},t_{ij}^{\alpha}] =0~.$
\end{enumerate}
for all distinct $i,j,k,l$ and for all $\alpha,\beta$ in $\Gamma$.
\end{enumerate}
\end{definition}

\begin{proposition}[\cite{calaque20moperadic}]\label{cyclotomic moperad structure on Kohno-Drinfeld}
The collection $(\mathfrak{t}_n^{\Gamma})_{n\geq0}$ can be endowed with the following moperad structure over the operad 
$(\mathfrak{t}_n)_{n\geq0}$ in the symmetric monoidal category $(\mathsf{Lie},\oplus,0)$:
\begin{enumerate}
\item The action of the symmetric group $\mathfrak{S}_n$ on $\mathfrak{t}_n^{\Gamma}$ is given by :
\[
\sigma \bullet t_{ij}^{\alpha} \coloneqq t_{\sigma(i)\sigma(j)}^{\alpha} ~~ \text{and} ~~\sigma \bullet t_{0i} \coloneqq t_{0\sigma(i)}~.
\]
\item The partial composition maps $\{\circ_p: \mathfrak{t}_n^{\Gamma} \oplus \mathfrak{t}_m \longrightarrow \mathfrak{t}_{n+m-1}^{\Gamma}\}$
of the right module structure are given by the following components (we assume that $i<j$):
\begin{eqnarray*}
\mathfrak{t}_m \ni t_{ij} & \longmapsto & t_{i+p-1 \hspace{1pt} j+p-1}^{0} \\[0.35cm]
\mathfrak{t}_n^{\Gamma} \ni t_{ij}^{\alpha} & \longmapsto & 
\begin{cases*}
t_{i+m-1\hspace{1pt}j+m-1}^{\alpha} & if~~p < i~.\\ 
{\displaystyle \sum_{k=i}^{i+m-1}t_{k\hspace{1pt}j+m-1}^{\alpha}} & if~~p = i~. \\
t_{i\hspace{1pt}j+m-1}^{\alpha}  & if~~ i < p < j~.\\
{\displaystyle \sum_{k=j}^{j+m-1}t_{i\hspace{1pt}k}^{\alpha}} & if~~ p = j~.\\
t_{i\hspace{1pt}j}^{\alpha}  & if~~ j < p~.
\end{cases*}\\[0.35cm]
\mathfrak{t}_n^{\Gamma} \ni t_{0j} & \longmapsto &
\begin{cases*}
t_{0\hspace{1pt}j+m-1} & if ~~p < j~.\\
{\displaystyle \sum_{k=j}^{j+m-1}t_{0\hspace{1pt}k} + \sum_{\gamma \in \Gamma}\sum_{j\leq l< r<j+m}t_{l\hspace{1pt}r}^{\gamma}} & if~~ p = j~.\\
t_{0j} & if~~ j < p~.
\end{cases*}
\end{eqnarray*}

%\[
%\begin{tikzcd}[column sep=1.5pc,row sep=0pc]
%t_{ij} ~~\text{in} ~~ \mathfrak{t}_m \arrow[r,mapsto]
%& t_{i+p-1 \hspace{1pt} j+p-1}^{0}~. \\
%t_{ij}^{\alpha} ~~\text{in} ~~ \mathfrak{t}_n^{\Gamma} \arrow[r,mapsto]
%& 
%\left\{\arraycolsep=1.4pt\def\arraystretch{1.7}\begin{array}{lr}
%       t_{i+m-1\hspace{1pt}j+m-1}^{\alpha}  ~~\mathrm{if} ~~p<i,j~.\\
%        t_{i\hspace{1pt}j+m-1}^{\alpha}+t_{i+1\hspace{1pt}j+m-1}^{\alpha}+\cdots+ t_{i+m-1\hspace{1pt}j+m-1}^{\alpha}  ~~\mathrm{if}~~ p = i~.\\
%        t_{i\hspace{1pt}j+m-1}^{\alpha}  ~~\mathrm{if}~~ i < p < j~.\\
%        \displaystyle \sum_{k=j}^{j+m-1}t_{i\hspace{1pt}k}^{\alpha} ~~\mathrm{if}~~ p = j~.\\
%        t_{i\hspace{1pt}j}^{\alpha}  ~~\mathrm{if}~~ p > i,j~.
%    \end{array}
%\right. \\
%t_{0j} ~~\text{in} ~~ \mathfrak{t}_n^{\Gamma} \arrow[r,mapsto]
%& \left\{
%    \arraycolsep=1.4pt\def\arraystretch{2.2}\begin{array}{lll}
%        t_{0\hspace{1pt}j+m-1} ~~\mathrm{if} ~~p<j~.\\
%        \displaystyle \sum_{k=j}^{j+m-1}t_{0\hspace{1pt}k} + \sum_{\gamma \in \Gamma}\sum_{j\leq l< r<j+m}t_{l\hspace{1pt}r}^{\gamma} ~~\mathrm{if}~~ p = j~.\\
%        t_{0j}~~\mathrm{if}~~ p > j~.
%    \end{array}
%\right. 
%\end{tikzcd}
%\]

\item The monoidal structure maps $\{\circ_0: \mathfrak{t}_n^{\Gamma} \oplus \mathfrak{t}_m^{\Gamma} \longrightarrow \mathfrak{t}_{n+m}^{\Gamma}\}$ 
are given by the following components:
\begin{eqnarray*}
\mathfrak{t}_m^{\Gamma} \ni t_{0k} & \longmapsto & t_{0k}~.\\
\mathfrak{t}_m^{\Gamma} \ni t_{kl}^{\alpha} & \longmapsto & t_{kl}^{\alpha}~.\\
\mathfrak{t}_n^{\Gamma}\ni t_{0i} & \longmapsto & {\displaystyle t_{0i} + \sum_{\gamma \in \Gamma}\sum_{r\geq 1}^{m} t_{r\hspace{1pt}i}^{\gamma}}~.\\
\mathfrak{t}_n^{\Gamma} \ni t_{ij}^{\alpha} & \longmapsto & t_{ij}^{\alpha}~.\\
\end{eqnarray*}
%\[
%\begin{tikzcd}[column sep=1.5pc,row sep=0pc]
%t_{0k} ~~\text{in} ~~ \mathfrak{t}_m^{\Gamma} \arrow[r,mapsto]
%& t_{0k}~. \\
%t_{kl}^{\alpha} ~~\text{in} ~~ \mathfrak{t}_m^{\Gamma} \arrow[r,mapsto]
%& t_{kl}^{\alpha}~. \\
%t_{ij}^{\alpha} ~~\text{in} ~~ \mathfrak{t}_n^{\Gamma}\arrow[r,mapsto]
%& t_{ij}^{\alpha}~. \\
%t_{0i} ~~\text{in} ~~ \mathfrak{t}_n^{\Gamma} \arrow[r,mapsto]
%& \displaystyle t_{0i} + \sum_{\gamma \in \Gamma}\sum_{r\geq 1}^{m} t_{r\hspace{1pt}i}^{\gamma}~.
%\end{tikzcd}
%\]
\end{enumerate}
It comes equipped with the following action of $\Gamma$: let $(\gamma_1,\cdots, \gamma_i, \cdots, \gamma_n)$ be in $\Gamma^n$, then its action 
on the generators of $\mathfrak{t}_n^{\Gamma}$ is given by
\[
\left\{
    \arraycolsep=1.4pt\def\arraystretch{1.7}\begin{array}{llll}
        \gamma_i \bullet t_{0j} = t_{0j} ~~\mathrm{for} ~~1 \leq k \leq n~,\\
        \gamma_i \bullet t_{jk}^{\alpha}  = t_{jk}^{\alpha} ~~\mathrm{if}~~ i \neq j,k~,\\
        \gamma_i \bullet t_{jk}^{\alpha}  = t_{jk}^{\alpha + \gamma_i} ~~\mathrm{if}~~ i = j ~,\\
        \gamma_i \bullet t_{jk}^{\alpha}  = t_{jk}^{\alpha - \gamma_i} ~~\mathrm{if}~~ i = k ~.
    \end{array}
\right. 
\]
\end{proposition}

\begin{proof}
We have already seen in the Introduction that the insertion-coproduct morphisms gives $t_\bullet$ the structure of a presheaf of graded Lie algebras 
on $\mathrm{Fin}_*$. Following \cite{Enriquez07} one can extend it to a presheaf of graded Lie algebras on $\mathrm{bFin}_{*,*}$ in the following way 
(without loss of generality, we allow ourselves to work with a skeleton of $\mathrm{bFin}_{*,*}$): 
\begin{itemize}
\item The Lie algebra associated with $\{*=0,1,\dots,n\}$ is $\mathfrak{t}_n$. 
\item The Lie algebra associated with $\{*,0,1,\dots,n\}$ is $\mathfrak{t}_n^\Gamma$. 
\item For every doubly pointed map $f:\{*,0,1,\dots,m\}\to \{*,0,1,\dots,n\}$, the corresponding map $(-)^f:\to \mathfrak{t}_n^\Gamma\to\mathfrak{t}_m^\Gamma$ 
is defined by 
\[
(t_{ij}^\alpha)^f:=\sum_{\substack{k\in f^{-1}(i) \\ l\in f^{-1}(j)}}t_{kl}^\alpha
\]
and 
\[
(t_{0i})^f:=\sum_{j\in f^{-1}(i)}t_{0j}+\sum_{\substack{j,k\in f^{-1}(i) \\ j<k}}\sum_{\gamma\in\Gamma}t_{jk}^\gamma
+\sum_{\substack{j\in f^{-1}(0)\backslash\{0\} \\ k\in f^{-1}(i)}}\sum_{\gamma\in\Gamma} t_{jk}^\gamma~.
\]
\item For every doubly pointed map $f:\{*,0,1,\dots,m\}\to \{*=0,1,\dots,n\}$, the corresponding map $(-)^f:\mathfrak{t}_n\to\mathfrak{t}_m^\Gamma$ 
is defined by
\[
(t_{ij})^f:=\sum_{\substack{k\in f^{-1}(i) \\ l\in f^{-1}(j)}}t_{kl}^0~.
\]
\item For every pointed map $f:\{*=0,1,\dots,m\}\to \{*=0,1,\dots,n\}$, the corresponding map $(-)^f:\mathfrak{t}_n\to\mathfrak{t}_m$ is the usual 
insertion-coproduct morphism from \cref{def: insertion-coproduct morphisms}. \hfill$\triangle$
\end{itemize} 

This defines a presheaf on $\mathrm{bFin}_{*,*}$, and therefore by Appendix \ref{Appendix B3: moperads}, a moperad structure in the category of Lie algebras with the coproduct. Like in Proposition \ref{prop: operad structure on Kohno-Drinfeld}, the images of theses maps commute and therefore factor through the product of Lie algebras.
\end{proof}

\begin{definition}[$N$-chord diagram moperad]
The $N$-\textit{chord diagram moperad} $\mathcal{CD}_0^\Gamma$ is the moperad over $\mathcal{CD}$ given by applying the universal enveloping algebra 
to the cyclotomic Drinfeld--Kohno moperad. 
\end{definition}

In each arity, the moperad $\mathcal{CD}_0^\Gamma$ is given by
\[
\mathcal{CD}_0^\Gamma(n) = \mathfrak{U}(\mathfrak{t}_n^{\Gamma})~, 
\]
where $\mathfrak{U}(\mathfrak{t}_n^{\Gamma})$ is the algebra of $N$-chord diagrams on $n+1$ strands. Following \cite{Brochier13}, one can represent the elements of this algebra 
as chord diagrams on $n+1$ strands, numbered from $0$ to $n$, with labels on the last $n$ strands: each strand outside the zeroth one is labeled with elements 
of $\mathbb{Z}/N\mathbb{Z}$ such that their sum is $0$ in $\mathbb{Z}/N\mathbb{Z}~.$ For instance, the element 
$t_{0i}.t_{jk}^{\alpha}.t_{jl}^{\beta}$ in $\mathfrak{U}(\mathfrak{t}_n^{\Gamma})$ can be represented as 

\[
\includegraphics[width=70mm,scale=1]{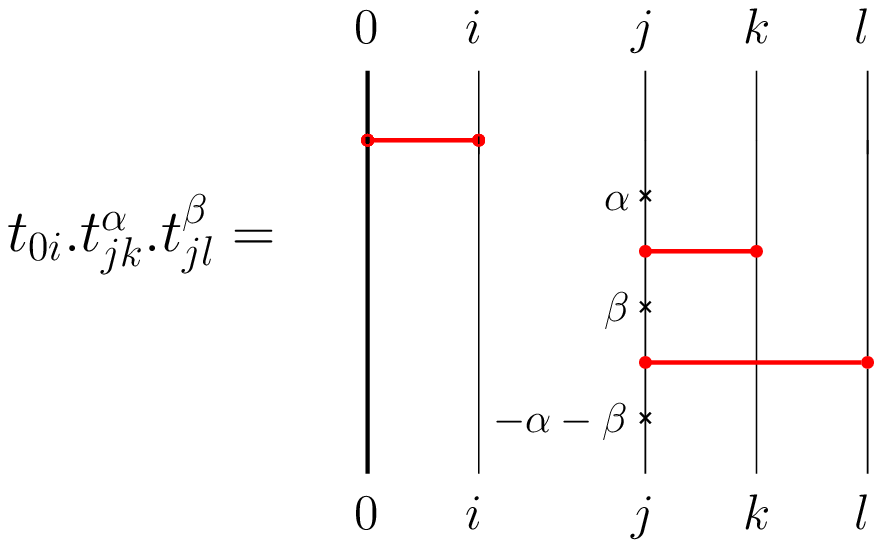}
\]
One can also represent all the relations that define the algebra $\mathfrak{U}(\mathfrak{t}_n^{\Gamma})$ in terms of equalities of such $N$-chord diagrams (see 
\cite{calaque20moperadic}). 

\medskip

We have that $\mathcal{CD}_0^\Gamma$ is a moperad over $\mathcal{CD}$ in the category $\mathsf{Cat}(\mathsf{Coalg}_\kk)$ of small categories enriched over counital cocommutative 
$\kk$-coalgebras, where both only have one object. In order to obtain a moperad over the operad $\pacd$, we need to add 
objects to $\mathcal{CD}_0^\Gamma$ first. Consider the moperad $\underline{\Gamma}$ over $u\mathcal{C}om$ in the category of sets. 
By applying the $\kk$-linear extension functor introduced in \cref{def: k-extension}, we obtain a moperad $\underline{\Gamma}_\kk$ over the terminal operad 
$u\mathcal{C}om$, this time in the category $\mathsf{Cat}(\mathsf{Coalg}_\kk)~.$ The category $\underline{\Gamma}_\kk(n)$ has $\Gamma^n$ a set of objects 
and $\kk$ as its cocommutative coalgebra of morphisms for any two objects.

\begin{definition}[Cyclotomic chord moperad]
The \textit{cyclotomic chord moperad} $\mathcal{CD}^\Gamma$ is defined as the Hadamard product 
\[
\mathcal{CD}^\Gamma \coloneqq \mathcal{CD}_0^\Gamma \otimes \underline{\Gamma}_\kk
\]
of moperads in the category $\mathsf{Cat}(\mathsf{Coalg}_\kk)~.$ It is naturally a moperad over $\mathcal{CD} \otimes u\mathcal{C}om \cong \mathcal{CD}~.$ It comes with a natural $\Gamma$-action given by the product of the $\Gamma$-actions on each component of the product. (We view $\Gamma$ acting on $\kk[\underline{\Gamma}]$ via translation on objects). 
\end{definition}

\begin{remark}\label{element L}
Let $(\alpha_1, \cdots, \alpha_n)$ and $(\beta_1,\cdots, \beta_n)$ be two objects in $\mathcal{CD}^\Gamma(n)$. The elements of the cocommutative coalgebra $\mathfrak{U}(\mathfrak{t}_n^{\Gamma})$ of morphisms between them can be represented as chord diagrams on $n+1$ strands, where the last $n$ strands are labeled by elements of $\mathbb{Z}/N\mathbb{Z}$ in the following way: the sum of the labels on the $i$-th strand must be equal to $\beta_i - \alpha_i$, for all $1 \leq i \leq n$. For example, the following chord diagram 
\[
\includegraphics[width=25mm,scale=1]{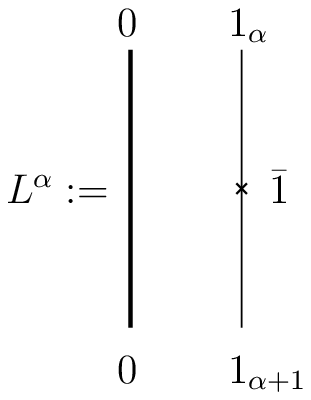}
\]
defines a morphism between $\alpha$ and $\alpha+\overline{1}$, for every $\alpha$ in $\mathcal{CD}^\Gamma(1) = \mathbb{Z}/N\mathbb{Z}$.
\end{remark}

Now recall the set-theoretical moperad $\mathcal{P}a_0$ over $\mathcal{P}a$ given in Example \ref{example: moperad Pa0}. By applying the $\kk$-linear extension functor to both, we obtain a moperad $(\mathcal{P}a_0)_\kk$ over $\mathcal{P}a_\kk$, this time in the category $\mathsf{Cat}(\mathsf{Coalg}_\kk)~.$ 

\begin{definition}[Cyclotomic parenthesized chord moperad]
The \textit{cyclotomic parenthesized chord moperad} $\pacd^\Gamma$ is defined by the Hadamard product 
\[
\pacd^\Gamma \coloneqq (\mathcal{P}a_0)_\kk \otimes \mathcal{CD}^\Gamma
\]
of moperads in the category $\mathsf{Cat}(\mathsf{Coalg}_\kk)~.$ It is naturally a moperad over $\mathcal{P}a_\kk \otimes \mathcal{CD} = \pacd$. The $\Gamma$-action on $\mathcal{CD}^\Gamma$ endows this product with a $\Gamma$-action as well.
\end{definition}

\begin{example}\label{element beta}
It follows from this definition that objects in $\pacd^\Gamma(n)$ are given by maximal parenthesizations on words 
$0 \hspace{1pt}(\sigma_1)_{\alpha_1} \cdots (\sigma_n)_{\alpha_n}$, where $\sigma$ is a permutation in $\mathfrak{S}_n$ and $(\alpha_1,\cdots,\alpha_n)$ 
is an element of $\Gamma^n$. Morphisms between two maximally parenthesized words $0 \hspace{1pt}(\sigma_1)_{\alpha_1} \cdots (\sigma_n)_{\alpha_n}$ 
and $0 \hspace{1pt}(\tau_1)_{\gamma_1} \cdots \allowbreak (\tau_n)_{\gamma_n}$ are given by chord diagrams on the unique strands that go from one permutation 
to the other, where each strand is labeled by elements of $\Gamma$, such that the sums of the labels of the $i$-th strand is equal to $\gamma_i - \alpha_i$, for all $1 \leq i \leq n$. Here are two distinguished morphisms of $\pacd^\Gamma$:
\[
\includegraphics[width=65mm,scale=1]{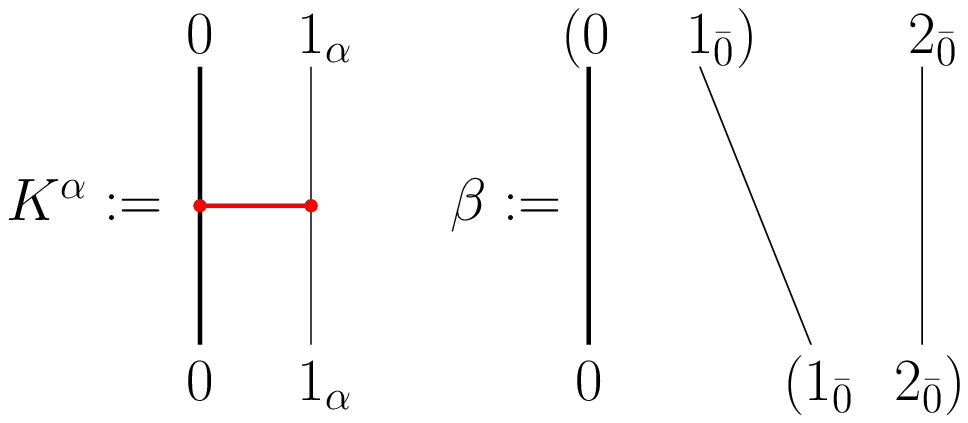}
\]
Here $K^\alpha$ is the endomorphism of $\alpha$ in $\pacd^\Gamma(1)$ which corresponds to $t_{01}$ and $\beta$ is the isomorphism between 
$(0\hspace{1pt} 1_{\bar{0} } )\hspace{1pt} 2_{\bar{0}}$ and $0 \hspace{1pt} (1_{\bar{0}} \hspace{1pt} 2_{\bar{0}})$ in $\pacd^\Gamma(2)$ 
which corresponds to the trivial chord diagram. \hfill$\triangle$
\end{example}

\begin{remark}\label{rmk: elements in pacd Gamma}
The elements $\beta$, $L^\alpha$ and $K^\alpha$ for all $\alpha$ in $\Gamma$ generate $\pacd^\Gamma$ as a moperad over $\pacd$ with a $\Gamma$-action. 
\end{remark}

Notice that the relations defining the cyclotomic Drinfeld--Kohno Lie algebras $\mathfrak{t}_n^{\Gamma}$ are homogeneous, therefore they inherit a weight 
filtration with generators in weight one. See \cref{rmk: weight filtration} for the standard case. For all $n \geq 0$, $\pacd^\Gamma(n)$ forms a Hopf groupoid, 
therefore if we want to obtain a pro-unipotent groupoid we ought to complete it. 

\medskip

We define $\widehat{\pacd}^\Gamma$ as the moperad over $\widehat{\pacd}$ in the category of categories enriched over complete cocommutative coalgebras 
obtained by completing each hom-set of $\pacd^\Gamma(n)$ with respect to the aforementioned weight filtration. It is a moperad in complete Hopf groupoids.

\medskip

We obtain a moperad $\mathrm{Grp}(\widehat{\pacd}^\Gamma)$ over $\mathrm{Grp}(\widehat{\pacd})$ by taking the group-like elements functor to both. This gives a moperad in the category of pro-unipotent $\kk$-groupoids. This will be \textit{our second cyclotomic player}. 

\subsection{Cyclotomic associators and Grothendieck--Teichmüller groups}\label{ssec-Cyclo associator and gts}
We are again able to present a purely (m)operadic version of cyclotomic associators. 

\begin{definition}[Cyclotomic associators] \label{Def: cyclotomic associator}
The set of \textit{cyclotomic associators} is given by 
\[
\mathrm{Assoc}^\Gamma(\kk) \coloneqq \mathrm{Iso}^+_{\mathsf{Mop}(\mathsf{p.u}\text{-}\mathsf{Grpd}_{\kk})}\left(\left(\widehat{\pab}^\Gamma(\kk),\widehat{\pab}(\kk)\right),\left(\mathrm{Grp}(\widehat{\pacd}^\Gamma),\mathrm{Grp}(\widehat{\pacd})\right)\right)^\Gamma~.
\]
An element of $\mathrm{Assoc}^\Gamma(\kk)$ amounts to the data of:

\medskip

\begin{enumerate}
\item an isomorphism $F$ between $\widehat{\pab}(\kk)$ and $\mathrm{Grp}(\widehat{\pacd})$ which is the identity on objects (an associator)~,

\medskip

\item an isomorphism $G$ between $\widehat{\pab}^\Gamma(\kk)$ and $\mathrm{Grp}(\widehat{\pacd}^\Gamma)$ of moperads which is $\Gamma$-equivariant 
and which is the identity on objects, lying above the isomorphism $F$. 
\end{enumerate}
\end{definition}

This definition entails that the set of cyclotomic associators is a bitorsor over the automorphisms groups of each moperad that are the identity morphism on objects. 
This allows us to define the cyclotomic versions of the Grothendieck--Teichmüller group and of the graded Grothendieck--Teichmüller group.

\begin{definition}[Cyclotomic Grothendieck--Teichmüller group]
The \textit{cyclotomic Grothendieck--Teichmüller group} over $\kk$ is given by 
\[
\mathrm{GT}^\Gamma(\kk) \coloneqq \mathrm{Aut}_{\mathsf{Mop}(\mathsf{p.u}\text{-}\mathsf{Grpd}_{\kk})}^+\left(\widehat{\pab}^\Gamma(\kk),\widehat{\pab}(\kk)\right)^\Gamma,
\]
that is, the group of $\Gamma$-equivariant automorphisms of the moperads of $\widehat{\pab}^\Gamma(\kk)$ which are the identity on objects.
\end{definition}

\begin{definition}[Cyclotomic graded Grothendieck--Teichmüller group]
The \textit{cyclotomic graded Grothendieck--Teichmüller group} over $\kk$ is given by 
\[
\mathrm{GRT}^\Gamma(\kk) \coloneqq \mathrm{Aut}_{\mathsf{MopCat(Coalg_\kk)}}^+\left(\widehat{\pacd}^\Gamma,\widehat{\pacd}\right)^\Gamma,
\]
that is, the group of $\Gamma$-equivariant automorphisms of the moperads of $\widehat{\pacd}^\Gamma$ which are the identity on objects.
\end{definition}

\subsection{More concrete descriptions}\label{ssec-more concrete description of cyclos}

\begin{theorem}[{\cite[Theorem 4.6]{calaque20moperadic}}]\label{thm: presentation moperad pab Gamma}
The moperad $\pab^\Gamma$, as a moperad over $\pab$ in groupoids endowed with a $\Gamma$-action, and having $\mathcal{P}a_0$ as its moperad of 
objects, admits the following presentation. The generating morphisms are 

\begin{enumerate}
\item The \textit{loop} $E$ around the frozen zeroth strand, pictorially given by
\[
\includegraphics[width=50mm,scale=1]{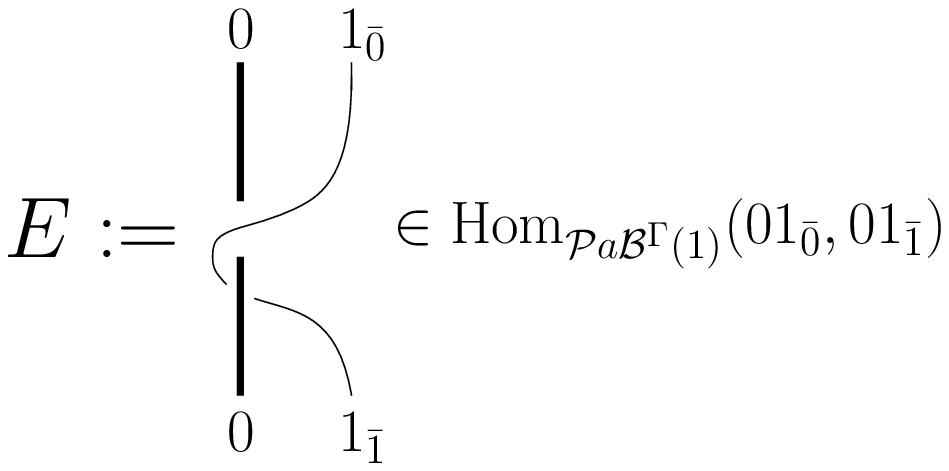}
\]
\item The \textit{associator} $\Psi$ which pictorially is given by the following braid
\[
\includegraphics[width=70mm,scale=1]{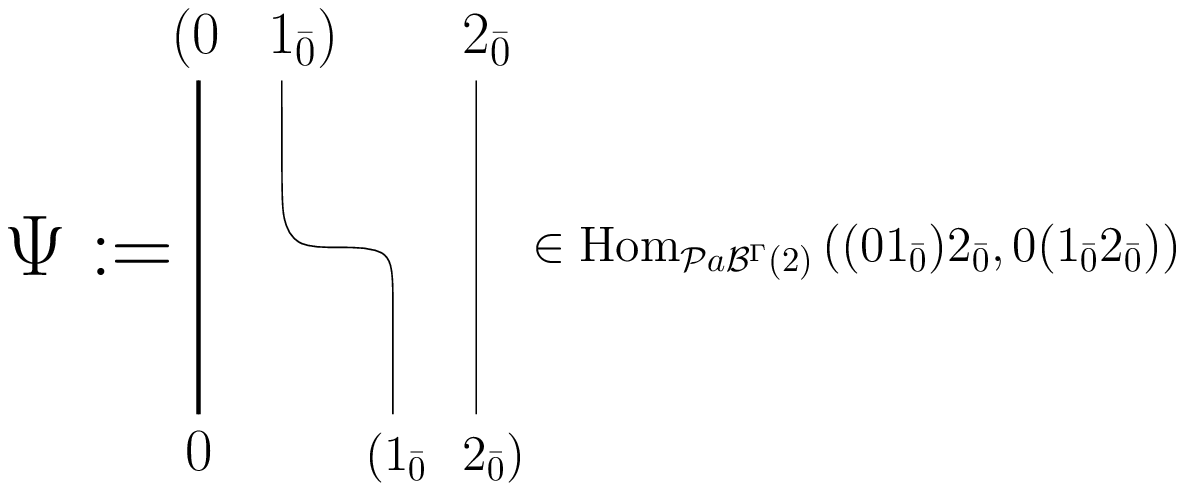}
\]
\end{enumerate}

They satisfy the following relations

\begin{enumerate}
\item Recall that $\pab(0)$ is the trivial groupoid $\{*\}$. The \textit{unital relation} states that
\[
\Psi \circ_1 \mathrm{id}_{\{*\}} = \mathrm{id}_{01_{\bar{0}}}~. 
\]
Notice that it implies that $\Psi \circ_2 \mathrm{id}_{\{*\}} = \mathrm{id}_{01_{\bar{0}}}.$

\item The \textit{mixed pentagon relation} amounts to the commutativity of the following diagram
\[
\begin{tikzcd}[column sep=2.5pc,row sep=2pc]
&((0 1_{\bar{0}}) 2_{\bar{0}}) 3_{\bar{0}} \arrow[dl,"\mathrm{id}_{0 1_{\bar{0}}}~\circ_0~\Psi",swap] \arrow[rd,"\Psi~\circ_0~\mathrm{id}_{0 1_{\bar{0}}}"]
& \\
(0 (1_{\bar{0}} 2_{\bar{0}})) 3_{\bar{0}} \arrow[d,"\Psi~\circ_1~\mathrm{id}_{12}",swap]
&
&(0 1_{\bar{0}})( 2_{\bar{0}} 3_{\bar{0}}) \arrow[d,"\Psi~\circ_2~\mathrm{id}_{12}"]\\
0 ((1_{\bar{0}} 2_{\bar{0}}) 3_{\bar{0}}) \arrow[rr,"\mathrm{id}_{0 1_{\bar{0}}}~\circ_1~\Phi"] 
&
&0 (1_{\bar{0}} (2_{\bar{0}} 3_{\bar{0}}))
\end{tikzcd}
\]

\item The \textit{twisted ribbon relation} amounts to the commutativity of the following diagram
\[
\begin{tikzcd}[column sep=2.5pc,row sep=2pc]
&(0 1_{\bar{0}}) 2_{\bar{0}} \arrow[dl,"\Psi",swap] \arrow[rd,"\mathrm{id}_{0 1_{\bar{0}}}~\circ_0~E"]
& \\
0 (1_{\bar{0}} 2_{\bar{0}}) \arrow[d,"\mathrm{id}_{0 1_{\bar{0}}}~\circ_1~E",swap]
&
&(0 1_{\bar{1}})2_{\bar{0}} \arrow[d,"E ~\circ_0~\mathrm{id}_{0 1_{\bar{0}}}"]\\
0 (1_{\bar{1}} 2_{\bar{1}}) \arrow[rr,"(\bar{1}\, \bar{1})~ \ast~\Psi^{-1}"] 
&
&(0 1_{\bar{1}})2_{\bar{1}}
\end{tikzcd}
\]

\item The \textit{twisted octagon relation} amounts to the commutativity of the following diagram
\[
\begin{tikzcd}[column sep=5pc,row sep=3pc]
&(0 1_{\bar{0}}) 2_{\bar{0}} \arrow[r,"E~ \circ_0~ \mathrm{id}_{0 1_{\bar{0}}}"] \arrow[dl,"\Psi",swap]
&(0 1_{\bar{0}}) 2_{\bar{1}} 
& \\
0 (1_{\bar{0}} 2_{\bar{0}}) \arrow[d,"\mathrm{id}_{0 1_{\bar{0}}} ~\circ_1 ~R",swap]
&
&
&0 (1_{\bar{0}} 2_{\bar{1}}) \arrow[ul,"(\bar{0}\, \bar{1})~ \ast~ \Psi^{-1}",swap] \\
0 (2_{\bar{0}} 1_{\bar{0}}) \arrow[rd,swap,"(21)~\bullet ~\Psi^{-1}"]
&
&
&0 (2_{\bar{1}} 1_{\bar{0}}) \arrow[u,"(\bar{1} \, \bar{0}) ~ \ast ~(\mathrm{id}_{0 1_{\bar{0}}} ~\circ_1 ~R^{-1})"] \\
&(0 2_{\bar{0}}) 1_{\bar{0}} \arrow[r,"(21)~ \bullet ~(\mathrm{id}_{0 1_{\bar{0}}} ~ \circ_0 ~E)"]
&(0 2_{\bar{1}}) 1_{\bar{0}} \arrow[ru,"(\bar{1} \, \bar{0}) ~\ast ~((21) ~\bullet ~\Psi)",swap]
&
\end{tikzcd}
\]

\end{enumerate} 
\end{theorem}

\begin{remark}
For the sake of completeness, let us rewrite the above relations in the ``insertion-coproduct'' notation:

\begin{enumerate}
\item The unital relation states that $\Psi^{0,\emptyset,1_{\bar{0}}} = \mathrm{id}$.

\item The mixed pentagon relation can be written as
\[
\Psi^{01_{\bar{0}},2_{\bar{0}},3_{\bar{0}}}~\Psi^{0,1_{\bar{0}},2_{\bar{0}}3_{\bar{0}}}= \Psi^{0,1_{\bar{0}},2_{\bar{0}}}~\Psi^{0,1_{\bar{0}}2_{\bar{0}},3_{\bar{0}}}~\Phi^{1_{\bar{0}},2_{\bar{0}},3_{\bar{0}}}\Psi^{0,1_{\bar{0}},2_{\bar{0}}}~.
\]

\item The twisted ribbon relation can be written as 
\[
\Psi^{0,1_{\bar{0}},2_{\bar{0}}}~E^{0,1_{\bar{0}}2_{\bar{0}}}~(\Psi^{0,1_{\bar{1}},2_{\bar{1}}})^{-1}= E^{0,1_{\bar{0}}}~E^{01_{\bar{1}},2_{\bar{0}}}~.
\]

\item The twisted octagon relation can be written as 
\[
E^{01_{\bar{1}},2_{\bar{0}}} = \Psi^{0,1_{\bar{0}},2_{\bar{0}}}~R^{1_{\bar{0}},2_{\bar{0}}}~(\Psi^{0,2_{\bar{0}},1_{\bar{0}}})^{-1}~E^{0,2_{\bar{0}}}~\Psi^{0,2_{\bar{1}},1_{\bar{0}}}~R^{2_{\bar{1}},1_{\bar{0}}}~(\Psi^{0,1_{\bar{0}},2_{\bar{1}}})^{-1}~.
\]
\end{enumerate}
\end{remark}

\subsubsection*{Relative Malcev completions}

Using this presentation of the moperad $\pab^\Gamma$, our goal is going to be to describe cyclotomic associators. 
Recall that $\Gamma = \mathbb{Z}/N\mathbb{Z}~.$ First consider the free group $\mathbb{F}_2$ generated by two elements $x$ and $y$. There is a group morphism

\[
\left\{
\begin{tikzcd}[column sep=3.5pc,row sep=0pc]
\mathbb{F}_2 \arrow[r,"p",twoheadrightarrow]
&\mathbb{Z} \\
x \arrow[r,mapsto]
&1 \\
y \arrow[r,mapsto]
&0 \\
\end{tikzcd}\right.
\]
\vspace{0.1pc}

This gives a morphism $p_N: \mathbb{F}_2 \twoheadrightarrow \mathbb{Z}/N\mathbb{Z}$ by composing $p$ with the obvious projection.

\begin{lemma}
There is an isomorphism of groups 

\[
\mathrm{Ker}(p_N) \cong \mathbb{F}_{N+1}~,
\]
\vspace{0.1pc}

where $\mathbb{F}_{N+1}$ is a free group generated by $x^N$ and $x^a y x^{-a}$ for $a = 0, \cdots, N -1~.$
\end{lemma}

This gives an exact sequence of groups:

\[
\begin{tikzcd}
1 \arrow[r]
&\mathbb{F}_{N+1} \arrow[r,hookrightarrow]
&\mathbb{F}_{2}  \arrow[r,"p_N",twoheadrightarrow]
&\mathbb{Z}/N\mathbb{Z} \arrow[r]
&1~.
\end{tikzcd}
\]

Let $f: G \twoheadrightarrow H$ be a surjective morphism of groups, and suppose that $\mathrm{Ker}(f)$ is finitely generated as a group. 

\medskip

The \textit{relative} $\kk$\textit{-Malcev completion} $\widehat{f}: \widehat{G}(\kk,f) \twoheadrightarrow H$ of the morphism $f$ can be defined as the initial $\kk$-\textit{pro-unipotent group} whose $\kk$-points lie above $H$. See \cite[Section 1.1]{Enriquez07} for more details or \cite{Hain98} for the original construction. One can show that it fits into the following exact sequence: 

\[
\begin{tikzcd}
1 \arrow[r]
&\widehat{\mathrm{Ker}(f)}(\kk) \arrow[r,hookrightarrow]
&\widehat{G}(\kk,f)  \arrow[r,"\widehat{f}",twoheadrightarrow]
&H \arrow[r]
&1~.
\end{tikzcd}
\]

where $\widehat{\mathrm{Ker}(f)}(\kk)$ is the classical $\kk$-pro-unipotent completion. 

\begin{remark}
There is always a monomorphism of groups $\widehat{G}(\kk,f) \hookrightarrow \widehat{G}(\kk)~.$
\end{remark}

\begin{lemma}\label{lemma: semi direct relative completion}
There is an isomorphism of groups 

\[
\widehat{\mathbb{F}_{2}}(\kk,p_N) \cong \widehat{\mathbb{F}_{N+1}}(\kk) \ltimes \mathbb{Z}/N\mathbb{Z}
\]
\vspace{0.1pc}

where the semi-direct product on the right is given by the following action 

\[
\alpha \ast x^N = 0 \quad \text{and} \quad \alpha \ast x^a y x^{-a} =  x^{a + \alpha} y x^{-a -\alpha}~,
\]
\vspace{0.1pc}

for $\alpha$ in $\mathbb{Z}/N\mathbb{Z}~.$
\end{lemma}

\begin{proof}
The exact sequence 

\[
\begin{tikzcd}
1 \arrow[r]
&\widehat{\mathbb{F}_{N+1}}(\kk) \arrow[r,hookrightarrow]
&\widehat{\mathbb{F}_{2}}(\kk,p_N)  \arrow[r,"\widehat{p_N}",twoheadrightarrow]
&\mathbb{Z}/N\mathbb{Z} \arrow[r]
&1~.
\end{tikzcd}
\]

admits a splitting since the group $\mathbb{Z}/N\mathbb{Z}$ is finite. One can check the formula by hand.  
\end{proof}

\subsubsection*{Cyclotomic braid groups}

Recall that the \textit{cyclotomic pure braid groups} $\mathrm{PB}_n^\Gamma$ are given by the kernel of the following morphism

\[
\left\{
\begin{tikzcd}[column sep=3.5pc,row sep=0pc]
\mathrm{PB}_n^1 \arrow[r,"c_N",twoheadrightarrow]
&\Gamma^n \\
x_{0j} \arrow[r,mapsto]
&(\overline{0},\cdots,\overline{0},\overline{1},\overline{0},\cdots, \overline{0}) \\
x_{ij} \arrow[r,mapsto]
&(\overline{0},\cdots,\overline{0}) \\
\end{tikzcd}\right.
\]
\vspace{0.1pc}

where $(\overline{0},\cdots,\overline{0},\overline{1},\overline{0},\cdots, \overline{0})$ has $\overline{1}$ in the $j$-th position, and where $\mathrm{PB}_n^1$ 
denotes the pure braid group with one fixed strand and $n$ free strands. We therefore have a morphism of short exact sequences 
\[
\begin{tikzcd}
0 \arrow[r]
&\mathbb{Z} \arrow[r,hookrightarrow]\arrow[d,twoheadrightarrow]
&\mathrm{PB}_2^1  \arrow[r,twoheadrightarrow]\arrow[d,twoheadrightarrow,"c_N"]
& \mathbb{F}_2 \arrow[r]\arrow[d,twoheadrightarrow,"p_N"]
&1 \\
0 \arrow[r]
& \Gamma \arrow[r,hookrightarrow]
& \Gamma^2 \arrow[r,twoheadrightarrow]
& \Gamma \arrow[r]
& 0~.
\end{tikzcd}
\]
Recall the splitting\footnote{Up to a shift of the indices/strand labels (that corresponds to the canonical identification 
$\mathrm{PB}_2^1=\mathrm{PB}_3$) this is the splitting from Lemma \ref{lemma: PB_3 iso au produit de libres}. } 
$\mathbb{F}_2 \longrightarrow \mathrm{PB}_2^1$ given by $x\mapsto x_{01}$ and $y\mapsto x_{12}$ of the upper short exact sequence, 
that lifts the splitting of the lower one given by $\Gamma\ni\alpha\mapsto (\alpha,\overline{0})$. 
Passing to relative Malcev completions, we get a splitting $\widehat{\mathbb{F}_{2}}(\kk,p_N)  \longrightarrow \widehat{\mathrm{PB}_2^1}(\kk,c_N)$
of the short exact sequence 
\[
\begin{tikzcd}
0 \arrow[r]
& {\kk} \arrow[r,hookrightarrow]
& \widehat{\mathrm{PB}_2^1}(\kk,c_N)  \arrow[r,twoheadrightarrow]
& \widehat{\mathbb{F}_{2}}(\kk,p_N) \arrow[r]
& 1~,
\end{tikzcd}
\]
and, taking kernels, a splitting $\widehat{\mathbb{F}_{N+1}}(\kk) \longrightarrow \widehat{\mathrm{PB}_2^\Gamma}(\kk)$ of the short exact sequence 
\[
\begin{tikzcd}
0 \arrow[r]
& {\kk} \arrow[r,hookrightarrow]
& \widehat{\mathrm{PB}_2^\Gamma}(\kk)  \arrow[r,twoheadrightarrow]
& \widehat{\mathbb{F}_{N+1}}(\kk) \arrow[r]
& 1~.
\end{tikzcd}
\]
It can be checked that this induces an isomorphism of groups 
\[
\widehat{\mathrm{PB}_2^\Gamma}(\kk) \cong \kk \times \widehat{\mathbb{F}_{N+1}}(\kk) ~.
\]
Similarly, one can check that the center of $\mathfrak{t}_2^\Gamma$ is a one dimensional vector space generated by $t_{01} + t_{02} + \sum_{\alpha \in \Gamma} t_{12}^{\alpha}$, and that furthermore

\[
\mathfrak{t}_2^\Gamma \cong \frac{\mathcal{L}ie(t_{01},t_{02},t_{12}^\alpha ~|~ \alpha = 0, \cdots, N-1)}{(t_{01} + t_{02} + \sum_{\alpha \in \Gamma} t_{12}^{\alpha} ~~ \text{is central})}~.
\]
\vspace{0.1pc}

There is a canonical isomorphism of Lie algebras 
\[
\mathfrak{t}_2^\Gamma \cong \kk.c \oplus \mathfrak{f}_{N+1}~,
\]
where $c = t_{01} + t_{02} + \sum_{\alpha \in \Gamma} t_{12}^{\alpha}$ and where $\mathfrak{f}_{N+1}$ is generated by $ t_{01},t_{12}^\alpha$, for $\alpha = 0, \cdots, N-1$.

\medskip

Therefore an element in $\mathrm{exp}(\widehat{\mathfrak{t}}_2^{~\Gamma})$ can by written as a pair

\[
\left(\Xi(c),\Psi(t_{01},t_{12}^{\overline{0}}, \cdots, t_{12}^{\overline{N-1}})\right) \quad \text{in} \quad \kk \times \kk \langle \langle t_{01},t_{02},t_{12}^\alpha ~|~ \alpha = 0, \cdots, N-1 \rangle \rangle~,
\]
\vspace{0.1pc}

where $\Psi$ is a non-commutative formal power series in variables $t_{01}$ and $t_{12}^\alpha$ for $\alpha = 0, \cdots, N-1$.

\begin{theorem}[{\cite[Theorem 5.5]{calaque20moperadic}}]
There is a bijection between the set of cyclotomic associators $\mathrm{Assos}^\Gamma$ and the set of triples $(\lambda,\Phi,\Psi)$ where 

\begin{enumerate}
\item $(\lambda,\Phi)$ is a Drinfeld associator;

\medskip

\item $\Psi$ is a non-commutative formal series $\Psi(t_{01} ,t_{12}^{\overline{0}}, \cdots, t_{12}^{\overline{N-1}})$ such that 

\medskip

\begin{enumerate}
\item the following equation is satisfied in $\mathrm{exp}(\widehat{\mathfrak{t}}_3^{~\Gamma})$:
\[
\Psi\left(t_{01} + \sum_{\alpha \in \Gamma} t_{12}^{\alpha} + \sum_{\alpha \in \Gamma} t_{13}^{\alpha},t_{12}^{\overline{0}}, \cdots, t_{12}^{\overline{N-1}}\right)~ \Psi\left(t_{01},t_{12}^{\bar{0}} + t_{13}^{\overline{0}} , \cdots, t_{12}^{\overline{N-1}} +  t_{13}^{\overline{N-1}}\right) = 
\]
\[
\Psi\left(t_{01},t_{12}^{\overline{0}}, \cdots, t_{12}^{\overline{N-1}}\right)~ \Psi\left(t_{01} + t_{02} + \sum_{\alpha \in \Gamma} t_{12}^{\alpha},t_{12}^{\overline{0}}+ t_{13}^{\overline{0}},\cdots, t_{12}^{-\overline{N+1}}+t_{13}^{-\overline{N+1}}\right) ~ \Phi\left(t_{12},t_{23}\right)
\]

\medskip

\item the following equation is satisfied in $\mathrm{exp}(\widehat{\mathfrak{t}}_2^{~\Gamma})$:
\[
e^{\frac{\lambda}{N} t_{01}} ~ \Psi\left(t_{01},t_{12}^{\overline{0}}, \cdots, t_{12}^{\overline{N-1}}\right) ~ e^{\frac{\lambda}{2} t_{12}^{\bar{0}}} ~ \Psi\left(t_{01} ,t_{12}^{\overline{0}}, \cdots, t_{12}^{-\alpha}, \cdots, t_{12}^{-\overline{N+1}}\right)^{-1} ~ e^{\frac{\lambda}{N} t_{02}}
\]
\[
(\bar{0},\bar{1}) \ast \left(\Psi\left(t_{01},t_{12}^{\overline{0}}, \cdots, t_{12}^{-\alpha}, \cdots, t_{12}^{-\overline{N+1}}\right) ~ e^{\frac{\lambda}{2} t_{12}^{\overline{0}}}~ \Psi\left(t_{01},t_{12}^{\overline{0}}, \cdots, t_{12}^{\overline{N-1}}\right)^{-1} \right) = 1~.
\]
\end{enumerate}
\end{enumerate} 
\end{theorem}

\begin{proof}[Sketch of proof]
The data of an isomorphism $F: \widehat{\pab}(\kk) \qi \mathrm{Grp}(\widehat{\pacd})$ which is the identity on objects is equivalent to the data of a pair $(\lambda_1,\Phi)$ which is a Drinfeld associator.

\medskip

Let $G: \widehat{\pab}^\Gamma(\kk) \qi \mathrm{Grp}(\widehat{\pacd}^\Gamma)$ be an $\Gamma$-equivariant isomorphism of moperads above an associator $F$, which is also the identity on objects. The morphism $G$ is completely determined by the image of $E$ and $\Psi$. 

\medskip

The image of $E$ by $G$ has to be of the form 
\[
G(E) = e^{\lambda_2 t_{01}}L^{\bar{1}} \quad \text{in} \quad \mathrm{Hom}_{\pacd^\Gamma(1)}(01_{\bar{0}},01_{\bar{1}}) ~,
\]
where $e^{\lambda_2 t_{01}}$ is in $\mathrm{exp}(\widehat{\mathfrak{t}}_1^{~\Gamma}) \cong \kk$ and where is the $L^{\bar{1}}$ element described in \cref{element L}. One can show that $\lambda_2 = \lambda_1/N $, see \cite[Lemma 5.7]{calaque20moperadic}.

\medskip

The image of $\Psi$ by $G$ has to be of the form
\[
G(\Psi) = g(\Psi)\beta~, \quad \mathrm{in} \quad \mathrm{Hom}_{\pacd^\Gamma(2)}((01_{\bar{0}})2_{\bar{0}},0(1_{\bar{0}}2_{\bar{0}}))~,
\]
where $g(\Psi)$ is in $\mathrm{exp}(\widehat{\mathfrak{t}}_2^{~\Gamma})$ and where $\beta$ is the element of $\widehat{\pacd}^\Gamma$ described in \cref{element beta}. The isomorphism 
\[
\mathfrak{t}_2^\Gamma \cong \kk.c \oplus \mathfrak{f}_{N+1}~,
\]
allows us to canonically decompose 
\[
g(\Psi) = \left(\Xi(c),\Psi(t_{01} ,t_{12}^{\overline{0}}, \cdots, t_{12}^{\overline{N-1}})\right) \quad \text{in} \quad \mathrm{exp}(\widehat{\mathfrak{t}}_2^{~\Gamma}) \cong \kk \times \widehat{\mathbb{F}_{N+1}}(\kk) ~,
\]
where $c = t_{01} + t_{02} + \sum_{\alpha \in \Gamma} t_{12}^{\alpha}$. The morphism 
\[
- \circ_1 \mathrm{id}_{\{*\}}: \mathrm{Hom}_{\pacd^\Gamma(2)}((01_{\bar{0}})2_{\bar{0}},0(1_{\bar{0}}2_{\bar{0}})) \longrightarrow \mathrm{Hom}_{\pacd^\Gamma(1)}(01_{\bar{0}},01_{\bar{0}})
\]
is the one determined by the morphism $\mathfrak{t}_2^{~\Gamma} \twoheadrightarrow \mathfrak{t}_1^{~\Gamma}$ which sends $t_{01} ,t_{12}^{\alpha}$ to zero and $t_{02}$ to $t_{01}$. Its pre-composition with the isomorphism $\kk.c \oplus \mathfrak{f}_{N+1} \cong \mathfrak{t}_2^\Gamma$ is given by sending $t_{01} ,t_{12}^{\alpha}$ to zero and $c$ to $t_{01}$. The unital relation 
\[
\Psi \circ_1 \mathrm{id}_{\{*\}} = \mathrm{id}_{01_{\bar{0}}}
\]
therefore gives that $\Xi(c) = 1$, and thus that the element $g(\Phi)$ is simply given by the formal power series $\Psi(t_{01} ,t_{12}^{\overline{0}}, \cdots, t_{12}^{\overline{N-1}})$ in $\mathrm{exp}(\widehat{\mathfrak{f}}_{N+1})$. 

\medskip

Finally, one shows by direct computations, starting from the above presentation of $\pab^\Gamma$, that the relations in this presentation impose the relations in the theorem.
\end{proof}

\makeatletter
\newcommand{\ostar}{\mathbin{\mathpalette\make@circled\star}}
\newcommand{\make@circled}[2]{%
  \ooalign{$\m@th#1\smallbigcirc{#1}$\cr\hidewidth$\m@th#1#2$\hidewidth\cr}%
}
\newcommand{\smallbigcirc}[1]{%
  \vcenter{\hbox{\scalebox{0.77778}{$\m@th#1\bigcirc$}}}%
}
\makeatother

%\begin{remark}
%The family of Lie algebras $(\mathfrak{t}_n^\Gamma)_{n\geq0}$ can be thought as the controlling the infinitesimal deformations of a universal braided module over a universal symmetric monoidal category. Therefore analogue heuristics as explained in Problem \ref{Problem 2} also hold in this case. In particular, since such a braided module should satisfy the \textit{twisted pentagon equation}

%\[
%\begin{tikzcd}[column sep=3pc,row sep=2.5pc]
%&((0 \ostar 1) \otimes 2) \otimes 3 \arrow[ld,"\Psi^{01,2,3}",swap] \arrow[rd,"\Psi^{0,1,2}"]
%& \\
%(0 \ostar 1) \otimes (2 \otimes 3) \arrow[d,"\Psi^{0,1,23}",swap]
%&
%&(0 \ostar (1 \otimes 2)) \otimes 3 \arrow[d,"\Psi^{0,12,3}"]\\
%0 \ostar (1 \otimes (2 \otimes 3)) 
%&
%&0 \ostar ((1 \otimes 2)) \otimes 3) \arrow[ll,"\Phi^{1,2,3}"]
%\end{tikzcd}
%\]

%where $\ostar$ is the tensoring given by this braided module structure. This imposes the following equation onto any cyclotomic associator

%\[
%\Psi^{0,1,2}.\Psi^{0,12,3}.\Phi^{1,2,3} = \Psi^{0,1,23}. \Psi^{01,2,3}~.
%\]

%Translating this equation using the \textit{cyclotomic insertion-coproduct morphisms} given by the moperad structure of Definition \ref{cyclotomic moperad structure on Kohno-Drinfeld} gives the twisted pentagon equation in the theorem above. Analogue correspondences can be made for the other equations that appear.
%\end{remark}

\begin{theorem}[{\cite[Section 2.2]{Enriquez07}}]
Let $\kk$ be a field of characteristic zero and let $\Gamma = \mathbb{Z}/N\mathbb{Z}$. Then the set of cyclotomic associators $\mathrm{Assoc}^\Gamma(\kk)$ is non empty. 
\end{theorem}
There is an explicit construction of a cyclotomic associator over $\kk = \mathbb{C}$ given as the renormalized holonomy from $0$ to $1$ of a cyclotomic version 
of the Knizhnik--Zamolodchikov differential equation. One can then use again descent methods to prove existence over $\kk = \mathbb{Q}$. 

\begin{theorem}[{\cite[Section 4.5]{calaque20moperadic}}]
Elements of the cyclotomic Grothendieck--Teichmüller group $\mathrm{GT}^\Gamma$ are in bijection with triples

\[
(\mu,f,g) \quad \text{in} \quad  \kk^\times  \times \widehat{\mathbb{F}}_2(\kk) \times \widehat{\mathbb{F}}_{N+1}(\kk)
\]

where $(\mu,f)$ is an element of the Grothendieck--Teichmüller group $\mathrm{GT}(\kk)$ and where $g$ is seen as a word 
$g(x^N, y, xyx^{-1},\dots,  x^{N-1} y x^{1-N})$ which satisfies the following relations:

\begin{enumerate}

\item The element $g$ can be seen inside of $\widehat{PB}_3^\Gamma(\kk)$, the Malcev completion of the cyclotomic pure braid group. Let $x_{ij}$ denote the standard generators of this group, then $g$ satisfies  

\[
g(x_{01},x_{12})g(x_{01}x_{02},x_{13}x_{23})f(x_{12},x_{23}) = g(x_{02}x_{12},x_{23})g(x_{01},x_{12}x_{13})
\]
\vspace{0.1pc}

in $\widehat{PB}_3^\Gamma(\kk)$.

\medskip

\item Let $\alpha$ be equal to $\bar{1}$ in $\Gamma$. Then we have

\[
x^{\frac{\mu -1}{N}}g(x,y)y^{\frac{\mu +1}{N}}g(z,y)^{-1}z^{\frac{\mu -1}{N}} \alpha \ast \left(g(z,y)y^{\frac{\mu -1}{N}}g(x,y)^{-1}\right) =1~,
\]

if $xyz = 1$ in $\widehat{\mathbb{F}_{2}}(\kk,p_N) $, where the action of $\alpha$ on $g$ is given by Lemma \ref{lemma: semi direct relative completion}.
\end{enumerate}

\medskip

Under this bijection, the group structure of the Grothendieck--Teichmüller group is given by: 
\[
(\mu_1,f_1,g_1) \star (\mu_2,f_2,g_2) = \left(\mu,f,g\right)~,
\]
where

\[
\mu \coloneqq \mu_1\mu_2 \quad \text{and} \quad f \coloneqq f_1\left(x^{\mu_2},f_2(x,y)y^{\mu_2}f_2(y,x)\right)f_2(x,y)~,
\]

and

\[
g \coloneqq  g_1\left(x^{\frac{\mu_2-1}{N}},g_2(x,y) y^{\mu_2}g_2(x,y)^{-1}\right)g_2(x,y).
\]

And, under this bijection, its action on the set of cyclotomic associators is given by:

\[
(\mu,f,g) \bullet (\lambda,\Phi,\Psi) = \left(\mu\lambda, f\left(e^{\lambda t_{12}}, \Phi(t_{12},t_{23})e^{\lambda t_{23}}\Phi(t_{23}, t_{12})\right)\Phi(t_{12},t_{23}), \Upsilon \right)~,
\]

where 

\begin{align*}
\Upsilon \coloneqq \Psi(t_{12}, t_{23}^0, \cdots, t_{23}^{N-1})~ g \Bigg( e^{\lambda t_{12}}, \Psi(t_{12}, t_{23}^0, \cdots, t_{23}^{N-1})~ e^{\lambda t_{23}} ~ \Psi(t_{12}, t_{23}^0, \cdots, t_{23}^{N-1})^{-1},   \\
e^{(\lambda/N) t_{12}} ~ \Psi(t_{12}, t_{23}^1, \cdots, t_{23}^{N})~ e^{\lambda t_{23}^1} ~ \Psi(t_{12}, t_{23}^1, \cdots, t_{23}^N)^{-1} e^{-(\lambda/N) t_{12}}, \cdots, \\
e^{(\lambda(N-1)/N) t_{12}}~ \Psi(t_{12}, t_{23}^{N-1}, \cdots, t_{23}^{2N-2}) ~ e^{\lambda t_{23}^{N-1}} ~\Psi(t_{12}, t_{23}^{N-1}, \cdots, t_{23}^{2N-2})^{-1} ~ e^{-(\lambda(N-1)/N)t_{12}}\Bigg)~.
\end{align*}
\end{theorem}

\begin{remark}
There is also an explicit description of the cyclotomic graded Grothendieck--Teichmüller group which can be obtained by computing the relations that elements of 
\cref{rmk: elements in pacd Gamma} satisfy.
\end{remark}

\subsection{Topological description of $\pab^\Gamma$}\label{subsec: topological cyclotomic}

Analogue statements to those mentioned in \S\ref{Subsection: Description topo} hold in the cyclotomic case. First, we introduce the configuration space 

\[
\mathrm{C}(\mathbb{C}^{\times},n) \coloneqq \left\{ (x_1, \cdots, x_n) \in \left(\mathbb{C}^{\times}\right)^n ~~|~~x_i \neq x_j~~\text{if}~~i \neq j \right\}/\mathbb{R}_{> 0}~.
\]
\vspace{0.1pc}

This space is isomorphic to the space of configurations of $n + 1$-points. Here the missing origin in $\mathbb{C}^{\times}$ is seen as a \textit{fixed} point in the configurations.

\medskip

We consider their Fulton-MacPherson compactifications $\overline{\mathrm{C}}(\mathbb{C}^{\times},n)$, which this time have a 
\textit{topological moperad structure} over the topological operad $\overline{\mathrm{C}}(\mathbb{C})$~. Indeed, one can compute that 

\[
\partial~\overline{\mathrm{C}}(\mathbb{C}^{\times},n) \cong \bigcup_{k \geq 1} \bigcup_{\mathrm{J_1} \sqcup \cdots \sqcup\mathrm{J_k} = [n]}  \overline{\mathrm{C}}(\mathbb{C}^{\times},[k]) \times \overline{\mathrm{C}}(\mathbb{C}^{\times},J_m) \times \left(\prod_{j =1, j \neq m}^k \overline{\mathrm{C}}(\mathbb{C},\mathrm{J}_j) \right)~,
\]

where $[n] = \{0, \cdots, n \}$ and where $J_m$ for $1 \leq m \leq k$ is the subset of $[n]$ which contains the element $0$. Using this decomposition of the 
boundary, one can write the moperad structure maps as obvious inclusions. We denote this topological moperad over $\overline{\mathrm{C}}(\mathbb{C})$ by $\overline{\mathrm{C}}(\mathbb{C}^{\times})$.

\begin{remark}
The moperad $\overline{\mathrm{C}}(\mathbb{C}^{\times})$ is also given by applying Proposition \ref{prop: derived moperad of an operad} to $\overline{\mathrm{C}}(\mathbb{C})$. 
\end{remark}

Recall $\mathcal{P}a_0$, the moperad over $\mathcal{P}a$ in the category of sets defined in \cref{example: moperad Pa0}. Now lets view this moperad as 
a topological moperad over $\mathcal{P}a$ where we impose the discrete topology both on $\mathcal{P}a_0$ and on $\mathcal{P}a$. 
 
\begin{lemma}
There is an inclusion of topological moperads 

\[
\epsilon: \mathcal{P}a_0 \hookrightarrow \overline{\mathrm{C}}(\mathbb{C}^{\times})
\]

over the topological operad $\mathcal{P}a$.
\end{lemma}

\begin{proof}
Follows from Lemma \ref{lemme: morphisme topo operades} and Proposition \ref{prop: derived moperad of an operad}. 
\end{proof}

\begin{proposition}\label{prop: iso avec Pab1}
There is an isomorphism of moperads in the category of groupoids

\[
\pab^{(1)} \cong \Pi_1(\overline{\mathrm{C}}(\mathbb{C}^{\times}), \mathcal{P}a_0)~.
\]

over the operad $\pab$.
\end{proposition}

\begin{proof}
Follows from Theorem \ref{thm: pab iso topo} and Proposition \ref{prop: derived moperad of an operad}. 
\end{proof}

\subsubsection*{Twisted configuration spaces}

Let again $\Gamma = \mathbb{Z}/N\mathbb{Z}$. We consider the following \textit{twisted configuration spaces} given by 

\[
\mathrm{C}(\mathbb{C}^{\times},n, \Gamma) \coloneqq \left\{ (x_1, \cdots, x_n) \in \left(\mathbb{C}^{\times}\right)^n ~~|~~x_i \neq \zeta.x_j~~\text{if}~~i \neq j ~|~ \zeta \in \mu_N \right\}/\mathbb{R}_{> 0}~,
\]

where $\mu_N$ denotes the group given by the complex $N$-roots of the unit. 
\begin{remark}
One easily sees that the fundamental group of $\mathrm{C}(\mathbb{C}^{\times},n, \Gamma)$ is the cyclotomic pure braid group $\mathrm{PB}_n^\Gamma$. 
Actually, the hyperplanes defined by $x_i = \zeta.x_j$ and $x_i=0$ are the reflection hyperplanes associated with the complex reflection group $\Gamma^n\rtimes\mathbb{S}_n$. 
Therefore $\mathrm{PB}_n^\Gamma$ is the pure braid group associated with $\Gamma^n\rtimes\mathbb{S}_n$, in the sense of \cite{BMR}.  
\end{remark}

\medskip

These spaces admits Fulton-MacPherson compactifications $\overline{\mathrm{C}}(\mathbb{C}^{\times},n, \Gamma)$. This family also has a natural moperad 
structure over the operad $\overline{\mathrm{C}}(\mathbb{C})$, which is again given by a similar decomposition of its boundary. We denote by 
$\overline{\mathrm{C}}(\mathbb{C}^{\times}, \Gamma)$ this moperad, which is naturally endowed, as a moperad, with a $\Gamma$-action.

\medskip 

The moperad $\mathcal{P}a^\Gamma$ is the set theoretical moperad over $\mathcal{P}a$, where $\mathcal{P}a^\Gamma(n)$ is given by maximal parenthesizations 
on words $0 \hspace{1pt}(\sigma_1)_{\alpha_1} \cdots (\sigma_n)_{\alpha_n}$, where $\sigma$ is a permutation in $\mathfrak{S}_n$ and 
$(\alpha_1,\cdots,\alpha_n)$ is an element of $\Gamma^n$. It is naturally endowed, as a moperad, with a $\Gamma$-action. We view this construction in the category 
of topological spaces by endowing it with the discrete topology.

\begin{lemma}
There is a $\Gamma$-equivariant inclusion of topological moperads 

\[
\epsilon: \mathcal{P}a^\Gamma \hookrightarrow \overline{\mathrm{C}}(\mathbb{C}^{\times},\Gamma)
\]

over the topological operad $\mathcal{P}a$.
\end{lemma}

\begin{proposition}\label{prop: iso avec PabN}
There is a $\Gamma$-equivariant isomorphism of moperads in the category of groupoids

\[
\pab^{(N)} \cong \Pi_1(\overline{\mathrm{C}}(\mathbb{C}^{\times},\Gamma), \mathcal{P}a^\Gamma)~.
\]

over the operad $\pab$.
\end{proposition}

\begin{remark}
Proofs of the above results can be found in \cite{calaque20moperadic}, in particular inside the proof of Theorem 3.4 and inside the proof of Theorem 4.6.
\end{remark}

\newpage

%%%%%%%%%%%%% SECTION 4 %%%%%%%%%%%%%%%

\section{Elliptic and ellipsitomic associators}\label{Section: Elliptic}

\subsection{Motivations and general context}\label{ssec: Motivations elliptic}

One can phrase a deformation problem extending \cref{Problem 1}, making use of the notion of 
an \textit{elliptic structure} over a braided monoidal category \cite{UniversalKZB09,Enriquez14}. According to \cite{Enriquez14} an elliptic structure over a braided 
monoidal category $\mathcal C$ is a functor $F:\mathcal C\to\mathcal C_1$ together with two natural automorphisms of $F(-\otimes -)$ satisfying certain relations 
that we won't specify here. These relations naturally live in $\mathrm{End}\big(F(\otimes^n)\big)$, that is an operadic module over $\mathrm{End}(\otimes^n)$. 

\medskip

For a finite dimensional Lie algebra $\mathfrak{g}$, following \cite{UniversalKZB09} we consider the algebra morphism from $\mathfrak{U}(\mathfrak{g})$ to the algebra of 
$\mathsf{D}(\mathfrak{g})$ of differential operators on $\mathfrak{g}$ sending a generator $x\in\mathfrak{g}$ to the linear vector field on $\mathfrak{g}$ given by $\mathrm{ad}_x=[x,-]$. 
The induction functor $F:=\mathsf{D}(\mathfrak{g})\otimes_{\mathfrak{U}(\mathfrak{g})}-$ then goes from $\mathcal C:=\mathsf{Rep}(\mathfrak{g})$ 
to $\mathcal C_1:=\mathsf{Rep}\big(\mathsf{D}(\mathfrak{g})\big)$. One then gets that 
\begin{eqnarray*}
\mathrm{End}\big(F(\otimes^n)\big) & \cong & \mathrm{Hom}\big(\otimes^n,\mathrm{Res}\circ F(\otimes^n)\big)
\cong\Big(\mathsf{D}(\mathfrak{g})\otimes_{\mathfrak{U}(\mathfrak{g})}\big(\mathfrak{U}(\mathfrak{g})^{\otimes n}\big)\Big)^{\mathfrak g}\,,
\end{eqnarray*}
where $\mathrm{Res}:\mathcal C_1\to \mathcal{C}$ is the restriction functor, that is right adjoint to $F$, and the algebra map $\mathfrak{U}(\mathfrak{g})\to \mathfrak{U}(\mathfrak{g})^{\otimes n}$ 
is the iterated coproduct\footnote{Note that, according to \cite[Remark 5.7]{UniversalKZB09}, there is an isomorphism 
\[
\Big(\mathsf{D}(\mathfrak{g})\otimes_{\mathfrak{U}(\mathfrak{g})}\big(\mathfrak{U}(\mathfrak{g})^{\otimes n}\big)\Big)^{\mathfrak g} \cong
\Big(\mathsf{D}(\mathfrak{g})\otimes\big(\mathfrak{U}(\mathfrak{g})^{\otimes (n-1)}\big)\Big)^{\mathfrak g}\,.
\]}. 
%For a Lie group $G$ with Lie algebra $\mathfrak{g}$, the universal enveloping algebra is isomorphic to the subalgebra of left invariant differential operators on $G$; 
%hence we have an algebra morphism $\mathfrak{U}(\mathfrak{g})\to \mathsf{D}(G)$, leading to an induction functor 
%$F:\mathcal C:=\mathsf{Rep}(\mathfrak{g})\to\mathsf{Rep}\big(\mathsf{D}(G)\big)=:\mathcal C_1$; one can see that 
%$\mathrm{End}\big(F(\otimes^n)\big)\cong \big(\mathsf{D}(G)\otimes \mathfrak{U}(\mathfrak{g})^{\otimes(n-1)}\big)^{\mathfrak g}$. 
%\begin{remark}
%In \cite{UniversalKZB09}, one uses $\mathfrak{U}(\mathfrak{g})\to \mathsf{D}(\mathfrak{g})$ instead, where $x\in\mathfrak{g}$ is sent to the linear vector field on 
%$\mathfrak{g}$ given by $\mathrm{ad}_x$. In this case one obtains 
%$\mathrm{End}\big(F(\otimes^n)\big)\cong \big(S(\mathfrak{g})\otimes \mathfrak{U}(\mathfrak{g})^n\big)^{\mathfrak g}$. 
%The two approaches coincide when one works in a formal neighborhood of the origin (where $\mathfrak{g}$ and $G$ are isomorphic, thanks to the exponential map). 
%\end{remark}
In this context we have a (quite trivial) elliptic structure, and any $t\in S^2(\mathfrak g)^{\mathfrak g}$ that is non-degenerate (meaning that the induced map 
$\mathfrak{g}^*\to\mathfrak{g}$ is an isomorphism) provides a first order deformation of it. A natural deformation problem is the following: can one extend this 
first order deformation to a formal one?

\medskip

Just as in \S\ref{ssec-1.2-universal} and \S\ref{ssec-cyclotomic-motivation}, there is a collection of graded Lie algebras 
$(\overline{\mathfrak{t}}_n^{~\mathrm{ell}})_{n\geq0}$, see \S\ref{ssec-4.3-ellipticCD}, together with insertion-coproduct type morphisms that makes 
an operadic module over the Drinfeld--Kohno operad $(\mathfrak{t}_n)_{n\geq0}$. It comes with an operadic module morphism 
$\overline{\mathfrak{t}}_\bullet^{~\mathrm{ell}}\longrightarrow \mathrm{End}\big(F(\otimes^\bullet)\big)$, making it a perfect location for finding ``universal'' 
solutions to our deformation problem, that we call \textit{elliptic associators}. 

\medskip

One can construct such universal solutions as holonomies of a genus one version of the KZ connection (known as the Knizhnik--Zamolodchikov--Bernard connection, 
see \cite{UniversalKZB09}). Such a connection depends on the choice of a complex structure on the torus, and actually extends to the whole moduli space of 
marked elliptic curves. As a result, we get that these holonomies have many interesting features: 
\begin{itemize}
\item They satisfy modular invariance properties (see \cite{Enriquez14}). 
\item Their coefficient are interesting numbers, involving iterated integrals of Eiseinstein series, that one can view as elliptic analogs of MZV \cite{Enriquez16}. 
As a consequence (of the defining relations of an elliptic associator), one gets relations between iterated integrals of Eiseinstein series and ordinary MZV. 
\item All of the above shall lift to some appropriate motivic context (presumbably the one of \cite{HainMatsumoto}). 
\end{itemize}

\medskip

Note that the notion of an elliptic structure has been generalized to higher genus in \cite{HumbertThese}, under the name \textit{genus $g$ structure} 
($g=1$ corresponding to the elliptic case), in relation with the study of invariants for tangles in a thickened surface. 

\subsection{The module of parenthesized elliptic braids}\label{ssec: parenthesized elliptic braids}

Our goal is to build an elliptic version of $\pab$, much in the same way as in Subsection \ref{subsection: pab}. The elliptic version $\pab_{\mathit{ell}}$ is going to 
be a right module over the operad $\pab$. In order to construct it, we will need to consider the \textit{pure braid group of the torus}. 

\medskip

Let $\mathbb{T}$ be the topological $2$-torus. The configuration space of $n$-points in the torus $\mathbb{T}$ is defined as follows:

\[
\mathrm{Conf}_n(\mathbb{T}) \coloneqq \left\{ (x_1, \cdots, x_n) \in \mathbb{T}^n~~|~~x_i \neq x_j~~\text{if}~~i \neq j \right\}~.
\]
\vspace{0.1pc}

Its \textit{reduced version} 

\[
\mathrm{C}(\mathbb{T},n) \coloneqq \mathrm{Conf}_n(\mathbb{T})/\mathbb{T}~,
\]
\vspace{0.1pc}

is obtained by considering configurations of $n$-points modulo the action of the torus on itself by multiplication, using its abelian group structure. It is a path-connected topological space for $n \geq 0$. It comes equipped with an action of the symmetric group $\mathfrak{S}_n$, given by the permutation of the indices of the points.

\begin{definition}[Reduced braid group of the torus]
The \textit{reduced pure braid group of the torus group} $\overline{\mathrm{PB}}_n(\mathbb{T})$ on $n$-strands is given by the fundamental group 

\[
\overline{\mathrm{PB}}_n(\mathbb{T}) \coloneqq \pi_1(\mathrm{C}(\mathbb{T},n))~.
\]
\end{definition}

Let us describe this group. It is generated by elements $\{a_i\}$ and $\{b_i\}$ for $1 \leq i \leq n$, which correspond to $i$-th point following the trajectory given by each of the generators of the fundamental group of the torus $\pi_1(\mathbb{T})$. They can be depicted as follows:
\[
\includegraphics[width=100mm,scale=1]{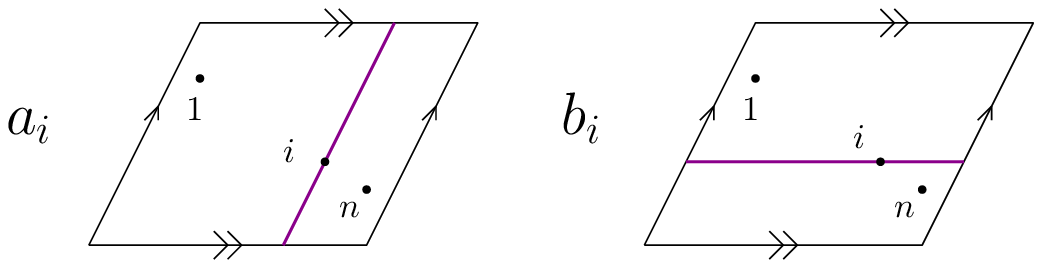}
\]

For all $n$, there is a canonical morphism of groups 
\[
\begin{tikzcd}[column sep=1.5pc,row sep=0pc]
\mathrm{PB}_n \arrow[r]
&\overline{\mathrm{PB}}_n(\mathbb{T}) \\
x_{ij} \arrow[r,mapsto] 
& \left[ a_j^{-1},b_i^{-1} \right] ~,
\end{tikzcd}
\]

induced by the inclusion of the plane into the fundamental domain of the torus. Thus any generator $x_{ij}$ of the pure braid group can be seen as an element in the reduced pure braid group of the torus, and it still corresponds to the $i$-th point moving around the $j$-th point.

\medskip

Without giving a full presentation of $\overline{\mathrm{PB}}_n(\mathbb{T})$, let us mention some of the important relations that the generators $\{a_i\}$ and $\{b_i\}$ satisfy. First notice that these generators commute, that is 
\[
[a_i,a_j] = [b_i,b_j] = 1
\]
for $i <j$, and we have that $a_1 \cdots a_n =b_1 \cdots b_n = 1$. Finally, for all $1 \leq i \leq n$, we have the relations

\[
a_{i+1} = \sigma_{i}^{-1}a_i \sigma_{i}^{-1} \quad \text{and} \quad b_{i+1} = \sigma_{i}^{-1} b_{i} \sigma_{i}^{-1}~,
\]
\vspace{0.1pc}

which are satisfied in the \textit{reduced braid group of the torus} $\overline{\mathrm{B}}_n(\mathbb{T})$, defined as  

\[
\overline{\mathrm{B}}_n(\mathbb{T}) \coloneqq \pi_1(\mathrm{C}(\mathbb{T},n)/\mathfrak{S}_n)~,
\]
\vspace{0.1pc}

and where $\sigma_{i}$ is the following generator:

\[
\includegraphics[width=65mm,scale=1]{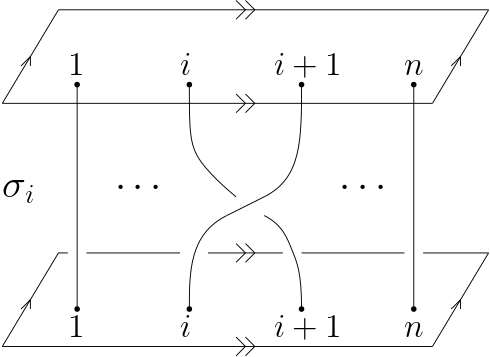}
\]
\vspace{0.1pc}

There is an elliptic analogue of the operad in the category of groupoids $\cob$, defined in Example \ref{example: cob}. The symmetric sequence in the category of groupoids $(\cob_{\mathit{ell}}(n))_{n \geq 0}$ given as follows.

\medskip

\begin{enumerate}
\item The objects of $\cob_{\mathit{ell}}(n)$ are elements of the symmetric group $\mathfrak{S}_n$.

\medskip

\item Morphisms between two permutations $\sigma$ and $\tau$ in $\mathfrak{S}_n$ are given by reduced pure braids on the torus going from $\sigma$ to $\tau$. 
\end{enumerate}

\medskip

The $\mathfrak{S}_n$-module $\cob_{\mathit{ell}}(n)$ has a canonical right module structure over the operad $\cob$, that is, there are maps 

\[
\{\circ_i: \cob_{\mathit{ell}}(n) \times \cob(k) \longrightarrow \cob_{\mathit{ell}}(n+k-1)\}
\]
\vspace{0.1pc}

for $1 \leq i \leq n$ which satisfy the obvious axioms. These functors are given on objects by the insertion of permutations and on morphisms by the inserting of a pure braid going from $\sigma'$ to $\tau'$ into the $i$-th spot of a reduced pure 
braid on the torus going from $\sigma$ to $\tau$. 

\medskip

There is an obvious morphism $\varphi: \mathrm{Ob}(\mathcal{P}a) \longrightarrow \mathrm{Ob}(\cob_{\mathit{ell}})$ of symmetric sequences in sets which simply forgets the parenthesis on permutations.

\begin{definition}[The right module $\pab_{\mathit{ell}}$]
The symmetric sequence in the category of groupoids $\pab_{\mathit{ell}}$ is defined to be the fake pull-back of $\cob_{\mathit{ell}}$ along the morphism 
$\varphi$. It inherits a right module structure over the operad $\pab$ coming from the right module structure of $\cob_{\mathit{ell}}$.
\end{definition}

\begin{remark}
In addition to the elements $R$ and $\Phi$ described in Theorem \ref{thm: presentation pab} which are also present in $\pab_{\mathit{ell}}$, there are two other distinguished elements. The object $(12)$ in $\pab_{\mathit{ell}}(2)$ has 
this time two non-trivial automorphisms, denoted by $A$ and $B$: $A$, resp.~$B$, is induced by the element $a_1=a_2^{-1}$, resp.~$b_1=b_2^{-1}$, in $\overline{\mathrm{PB}}_2(\mathbb{T})$. 
%They can be represented as follows 
%\[
%\includegraphics[width=65mm,scale=1]{elementAB.eps}
%\]
\end{remark}

The Malcev completion $\widehat{\pab}_{\mathit{ell}}(\kk)$ is a right module over $\widehat{\pab}(\kk)$ in the underlying category of pro-unipotent $\kk$-groupoids, since the Malcev completion functor is strong monoidal. This gives our \textit{first elliptic layer}.

\subsection{The module of elliptic chord diagrams}\label{ssec-4.3-ellipticCD}

Here we define the "infinitesimal version" of the right module $\pab_{\mathit{ell}}$, given in terms of elliptic chord diagrams.

\begin{definition}[Elliptic Drinfeld--Kohno algebras]\label{def: elliptic Drinfeld-Kohno}
The $n$-\textit{elliptic Drinfeld--Kohno algebra} $\mathfrak{t}_n^{\mathrm{ell}}$ is the bigraded Lie algebra given by the following presentation:

\medskip

\begin{enumerate}
\item It is generated by elements $\{t_{ij}\}$ for $1 \leq i \neq j \leq n$ in bidegree $(1,1)$, by elements $\alpha_i$ for $1 \leq i \leq n$ in bidegree $(1,0)$ and by elements $\beta_i$ for $1 \leq i \leq n$ in bidegree $(0,1)$.

\medskip

\item The generators are subject to the following relations:

\medskip

\begin{enumerate}
\item $t_{ij} = t_{ji}$. 
\medskip
\item $[t_{ij},t_{kl}] =0$.
\medskip
\item $[t_{ij},t_{ik} + t_{jk}] = 0$.
\medskip
\item $[\alpha_i,\beta_j] = t_{ij}$.
\medskip
\item $[\alpha_i,\alpha_j] = [\beta_i, \beta_j] = 0$.
\medskip
\item $[\alpha_i,\beta_i] = \displaystyle - \sum_{j~|~ j \neq i} t_{ij}$.
\medskip
\item $[\alpha_i, t_{jk}] = [\beta_i,t_{jk}] = 0$. 
\medskip
\item $[\alpha_i + \alpha_j, t_{ij} ] = [\beta_i + \beta_j, t_{ij}] = 0$.

\end{enumerate}
\end{enumerate}
\end{definition}

\begin{remark}
The elements $\sum_i \alpha_i$ and $\sum_i \beta_i$ are central elements in $\mathfrak{t}_n^{\mathrm{ell}}$.
\end{remark}

\begin{proposition}\label{prop: right mod structure}
The family $(\mathfrak{t}_n^{\mathrm{ell}})_{n\geq0}$ can be endowed with a right module structure over the operad $(\mathfrak{t}_n)_{n \geq 0}$ in the category of graded Lie algebras.

\begin{enumerate}
\item The action of the symmetric group $\mathfrak{S}_n$ on $\mathfrak{t}_n^{\mathrm{ell}}$ is given by :
\[
\sigma \bullet t_{ij} \coloneqq t_{\sigma(i)\sigma(j)}~~,~~ \sigma \bullet \alpha_i \coloneqq \alpha_{\sigma(i)} ~~ \text{and} ~~ \sigma \bullet \beta_i \coloneqq \beta_{\sigma(i)}~.
\]
\item The partial composition maps 
\[
\{\circ_p: \mathfrak{t}_n^{\mathrm{ell}} \oplus \mathfrak{t}_m \longrightarrow \mathfrak{t}_{n+m-1}^{\mathrm{ell}}\}
\]

of the right module structure are given by the following components (where we assume $i<j$):
\begin{eqnarray*}
\mathfrak{t}_m \ni t_{ij} & \longmapsto & t_{i+p-1 \hspace{1pt} j+p-1}~. \\[0.35cm]
\mathfrak{t}_n^{\mathrm{ell}} \ni t_{ij} & \longmapsto & 
\begin{cases*}
        t_{i+m-1\hspace{1pt}j+m-1} & if~~p < i~.\\
        {\displaystyle \sum_{k=i}^{i+m-1}t_{k\hspace{1pt}j+m-1}} & if~~ p = i~.\\
        t_{i\hspace{1pt}j+m-1} & if~~ i < p < j~.\\
        {\displaystyle \sum_{k=j}^{j+m-1}t_{i\hspace{1pt}k}} & if~~ p = j~.\\
        t_{i\hspace{1pt}j} & if~~ j < p~.
\end{cases*}\\[0.35cm]
\mathfrak{t}_n^{\mathrm{ell}} \ni \alpha_i & \longmapsto & 
\begin{cases*}
\alpha_{i+m-1} & if~~p < i~.\\
{\displaystyle \sum_{k=i}^{i+m-1}\alpha_k} & if~~p = i~.\\
\alpha_{i} & if~~i < p~.
\end{cases*}\\[0.35cm]
\mathfrak{t}_n^{\mathrm{ell}} \ni \beta_i & \longmapsto & 
\begin{cases*}
\beta_{i+m-1} & if~~p < i~.\\
{\displaystyle \sum_{k=i}^{i+m-1}\beta_k} & if~~p = i~.\\
\beta_{i} & if~~i < p~.
\end{cases*}
\end{eqnarray*}

%\[
%\begin{tikzcd}[column sep=1.5pc,row sep=0pc]
%t_{ij} ~~\text{in} ~~ \mathfrak{t}_m \arrow[r,mapsto]
%& t_{i+p-1 \hspace{1pt} j+p-1}~. \\
%t_{ij} ~~\text{in} ~~ \mathfrak{t}_n^{\mathrm{ell}} \arrow[r,mapsto]
%& \left\{
%    \arraycolsep=1.4pt\def\arraystretch{1.7}\begin{array}{lllll}
%        t_{i+m-1\hspace{1pt}j+m-1} ~~\mathrm{if} ~~p<i,j~.\\
%        t_{i\hspace{1pt}j+m-1}+t_{i+1\hspace{1pt}j+m-1}+\cdots+ t_{i+m-1\hspace{1pt}j+m-1}~~\mathrm{if}~~ p = i~.\\
%        t_{i\hspace{1pt}j+m-1} ~~\mathrm{if}~~ i < p < j~.\\
%        t_{i\hspace{1pt}j}+t_{i\hspace{1pt}j+1}+\cdots+t_{i\hspace{1pt}j+m-1}~~\mathrm{if}~~ p = j~.\\
%        t_{i\hspace{1pt}j}~~\mathrm{if}~~ p > i,j~.
%    \end{array}
%\right. \\
%\alpha_i ~~\text{in} ~~ \mathfrak{t}_n^{\mathrm{ell}} \arrow[r,mapsto]
%& \left\{
%    \arraycolsep=1.4pt\def\arraystretch{2.2}\begin{array}{ll}
%        \alpha_i ~~\mathrm{if} ~~p \neq i~.\\
%        \displaystyle \sum_{r = 0}^{m-1} \alpha_{i+r}~~\mathrm{if}~~ p = i~.\\
%    \end{array}
%\right. \\
%\beta_i ~~\text{in} ~~ \mathfrak{t}_n^{\mathrm{ell}} \arrow[r,mapsto]
%& \left\{
%    \arraycolsep=1.4pt\def\arraystretch{2.2}\begin{array}{ll}
%        \beta_i ~~\mathrm{if} ~~p \neq i~.\\
%        \displaystyle \sum_{r = 0}^{m-1} \beta_{i+r}~~\mathrm{if}~~ p = i~.\\
%    \end{array}
%\right.
%\end{tikzcd}
%\]
\end{enumerate}
\end{proposition}

\begin{proof}
Using the formulas above, one can define a presheaf on the category $\mathrm{Fin}_{*,\not*}$. As explained in the Appendix \ref{Appendix B2: modules}, this induces a right module structure for the cocartesian monoidal structure, which descends to the cartesian monoidal structure on Lie algebras since the images of the insertion morphisms commute with the images of the coproduct morphisms. See proofs of Propositions \ref{prop: operad structure on Kohno-Drinfeld} and \ref{cyclotomic moperad structure on Kohno-Drinfeld} for analogue arguments.
\end{proof}

\begin{definition}[Reduced elliptic Drinfeld--Kohno algebras]
The $n$-\textit{reduced elliptic Drinfeld--Kohno algebra} $\overline{\mathfrak{t}}_n^{~\mathrm{ell}}$ is given by 

\[
\overline{\mathfrak{t}}_n^{~\mathrm{ell}} \coloneqq \frac{\mathfrak{t}_n^{\mathrm{ell}}}{\left(\sum_i \alpha_i~,~\sum_i \beta_i \right)}~.
\]
\end{definition}

The right module structure defined in Proposition \ref{prop: right mod structure} descends to a right module structure on the symmetric sequence $(\overline{\mathfrak{t}}_n^{~\mathrm{ell}})_{n \geq 0}$ in the category of graded Lie algebras given by the \textit{reduced} elliptic Drinfeld--Kohno algebras.

\begin{definition}[The right module of reduced elliptic chord diagrams]
The \textit{right module of reduced elliptic chord diagrams} $\mathcal{CD}_{\mathrm{ell}}$ is the right module over the operad of chord diagrams $\mathcal{CD}$ obtained by applying the universal enveloping algebra functor to the right module of \textit{reduced} elliptic Drinfeld--Kohno algebras.
\end{definition}

Morphisms in $\mathcal{CD}_{\mathrm{ell}}(n)$ can be represented as chord diagrams on $n$-strands, with two extra types of chords which correspond to the added generators $\alpha_i$ and $\beta_i$. These extra chords can be depicted as follows:

\[
\includegraphics[width=100mm,scale=1]{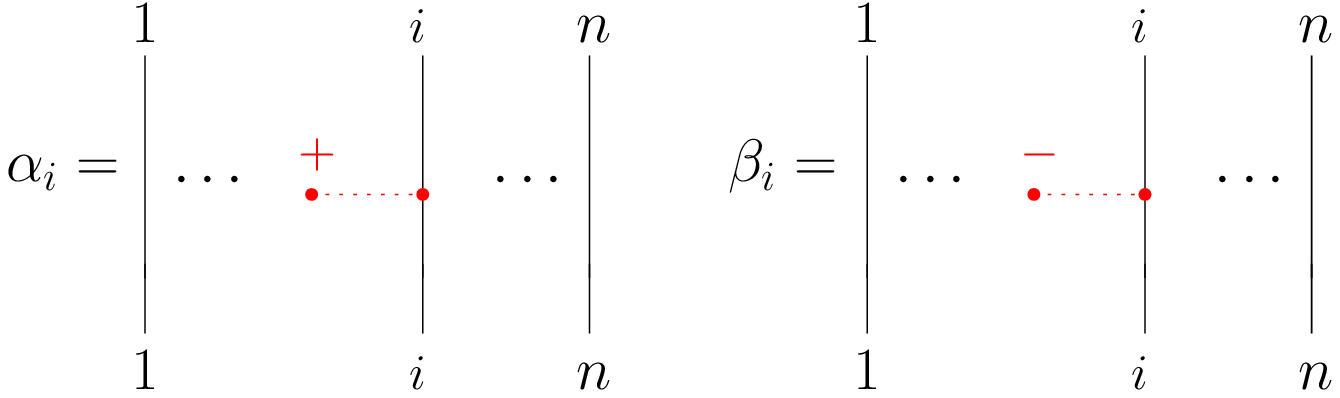}
\]

As an example, the relation (f) in the Definition \ref{def: elliptic Drinfeld-Kohno} can be represented in terms of chord diagrams as follows:

\[
\includegraphics[width=125mm,scale=1]{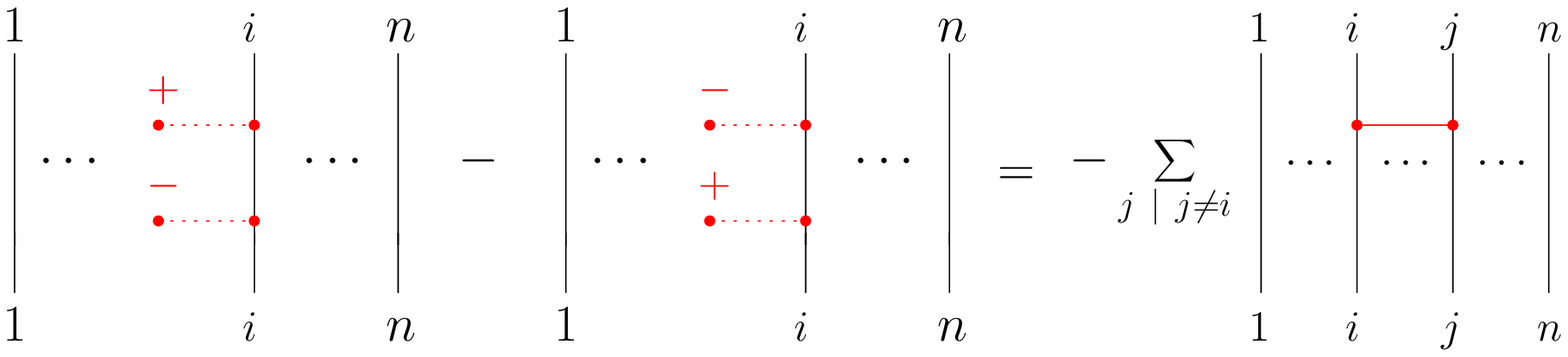}
\]

Recall the set-theoretical operad $\mathcal{P}a$ of parenthesized permutations. We consider its $\kk$-linear extension $(\mathcal{P}a)_\kk$ to the underlying category of small categories enriched in cocommutative $\kk$-coalgebras, constructed in Definition \ref{def: k-extension}. It is canonically a right module over itself.

\begin{definition}[Right module of parenthesized reduced elliptic chord diagrams]
The \textit{right module of parenthesized elliptic chord diagrams} $\mathcal{P}a\mathcal{CD}_{\mathrm{ell}}$ is defined as the Hadamard tensor product

\[
\mathcal{P}a\mathcal{CD}_{\mathrm{ell}} \coloneqq (\mathcal{P}a)_\kk \otimes \mathcal{CD}_{\mathrm{ell}}~,
\]
\vspace{0.1pc}

in the category of symmetric sequences in the category $\mathsf{Cat}(\mathsf{Coalg}_\kk)$. It comes equipped with a right module structure over the operad $\mathcal{P}a\mathcal{CD}$ since this operad is given by the Hadamard product of $(\mathcal{P}a)_\kk$ and $\mathcal{CD}$.
\end{definition}

\begin{example}
There are two distinguished endomorphisms of the object $(12)$ in $\mathcal{P}a\mathcal{CD}_{\mathrm{ell}}(2)$, called $\alpha$ and $\beta$, which are given by $1 \otimes \alpha_1$ and $1 \otimes \beta_1$, and which can be represented as follows:

\[
\includegraphics[width=75mm,scale=1]{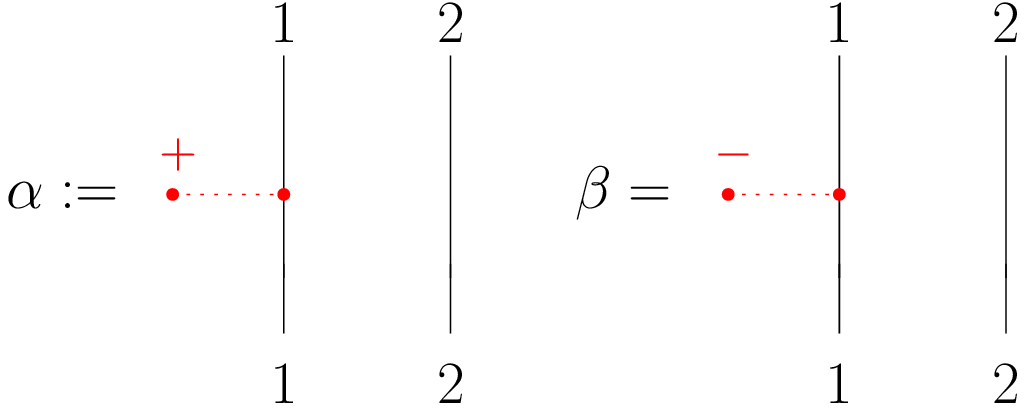}
\]
\hfill$\triangle$
\end{example}

\begin{remark}\label{rmk:presentation of PaCD elliptic}
The morphisms $\alpha$ and $\beta$ generate $\mathcal{P}a\mathcal{CD}_{\mathrm{ell}}$ as a right module over $\mathcal{P}a\mathcal{CD}$ the underlying category of $\mathsf{Cat}(\mathsf{Coalg}_\kk)$, and in fact one can give an explicit presentation of this right module. 
\end{remark}

The relations that define the elliptic Drinfeld--Kohno Lie algebras $\mathfrak{t}_n^{\mathrm{ell}}$ are homogeneous with respect to the bigrading imposed on generators. Hence it has a canonical filtration given by the total degree of the bifiltration. 

\medskip

We define $\widehat{\pacd}_{\mathrm{ell}}$ as the right module over $\widehat{\pacd}$ in the category of categories enriched over complete cocommutative coalgebras obtained by completing each hom-set of $\pacd_{\mathrm{ell}}(n)$ with 
respect to the aforementioned filtration. It defines a right module in the underlying category of complete Hopf groupoids. 

\medskip

We get a right module $\mathrm{Grp}(\widehat{\pacd}_{\mathrm{ell}})$ over $\mathrm{Grp}(\widehat{\pacd})$ in the underlying category of pro-unipotent $\kk$-groupoids by applying the group-like element functor on both sides. This will be our \textit{second elliptic player}. 

\subsection{Elliptic associators}\label{ssec: elliptic associators}
These two elliptic players enable us to define elliptic associators purely in operadic terms.

\begin{definition}[Elliptic associators] \label{Def: elliptic associator}
The set of \textit{elliptic associators} is given by 
\[
\mathrm{Assoc}_{\mathrm{ell}}(\kk) \coloneqq \mathrm{Iso}^+_{\mathsf{RMod}(\mathsf{p.u}\text{-}\mathsf{Grpd}_{\kk})}\left(\left(\widehat{\pab}_{\mathrm{ell}}(\kk),\widehat{\pab}(\kk)\right),\left(\mathrm{Grp}(\widehat{\pacd}_{\mathrm{ell}}),\mathrm{Grp}(\widehat{\pacd})\right)\right)~.
\]
An elements in $\mathrm{Assoc}_{\mathrm{ell}}(\kk)$ amounts to the data of:

\medskip

\begin{enumerate}
\item an isomorphism $F$ between $\widehat{\pab}(\kk)$ and $\mathrm{Grp}(\widehat{\pacd})$ which is the identity on objects (an associator)~,

\medskip

\item an isomorphism $S$ between $\widehat{\pab}_{\mathrm{ell}}(\kk)$ and $\mathrm{Grp}(\widehat{\pacd}_{\mathrm{ell}})$ of right modules which is the identity on objects, and which lies above the isomorphism $F$. 
\end{enumerate}
\end{definition}

This definition entails that the set of elliptic associators is a bitorsor over the automorphisms groups of each right module that are the identity morphism on objects. This allows us to define, as in the previous cases, the elliptic Grothendieck--Teichmüller group and the elliptic graded Grothendieck--Teichmüller group.

\begin{definition}[Elliptic Grothendieck--Teichmüller group]
The \textit{elliptic Grothendieck--Teichmüller group} over $\kk$ is given by 
\[
\mathrm{GT}_{\mathrm{ell}}(\kk) \coloneqq \mathrm{Aut}_{\mathsf{RMod}(\mathsf{p.u}\text{-}\mathsf{Grpd}_{\kk})}^+\left(\left(\widehat{\pab}_{\mathrm{ell}}(\kk),\widehat{\pab}(\kk)\right)\right),
\]
that is, the group of automorphisms of right modules of $\widehat{\pab}_{\mathrm{ell}}(\kk)$ which are the identity on objects.
\end{definition}

\begin{definition}[Elliptic graded Grothendieck--Teichmüller group]
The \textit{elliptic graded Grothendieck--Teichmüller group} over $\kk$ is given by 
\[
\mathrm{GRT}_{\mathrm{ell}}(\kk) \coloneqq \mathrm{Aut}_{\mathsf{RModCat(Coalg_\kk)}}^+\left(\left(\widehat{\pacd}_{\mathrm{ell}},\widehat{\pacd}\right)\right),
\]
that is, the group of automorphisms of right modules of $\widehat{\pacd}_{\mathrm{ell}}$ which are the identity on objects.
\end{definition}

\subsection{More concrete descriptions}\label{ssec: more concrete elliptic}

\begin{theorem}[{\cite[Theorem 3.3]{calaque2020ellipsitomic}}]\label{thm: presentation pab elliptic}
The right module $\pab_\mathrm{ell}$, as a right module over $\pab$ in groupoids having $\mathcal{P}a$ as a $\mathcal{P}a$-right module of objects, admits the following presentation. Thegenerating morphisms 
are the endomorphisms $A$ and $B$ of $(12)$ in $\pab_\mathrm{ell}(2)$. 
%They can represented as follows:
%\[
%\includegraphics[width=65mm,scale=1]{elementAB.eps}
%\]
They satisfy the following relations

\begin{enumerate}

\item The \textit{nonagon relation} for $A$ amounts to the commutativity of the following diagram

\medskip
\hspace{-8.5pc}
\begin{tikzcd}[column sep=4pc,row sep=2pc]
&(12)3 \arrow[ld,"\mathrm{id}_{(12)3}",swap] \arrow[r,"\Phi"]
&1(23) \arrow[r,"A~ \circ_2 ~ \mathrm{id}_{12}"]
&1(23) \arrow[rd,"\tilde{R} ~ \circ_2 ~ \mathrm{id}_{12}"]
&   \\
(12)3 
&
&
&
&(23)1 \arrow[d,"(231) ~\bullet~ \Phi"] \\
3(12) \arrow[u,"(312) ~\bullet~  (\tilde{R}~ \circ_2 ~ \mathrm{id}_{12})"]
&
&
&
&2(31) \arrow[ld,"(231) ~\bullet~  (A~ \circ_2 ~ \mathrm{id}_{12})"]\\
&3(12) \arrow[ul,"(312) ~\bullet~  (A~ \circ_2 ~ \mathrm{id}_{12})"] 
&(31)2 \arrow[l,"(312) ~\bullet~  \Phi",swap] 
&2(31) \arrow[l,"(231) \bullet  (\tilde{R} \circ_2 \mathrm{id}_{12})",swap]
&
\end{tikzcd}

where $\tilde{R}$ is given by $((21) \bullet R)^{-1}$. 

\medskip

\item The \textit{nonagon relation} for $B$ amounts to the above diagram, where one replaces $A$ with $B$.

\medskip

\item The \textit{mixed relation} amounts to the commutativity of the following diagram

\[
\begin{tikzcd}[column sep=5pc,row sep=3pc]
(12)3 \arrow[rr," f g f^{-1} g^{-1} "] \arrow[rd,"\mathrm{id}_{12} ~\circ_1 ~R",swap] 
&
&(12)3 \\
&(21)3 \arrow[ur,"\mathrm{id}_{12} ~\circ_1 ~((21) ~\bullet ~ R)",swap] 
&
\end{tikzcd}
\]

where the arrow $f$ is given by the following composition 

\[
\begin{tikzcd}[column sep=5pc,row sep=3pc]
(12)3 \arrow[r,"\Phi"]
&1(23) \arrow[r," A ~\circ_2 ~\mathrm{id}_{12}"]
&1(23) \arrow[r,"\Phi^{-1}"]
&(12)3~,
\end{tikzcd}
\]

and where the arrow $g$ is given by the following composition

\vspace{1pc}
\hspace{-7.5pc}
\begin{tikzcd}[column sep=4.25pc,row sep=3pc]
(12)3 \arrow[r,"\mathrm{id}_{12} ~\circ_1 ~\tilde{R}"] 
&(21)3 \arrow[r,"(213) ~\bullet ~ \Phi"] 
&2(13) \arrow[r,"(213) \bullet (B \circ_2  \mathrm{id}_{12})"] 
&2(13) \arrow[r,"((213) ~\bullet ~ \Phi)^{-1}"] 
&(21)3 \arrow[r,"\mathrm{id}_{12} ~\circ_1 ~((21) ~\bullet ~ \tilde{R})"]
&(12)3~.
\end{tikzcd}
\end{enumerate} 
\end{theorem}

We now deduce from this presentation a more explicit description of elliptic associators. 

\medskip

Observe that the reduced elliptic Drinfeld--Kohno algebra $\overline{t}_2^{~\mathrm{ell}}$ is a free Lie algebra on two generators, $\alpha \coloneqq \alpha_1$ and $\beta \coloneqq \beta_2$. 
Therefore an element in $\widehat{\mathfrak{U}}(\overline{t}_2^{~\mathrm{ell}})$ can be written a non-commutative formal power series in variables $\alpha_1$ and $\beta_2$.

\begin{theorem}[{\cite[Theorem 3.9]{calaque2020ellipsitomic}}]\label{thm: explicit ellip assos}
There is a bijection between the set of elliptic associators $\mathrm{Assos}_{\mathrm{ell}}(\kk)$ and the set of quadruples $(\lambda,\Phi,A_{+},A_{-})$ where 

\begin{enumerate}
\item $(\lambda,\Phi)$ is a Drinfeld associator;

\medskip

\item $A_{+}$ and $A_{-}$ are group-like formal power series in two  non-commutative variables, which satisfy the following equations in $\widehat{\mathfrak{U}}(\overline{t}_3^{~\mathrm{ell}})$: 

\begin{align*}
\mathrm{(a)}\quad\qquad\qquad\qquad\qquad\Phi(t_{12},t_{23})~ A_{\pm}(\alpha_1,\beta_2 + \beta_3)~e^{-\lambda(t_{12} + t_{13})/2} \\
\Phi(t_{23},t_{13})~A_{\pm}(\alpha_2,\beta_3 + \beta_1)~e^{-\lambda(t_{23} + t_{12})/2} \\
\Phi(t_{13},t_{12})~ A_{\pm}(\alpha_3,\beta_1 + \beta_2)~e^{-\lambda(t_{13} + t_{23})/2}~=~1
\end{align*}

and

\begin{align*}
\mathrm{(b)}\qquad\qquad e^{\lambda.t_{12}} = \left[\Phi(t_{12},t_{23})~ A_{+}(\alpha_1,\beta_2 + \beta_3)~\Phi(t_{12},t_{23})^{-1}~,\right. \\
\left.~e^{-\lambda.t_{12}/2}~\Phi(t_{12},t_{13})~A_{-}(\alpha_2,\beta_1 + \beta_3)~\Phi(t_{12},t_{13})^{-1}~e^{-\lambda.t_{12}/2}\right]~
\end{align*}
where the bracket denotes the commutator. 
\end{enumerate} 
\end{theorem}

\begin{remark}
The above description of an elliptic associator corresponds to the original definition given by Enriquez in \cite{Enriquez14}, up to taking inverses (because of our unusual convention for products). 
\end{remark}

\begin{proof}[Sketch of proof]
An isomorphism $S$ between $\widehat{\pab}_{\mathrm{ell}}(\kk)$ and $\mathrm{Grp}(\widehat{\pacd}_{\mathrm{ell}})$ is completely determined by the images 
of the generators $A$ and $B$. 
These images give two group-like elements in $\widehat{\mathfrak{U}}(\overline{t}_2^{~\mathrm{ell}})$, that is to say two group-like non-commutative formal 
power series in two variables. 
The relations above correspond precisely to the relations induced by the presentation of Theorem \ref{thm: presentation pab elliptic} (more precisely, (a) corresponds 
to the two nonagon relations, and (b) corresponds to the mixed relation).
\end{proof}

\begin{theorem}[{\cite{Enriquez14}}]
Let $\kk$ be a field of characteristic zero. The set of elliptic associators is non empty. 
\end{theorem}
For every choice of a complex structure on the torus (i.e.~for every elliptic curve), there is an explicit construction of an elliptic associator over $\kk=\mathbb{C}$ given as renormalized holonomies (along $A$- and $B$-cycles) 
of an elliptic version of the Knizhnik--Zamolodchikov differential equation, studied in \cite{UniversalKZB09}. One can then use descent methods, once again, to prove existence over $\kk=\mathbb{Q}$. 

\begin{remark}
The above still makes sense for the nodal curve, where one gets back the usual Knizhnik--Zamolodchikov differential equation (up to a change of coordinate, see \cite{UniversalKZB09}). The holonomies can therefore be expressed using a Drinfeld associator. 
Explicit formulas are given in \cite{UniversalKZB09,Enriquez14}) and provide an elliptic associator for \textit{any} given Drinfeld associator. In other words, the forgetful map $\mathrm{Assos}_{\mathrm{ell}}(\kk)\to \mathrm{Assos}(\kk)$ 
has a section. 
\end{remark}

\medskip

The presentation of $\pab_{\mathrm{ell}}$ as a right module also allows us to give an explicit description of the elliptic Grothendieck--Teichmüller group.

\begin{theorem}[{\cite[Section 3.7]{calaque2020ellipsitomic}}]\label{thm: explicit ellip GT}
Elements of the elliptic Grothendieck--Teichmüller group $\mathrm{GT}_{\mathrm{ell}}$ are in bijection with quadruples 

\[
(\mu,f,g_{+},g_{-}) \quad \text{in} \quad \kk^\times \times \widehat{\mathbb{F}}_2(\kk) \times \widehat{\mathbb{F}}_2(\kk)\times \widehat{\mathbb{F}}_2(\kk)~,
\]
\vspace{0.1pc}

where $(\mu,f)$ is an element of the Grothendieck--Teichmüller group $\mathrm{GT}$, and where $g_{\pm}$ are seen as words $g_{\pm}(A,B)$ which satisfy the following equations 

\begin{enumerate}
\medskip

\item $\left(f(\sigma_1^2,\sigma_2^2) g_{\pm}(a_1,b_1)(\sigma_1 \sigma_2^2 \sigma_1)^{\frac{-\mu-1}{2}} \sigma_1^{-1}\sigma_2^{-1} \right)^3 =1~,$

\medskip

\item $u^2 = \left[g_{+}(a_1,b_1), u^{-1}g_{-}(a_1,b_1)u^{-1} \right]~,$

\medskip
\end{enumerate}

in the reduced braid group of the torus $\overline{\mathrm{B}}_3(\mathbb{T})$, where the bracket denotes the commutator and where

\[
u = f(\sigma_1^2,\sigma_2^2)^{-1} \sigma_1^\mu f(\sigma_1^2,\sigma_2^2)~.
\] 

\medskip

Under this bijection, the group structure of the elliptic Grothendieck--Teichmüller group is given by: 
\[
(\mu_1,f_1,(g_1)_{\pm}) \star (\mu_2,f_2,(g_2)_{\pm}) = (\mu,f,g_{\pm})~,
\]

where 

\[
\mu = \mu_1\mu_2~, \quad \text{where} \quad f(x,y) = f_1\left(x^{\mu_2},f_2(x,y)y^{\mu_2}f_2(y,x)\right)f_2(x,y)~,
\]
\vspace{0.1pc}

and where 

\[
g_{\pm}(A,B) = (g_1)_{\pm}((g_2)_{+}(A,B),(g_2)_{-}(A,B))~.
\]
\vspace{0.1pc}

And, under this bijection, its action on the set of elliptic associators is given by:
\[
(\mu,f,g_{\pm}) \bullet (\lambda,\Phi,A_{\pm}) =
\]
\[
\left(\mu \lambda, f\left(e^{\lambda t_{12}}, \Phi(t_{12},t_{23})e^{\lambda t_{23}}\Phi(t_{23}, t_{12})\right)\Phi(t_{12},t_{23}),g_{\pm}(A_{+},A_{-})\right)~.
\]
\medskip
\end{theorem}
\begin{proof}[Main steps of proof]
The proof of the above theorem goes along the following lines: 
\begin{itemize}
\item Once the data of $(\mu,f)\in \mathrm{GR}$ is fixed, lifting the corresponding automorphism of $\widehat{\pab}(\kk)$ to an automorphism 
of the right $\widehat{\pab}(\kk)$-module $\widehat{\pab}_{\mathrm{ell}}(\kk)$ requires to provide images $g_+(A,B)$ and $g_-(A,B)$ of $A$ and $B$, respectively; 
\item According to Theorem \ref{thm: presentation pab elliptic} these images, together with the images of $R$ and $\Phi$ determined by $(\mu,f)$, should satisfy the two nonagon relations and the mixed reation;
\item One can easily check that equation (1) in the above theorem corresponds to the two nonagon relations. Conjugating with $f(\sigma_1^2,\sigma_2^2)$, one can then rewrite equation (2) in an equivalent form that directly 
corresponds to the mixed relation: 
\[
\sigma_1^{2\mu}=\left[f(\sigma_1^2,\sigma_2^2)g_+(a_1,b_1)f(\sigma_1^2,\sigma_2^2)^{-1}\,,\,\sigma_1^{-\mu}f(\sigma_1^2,\sigma_2^2)g_-(a_1,b_1)f(\sigma_1^2,\sigma_2^2)^{-1}\sigma_1^{-\mu}\right]\,;
\]
\item The remaining statements (about the group structure and the action on elliptic associators) are explicit calculations. 
\end{itemize}
\end{proof}

\begin{remark}
Using Remark \ref{rmk:presentation of PaCD elliptic}, one can also give an explicit description of the \textit{elliptic graded Grothendieck--Teichmüller group} $\mathrm{GRT}_{\mathrm{ell}}$.
\end{remark}

\subsection{Topological description of $\pab_{\mathrm{ell}}$}\label{ssec: topological description elliptic}
We consider the Fulton-MacPherson compactifications $\overline{\mathrm{C}}(\mathbb{T},\mathrm{I})$ of the reduced configuration spaces of the torus $\mathrm{C}(\mathbb{T},\mathrm{I})$, where $\mathrm{I}$ is a set with $n$ elements. One can compute that the boundary 

\[
\partial~\overline{\mathrm{C}}(\mathbb{C},\mathrm{I}) \cong \bigcup_{ k \geq 0} \bigcup_{\mathrm{J_1} \sqcup \cdots \sqcup\mathrm{J_k} = \mathrm{I}}  \overline{\mathrm{C}}(\mathbb{T},[k]) \times \left(\prod_{j =1}^k \overline{\mathrm{C}}(\mathbb{C},\mathrm{J}_j) \right)
\]
\medskip

where $[k] = \{1, \cdots, k \}$. Therefore the inclusions 

\[
\gamma_{\mathrm{J_1}, \cdots, \mathrm{J_k}}:  \overline{\mathrm{C}}(\mathbb{T},[k]) \times \left(\prod_{j =1}^k \overline{\mathrm{C}}(\mathbb{C},\mathrm{J}_j) \right) \hookrightarrow \partial~\overline{\mathrm{C}}(\mathbb{T},\mathrm{J_1} \sqcup \cdots \sqcup \mathrm{J_k}) \hookrightarrow \overline{\mathrm{C}}(\mathbb{T},\mathrm{J_1} \sqcup \cdots \sqcup \mathrm{J_k})~.
\]
\medskip

endow the family $\big(\overline{\mathrm{C}}(\mathbb{T},\mathrm{I})\big)_{n\geq0}$ with a right module structure over the operad 
$\overline{\mathrm{C}}(\mathbb{C})$. We denote this right module by $\overline{\mathrm{C}}(\mathbb{T})$.

\medskip

Consider again $\mathcal{P}a$, viewed as a right module over itself in the category of topological spaces when endowed with the discrete topology. 
There is an inclusion of topological right modules over $\mathcal{P}a$
\[
\iota: \mathcal{P}a \hookrightarrow \overline{\mathrm{C}}(\mathbb{T})~,  
\]
that arrizes from the embedding $(0,1)\hookrightarrow S^1\times S^1=\mathbb{T}$ given by $t\mapsto (\overline{t},\overline 0)$. 

\begin{theorem}
There is an isomorphism of right modules over $\pab$ in the category of groupoids
\[
\pab_{\mathrm{ell}} \cong \Pi_1(\overline{\mathrm{C}}(\mathbb{T}), \mathcal{P}a)~.
\]
\end{theorem}

\begin{remark}
In \cite{calaque2020ellipsitomic}, $\pab_{\mathrm{ell}}$ is in fact \textit{defined} using its topological description. 
In \cite{Enriquez14}, Enriquez proves that the algebraic definition satisfies a presentation by generators and relations, that is essentially the content of Theorem \ref{thm: presentation pab elliptic} (even though Enriquez doesn't use operads). 
The paper \cite{calaque2020ellipsitomic} implicitely contains a translation of Enriquez's definition in terms of operads, and the proof that it coincides with the topological definition. 
\end{remark}

\subsection{An overview of the ellipsitomic case}\label{ssec: overview of the ellipsitomic case}
One can think of the ellipsitomic case as a combination of the elliptic and the cyclotomic cases. The constructions performed in this case in order to define ellipsitomic associators follow the same steps as in the previous cases. For precise statements, we refer to \cite[Sections 4--6]{calaque2020ellipsitomic}.

\medskip

\subsubsection*{First player}

Let $\Gamma \coloneqq \mathbb{Z}/N\mathbb{Z} \times \mathbb{Z}/M\mathbb{Z}$, where $N,M \geq 1$. Recall that the fundamental group of the topological 
torus $\mathbb{T}$ is $\mathbb{Z} \times \mathbb{Z}$. Therefore $\Gamma$ can be obtained as a quotient of the fundamental group of the torus, meaning there 
is an unique covering space 

\[
p^\Gamma : \tilde{\mathbb{T}} \twoheadrightarrow \mathbb{T}~,
\]

associated to the canonical quotient map  $\mathbb{Z} \times \mathbb{Z}\twoheadrightarrow \Gamma$. We consider the reduced associated $\Gamma$-twisted configuration space 

\[
\mathrm{C}(\mathbb{T},n, \Gamma ) \coloneqq \left\{ (x_1, \cdots, x_n) \in \left(\tilde{\mathbb{T}}\right)^n ~~|~~p^\Gamma(x_i) \neq p^\Gamma(x_j)~~\text{if}~~i \neq j \right\}/ \tilde{\mathbb{T}}~.
\]

It has a natural $\Gamma$-action and a right module structure over the operad of configuration spaces. One can define the right module of \textit{parenthesized ellipsitomic braids} as the fundamental groupoid 

\[
\pab^\Gamma_{\mathrm{ell}} \coloneqq \Pi_1(\mathrm{C}(\mathbb{T},n, \Gamma), \mathcal{P}a^\Gamma )~,
\]
\vspace{0.1pc}

where $\mathcal{P}a^\Gamma$ is a very similar construction to that of Subsection \ref{subsec: topological cyclotomic}. It is a right module over $\pab$ endowed with a compatible $\Gamma$-action. It is generated as right module with a diagonally trivial 
$\Gamma$-action by two morphisms $A: 1_{(\bar{0},\bar{0})}2_{(\bar{0},\bar{0})} \longrightarrow 1_{(\bar{1},\bar{0})}2_{(\bar{0},\bar{0})}$ and 
$B: 1_{(\bar{0},\bar{0})}2_{(\bar{0},\bar{0})} \longrightarrow 1_{(\bar{0},\bar{1})}2_{(\bar{0},\bar{0})}$ are in $\pab^\Gamma_{\mathrm{ell}}(2)$. 
%\[
%\includegraphics[width=65mm,scale=1]{elementABcyclo.eps}
%\]

\medskip

In fact, one can give an explicit presentation of $\pab^\Gamma_{\mathrm{ell}}$ analogue to Theorem \ref{thm: presentation pab elliptic}, see \cite[Theorem 4.5]{calaque2020ellipsitomic}. After considering the appropriate Malcev completions, this gives the first ellipsitomic player. 

\medskip

\subsubsection*{Second player}
One defines the \textit{infinitesimal ellipsitomic braids} Lie algebra $\mathfrak{t}_n^{\mathrm{ell},\Gamma}$ by adding generators $\{t_{ij}^\gamma\}$ for 
$\gamma \in \Gamma$ in weight $(1,1)$ to the generators in Definition \ref{def: elliptic Drinfeld-Kohno} and then slightly modifying the elliptic relations they satisfy. 

\medskip

This symmetric sequence in the category of graded Lie algebras with a $\Gamma$-action has a right module structure over the operad $\{\mathfrak{t}_n\}$, given by analogous formulas to those of Proposition \ref{prop: right mod structure}, and which is compatible with the $\Gamma$-action. 

\medskip

Once again, one considers a \textit{reduced} version $\overline{\mathfrak{t}}_n^{\mathrm{ell},\Gamma}$ of these algebras by modding out the central elements $\sum_i \alpha_i$ and $\sum_i \beta_i$. 
One defines the symmetric sequence of \textit{ellipsitomic chord diagrams} $\mathcal{CD}_{\mathrm{ell}}^{\Gamma}$ by taking the universal enveloping algebra functor. 

\medskip

One then defines the right module $\mathcal{P}a\mathcal{CD}_{\mathrm{ell}}^{\Gamma}$ of \textit{parenthesized ellipsitomic chord diagrams} by performing 
analogue constructions to those already explained in Subsection \ref{subsec: inf cyclo braids}. It has once again a compatible $\Gamma$-action. By considering 
the completion with respect to the total weight filtration, one gets the second ellipsitomic player.

\begin{remark}
Contrary to the cyclotomic case, there is no \textit{frozen strand} neither on ellipsitomic braids nor on ellipsitomic chord diagrams. 
\end{remark}

\subsubsection*{Defining ellipsitomic associators}

Once this machinery is setup, the rest is completely analogue to the previous cases.

\begin{definition}[Ellipsitomic associators] \label{Def: ellipsitomic associator}
The set of \textit{ellipsitomic associators} is given by 
\[
\mathrm{Assoc}_{\mathrm{ell}}^{\Gamma}(\kk) \coloneqq \mathrm{Iso}^+_{\mathsf{RMod}(\mathsf{p.u}\text{-}\mathsf{Grpd}_{\kk})}\left(\left(\widehat{\pab}^\Gamma_{\mathrm{ell}}(\kk),\widehat{\pab}(\kk)\right),\left(\mathrm{Grp}(\widehat{\pacd}^\Gamma_{\mathrm{ell}}),\mathrm{Grp}(\widehat{\pacd})\right)\right)^\Gamma~.
\]
This amounts to the data of:

\medskip

\begin{enumerate}
\item an isomorphism $F$ between $\widehat{\pab}(\kk)$ and $\mathrm{Grp}(\widehat{\pacd})$ which is the identity on objects (an associator)~,

\medskip

\item an isomorphism $W$ between $\widehat{\pab}^\Gamma_{\mathrm{ell}}(\kk)$ and $\mathrm{Grp}(\widehat{\pacd}^\Gamma_{\mathrm{ell}})$ of right modules which is $\Gamma$-equivariant and is the identity on objects, and which lies above the isomorphism $F$. 
\end{enumerate}
\end{definition}

It is naturally a bitorsor over the \textit{ellipsitomic Grothendieck--Teichmüller group} $\mathrm{GT}_{\mathrm{ell}}^\Gamma(\kk)$ and the 
\textit{ellipsitomic Grothendieck--Teichmüller group} $\mathrm{GRT}^\Gamma_{\mathrm{ell}}(\kk)$, whose definitions should be straightforward by now. 

\medskip 

Using the presentation of $\widehat{\pab}^\Gamma_{\mathrm{ell}}$ mentioned before, one can give an explicit description of the set of ellipsitomic associators analogue to Theorem \ref{thm: explicit ellip assos} and an explicit description of the set of elements in the ellipsitomic Grothendieck--Teichmüller group analogue to Theorem \ref{thm: explicit ellip GT}. There is also an explicit description of elements in the \textit{graded} version of the ellipsitomic Grothendieck--Teichmüller group. 

\begin{theorem}[{\cite[Section 6]{calaque2020ellipsitomic}}]
Let $\kk$ be a field of characteristic zero. The set of ellipsitomic associators $\mathrm{Assoc}^{\Gamma}_{\mathrm{ell}}(\kk)$ is non empty. 
\end{theorem}

As in all the previous cases (``usual'' associators, cyclotomic associators, elliptic associators), the proof of the existence goes into two steps: (1) one proves that ellipsitomic associators exists over $\mathbb{C}$ using analytic methods, 
(2) one deduces that they exist over $\mathbb{Q}$ using descent methods. The second step isn't very difficult, and relies on the fact that ellipsitomic associators form a torsor over the ellipsitomic (graded) Grothendieck--Teichmüller group. 
The first step is the most difficult one: in \cite{EllipsitomicKZB}, a flat connection is constructed on the $\Gamma$-twisted configuration space of the torus, whose renormalized holonomies along appropriate paths gives rise to an ellipsitomic associator. 

\subsubsection*{Concluding remarks}

The ellipsitomic constructions satisfy some compatibilities with surjective morphisms $\Gamma_1 \twoheadrightarrow \Gamma_2$, that could probably be phrased in functorial terms. 
In would be interesting to investigate this, even in the simplest case of isomorphisms $\Gamma_1 \tilde\to \Gamma_2$, this would allow to define and study ellipsitomic associators associated 
with any finite quotient map $\mathbb{Z}^2\to \Gamma$. 

\medskip

A question remains open: what is the Lie theoretic deformation problem to which ellipsitomic associators provide a universal solution? 
Section 5 of \cite{EllipsitomicKZB} already provides some insight: realizations of (some aspects of) the second ellipsitomic player into Lie theoretic terms are given. 

\medskip

The connection involved in the construction of ellipsitomic associators given in \cite{EllipsitomicKZB} depends on the choice of an elliptic curves equipped with a $\Gamma$-level structure. 
Moreover, this connection extends to the whole moduli space of elliptic curves with $\Gamma$-level structure and marked points. As a result, the corresponding ellipsitomic associators 
(obtained as suitable renormalized holonomies) should have many interesting features similar to the ones of elliptic associators: 
\begin{itemize}
\item Their coefficients are interesting numbers, that should be expressed as iterated integrals of Eiseinstein series associated with a congruence subrgoup $\mathrm{SL}^\Gamma_2\subset\mathrm{SL}_2(\mathbb{Z})$. 
\item Modular invariance (with respect to the congruence subgroup $\mathrm{SL}^\Gamma_2$) of these Eiseistein series have been investigated in the last section of \cite{EllipsitomicKZB}. It would be interesting to understand 
the modular invariance properties of ellipsitomic associators themselves. 
\item All of this should probably lift to a more abstract motivic context, yet to be determined. 
\end{itemize}

\appendix

\newpage

%%%%% Appendix on pro-unipotent completions %%%%%

\section{Pro-unipotent completions}\label{Subsection: Pro-unipotent completions}
An \textit{unipotent algebraic group} $G$ over $\kk$ is an algebraic group which embeds into $\mathbb{UT}_n$, the algebraic group given by the upper-triangular matrices, for some $n$ in $\mathbb{N}$. A classical result states that the category of affine algebraic groups over $\kk$ is equivalent to the category of commutative Hopf algebras over $\kk$ which are 
finitely generated as algebras. This gives that the category of pro-affine algebraic groups over $\kk$ is to the category of commutative Hopf algebras over $\kk$, without any finite generation assumptions.

\medskip

Likewise, one can show that the category of unipotent algebraic groups over $\kk$ is equivalent to the category of commutative Hopf algebras over $\kk$ which are finitely generated as algebras and which are \textit{conilpotent} as coalgebras. 

\begin{proposition}
There is a contravariant equivalence of categories between the category of commutative conilpotent Hopf algebras over $\kk$ and the pro-category of unipotent 
algebraic groups over $\kk$:
\[
\begin{tikzcd}[column sep=5pc,row sep=3pc]
            \mathsf{Hopf}\text{-}\mathsf{alg}^{\mathsf{conil}} \arrow[r, shift left=1.1ex, "\mathrm{Spec}(-)"{name=F}] &\mathsf{pro}(\mathsf{UniGrp})^{\mathsf{op}}~, \arrow[l, shift left=.75ex, "\mathcal{O}(-)"{name=U}]
\end{tikzcd}
\]
given by taking the spectrum $\mathrm{Spec}(-)$ of the underlying commutative algebra and by taking the global sections of the structure sheaf $\mathcal{O}(-)$.
\end{proposition} 

\begin{remark}
We refer to the survey \cite{Vezzani} on this subject for more details.
\end{remark}

\subsection{How to pro-unipotently complete?}
By the above discussion, it becomes clear that a \textit{pro-unipotent completion} over $\kk$ amounts to the construction of a 
commutative conilpotent Hopf algebra over $\kk$, which is universal in a suitable sense.

\medskip

Let us construct it for a finitely generated abstract group $G$. The first step is to take the group algebra $\kk[G]$ of $G$ over $\kk$, which is naturally a cocommutative Hopf algebra. Now the goal is to dualize it in such a way that 
it gives a commutative conilpotent Hopf algebra. Notice that, under finiteness assumptions, complete algebras are dual to conilpotent coalgebras in a sense that is explained below.

\begin{definition}[$I$-adic completion] \label{Def: Canonical completion}
Let $A$ be an augmented $\kk$-algebra, let $I$ be its augmentation ideal. The $I$\textit{-adic completion} of $A$ is given by the following limit
\[
A^\wedge_I \coloneqq \lim_{n \in \mathbb{N}^*} A/I^n~,
\]
taken in the category of $\kk$-algebras.
\end{definition}

This pro-nilpotent $\kk$-algebra can be dualized under some minor finiteness assumptions.

\begin{definition}[Topological dual] 
Let $A$ be an augmented $\kk$-algebra, let $I$ be its augmentation ideal and suppose $A/I^n$ is finite dimensional over $\kk$ for all $n \geq 1$. The \textit{topological dual} of $A$ is given by 
\[
A^\vee_I \coloneqq \colim{n \in \mathbb{N}^*}(A/I^n)^*~,
\]
where each $(A/I^n)^*$ is the coalgebra given by the linear dual of $A/I^n$ and where the colimit is taken in the category of coalgebras. 
\end{definition}

The topological dual $A^\vee$ is a conilpotent coalgebra. Indeed, it is pretty straightforward that $(A/I^n)^*$ is conilpotent for all $n$, and since the category of conilpotent coalgebras is stable under colimits, so is $A^\vee$. The topological dual functor sends comonoids to monoids, thus the topological dual of a Hopf algebra 
is a conilpotent Hopf algebra. 

\begin{remark}
Note that the algebra given by the linear dual of $A^\vee_I$ is precisely the completed algebra $A^\wedge_I$. 
\end{remark}

\begin{lemma}
Let $A$ be an augmented $\kk$-algebra and let $I$ be its augmentation ideal. If $I/I^2$ is finite dimensional over $\kk$, then $A/I^n$ is finite dimensional 
over $\kk$ for all $n \geq 1$. 
\end{lemma}

\begin{proof}
By a straightforward \textit{dévissage} argument, $A/I^n$ is finite dimensional for every $n$ if and only if $I^n/I^{n+1}$ is finite dimensional. For every 
$n \geq 1$, there is a $\kk$-linear surjection $(I/I^2)^{\otimes n} \twoheadrightarrow I^n/I^{n+1}$ given by the multiplication, which implies that 
$I^n/I^{n+1}$ is finite dimensional over $\kk$.
\end{proof}

Let $I$ be the augmentation ideal of $\kk[G]$. One computes that $I/I^2 \cong G^{ab} \otimes \kk$, where $G^{ab}$ denotes the abelianization of $G$. 
Thus if $G$ is a finitely generated group, $I/I^2$ is finite dimensional over $\kk$ and $\kk[G]_I^{\wedge}$ admits a topological dual. 

\begin{example}\label{example: PBn fini}
It follows from \cite{Arnold} that $\mathrm{H}_1(\mathrm{Conf}_n(\mathbb{C});\mathbb{Z})$ is a free $\mathbb{Z}$-module, generated by $\frac{n(n-1)}{2}$-generators. Therefore the abelianization of the pure braid group $\mathrm{PB}_n$ is given by 
\[
\mathrm{PB}_n^{ab} \cong \mathbb{Z}^{\frac{n(n-1)}{2}}~.
\]
Hence $\kk[\mathrm{PB}_n]$ admits a topological dual for all $n$ in $\mathbb{N}$. \hfill$\triangle$
\end{example}

\begin{definition}[Pro-unipotent completion]
Let $G$ be a finitely generated abstract group. Its \textit{pro-unipotent completion} $G^{\mathrm{uni}}$ is the pro-unipotent algebraic group given by
\[
G^{\mathrm{uni}} \coloneqq \mathrm{Spec}(\kk[G]_I^{\vee})~.
\]
\end{definition}

\begin{remark}
The pro-unipotent group $G^{\mathrm{uni}}$ is the initial object in the category of pro-unipotent groups $U$ endowed with a group morphism $G \longrightarrow U(\kk)$, where $U(\kk)$ refer to the $\kk$-points of $U$.
\end{remark}

\subsection{Malcev completion}
Very closely related to the pro-unipotent completion of an abstract group $G$ is the notion of its Malcev completion. 

\begin{definition}[Group-like elements]
Let $A$ be a Hopf-algebra. An element $x$ in $A$ is said to be \textit{group-like} if $\Delta(x) = x \otimes x$ and $\epsilon(x) =1$. The set of group-like elements 
of $A$ is denoted by $\mathrm{Grp}(A)$.
\end{definition}

Group-like elements form a group, where the multiplication is given by the associative product of $A$, and where the inverse is given by the antipode map of the Hopf-algebra.

\begin{definition}[Malcev completion]
Let $G$ be a finitely generated abstract group. Its \textit{Malcev completion} $\widehat{G}(\kk)$ is the abstract group 
\[
\widehat{G}(\kk)\coloneqq\mathrm{Grp}(\kk[G]_I^{\wedge})
\]
given by the group-like elements of the $I$-adic completion of $\kk[G]$. 
\end{definition}

The Malcev completion of a group $G$ corresponds to the $\kk$-points of its pro-unipotent completion.

\begin{proposition}
Let $G$ be a finitely generated abstract group. There is a group isomorphism
\[
G^{\mathrm{uni}}(\kk) \cong \widehat{G}(\kk)~.
\]
\end{proposition}

\begin{proof}
There is an inclusion
\[
G^{\mathrm{uni}}(\kk) \coloneqq \mathrm{Hom}_{\mathsf{Alg_\kk}}(\kk[G]_I^{\vee}, \kk) \subset \mathrm{Hom}_{\mathsf{Vect_\kk}}(\kk[G]_I^{\vee}, \kk) \cong \kk[G]_I^{\wedge}~.
\]
Thus the $\kk$-points of the pro-unipotent completion of $G$ will be given by certain kind of element in $\kk[G]_I^{\wedge}$. One can check that $x$ in $\kk[G]_I^{\wedge}$ is a group-like element if and only if $x$ in $\mathrm{Hom}_{\mathsf{Vect_\kk}}(\kk[G]_I^{\vee}, \kk)$ is a morphism of $\kk$-algebras.
\end{proof}

\begin{remark}
Let $B$ be a commutative $\kk$-algebra. One can check that 
\[
G^{\mathrm{uni}}(B) = \mathrm{Hom}_{\mathsf{Alg_\kk}}(\kk[G]_I^{\vee}, B) \cong \mathrm{Grp}(\kk[G]_I^{\wedge} \widehat{\otimes} B)~,
\]
that is, the $B$-points of $G^{\mathrm{uni}}$ can be computed as the group-like elements of the complete Hopf algebra 
\[
\kk[G]_I^{\wedge} \widehat{\otimes} B \coloneqq \lim_{n \in \mathbb{N}^*} \left(\kk[G]/I^n \otimes B\right)~.
\]
So a "generalized" Malcev completion allows us to reconstruct the whole pro-unipotent algebraic group $G^{\mathrm{uni}}$.
\end{remark}

\begin{remark}
The Malcev completion of a group $G$ is a pro-nilpotent uniquely divisible group. Which we will refer to, by a slight abuse of language, as an \textit{abstract pro-unipotent} $\kk$\textit{-group}.
\end{remark}

\begin{proposition}
The Malcev completion defines a strong monoidal functor
\[
\widehat{(-)}(\kk): \left(\mathsf{Grp}^{\mathsf{f.g}},\times,\{e\}\right) \longrightarrow \left(\mathsf{p.u}\text{-}\mathsf{Grp}^{\mathsf{f.g}}_{\kk},\times,\{e\}\right)~,
\]
from the category of finitely generated abstract groups to the category of finitely generated abstract pro-unipotent $\kk$-groups. 
\end{proposition}

\begin{proof}
It is straightforward to check that the various steps of its construction commute with the cartesian product.
\end{proof}

\begin{remark}\label{rmk: not fully faithful inclusion of kk-groups}
The inclusion of abstract pro-unipotent $\kk$-groups into all groups is \textit{not} fully faithful in general, except when $\kk = \mathbb{Q}$. See \cite[Appendix A]{QuillenRHT} for more details. 
\end{remark}

\subsection{Malcev completion of groupoids.}\label{Completion pro-unipotente des groupoides}
There is a generalization of this procedure from groups to groupoids. Given a finitely generated groupoid $\mathcal{G}$, one can produce a category $\kk[\mathcal{G}]$ in the following way:

\begin{enumerate}
\medskip

\item The set of objects of $\kk[\mathcal{G}]$ is the same as the set of objects of $\mathcal{G}$. 

\medskip

\item Given two object $g,g'$ in $\mathcal{G}$,
\[
\mathrm{Hom}_{\kk[\mathcal{G}]}(g,g') \coloneqq \kk[\mathrm{Hom}_\mathcal{G}(g,g')]~.
\]
\end{enumerate}

Notice that for every object $g$ in $\kk[\mathcal{G}]$, its endomorphisms $\mathrm{End}_{\kk[\mathcal{G}]}(g)$ form a Hopf algebra. In general, 
$\mathrm{Hom}_{\kk[\mathcal{G}]}(g,g')$ has a cocommutative coalgebra structure compatible with the composition of morphisms. The category 
$\kk[\mathcal{G}]$ forms a \textit{Hopf groupoid}. See \cite[Chapter 9]{FresseGT1} for more on this notion. 

\medskip

For any $g,g'$ in $\kk[\mathcal{G}]$, we can complete the $\mathrm{Hom}_{\kk[\mathcal{G}]}(g,g')$ with respect to its augmentation ideal and get a category $\kk[\mathcal{G}]^\wedge$ enriched in cocommutative coalgebras, which forms a \textit{complete Hopf groupoid}. There is an analogue group-like element functor which produces a groupoid from a Hopf groupoid. 

\medskip

The groupoid obtained $\widehat{\mathcal{G}}(\kk)$ given by the group-like elements of $\kk[\mathcal{G}]^\wedge$ is called the 
\textit{Malcev completion} of $\mathcal{G}$. Notice that for each object $g$ in $\widehat{\mathcal{G}}(\kk)$, the set of automorphisms of $g$ is the $\kk$-points of a pro-unipotent group. In a slight abuse of terminology, we will refer to it as a \textit{pro-unipotent} $\kk$-\textit{groupoid}.

\medskip

The Malcev completion functor forms a strong monoidal endofunctor 
\[
\widehat{(-)}(\kk): \left(\mathsf{Grpd}^{\mathsf{f.g}},\times,\{e\}\right) \longrightarrow \left(\mathsf{p.u}\text{-}\mathsf{Grpd}^{\mathsf{f.g}}_{\kk},\times,\{e\}\right)~,
\]
from the category of finitely generated groupoids to the category of finitely generated pro-unipotent $\kk$-groupoids.

\subsection{How to compute the Malcev completion ?}
Over a field of characteristic zero, a pro-unipotent group is completely determined by its Lie algebra. Therefore, one way to compute the Malcev completion of a $G$ is by first computing its Lie algebra and then integrating it. 

\begin{definition}[Primitive elements]
Let $A$ be a Hopf-algebra. An element $x$ in $A$ is said to be \textit{primitive} if $\Delta(x) = 1 \otimes x + x \otimes 1$. The set of primitive elements of 
$A$ is denoted by $\mathrm{Prim}(A)$.
\end{definition}

\begin{remark}
Since $\kk$ is a field of characteristic different than $2$, the condition $\Delta(x) = 1 \otimes x + x \otimes 1$ implies that any primitive element is in the 
augmentation ideal of $A$. Indeed, if $x$ is a primitive element, we have that

\[
\epsilon(x) = (\epsilon \otimes \epsilon) \circ \Delta(x) = (\epsilon \otimes \epsilon)(1 \otimes x + x \otimes 1) = 2\epsilon(x)~,
\]
\vspace{0.1pc}

hence $\epsilon(x) = 0$. 
\end{remark}

The primitive elements of a Hopf-algebra form a Lie algebra, where the bracket is given by the skew-symmetrization of the associative product in $A$.

\begin{definition}[Exponential map]
Let $A$ be a Hopf-algebra and let $I$ be its augmentation ideal. Let 
\[
\widehat{I} \coloneqq \lim_{n \in \mathbb{N}^*} I/I^n
\]
denote the augmentation ideal of $A_I^{\wedge}$. The \textit{exponential map} is given by
\[
\begin{tikzcd}[column sep=1.5pc,row sep=0pc]
\mathrm{exp}: \widehat{I} \arrow[r]
&A_I^{\wedge} \\
~~~~~~~~x \arrow[r,mapsto] 
&\displaystyle \sum_{n \geq 0}\frac{x^n}{n!}~.
\end{tikzcd}
\]
\end{definition}

\begin{proposition}
Let $G$ be a finitely generated group. The exponential map defines an isomorphism of groups

\[
\mathrm{exp}: \left(\mathrm{Prim}(\kk[G]_I^{\wedge}),\mathrm{BCH},0 \right) \qi \widehat{G}(\kk)~,
\]
\vspace{0.1pc}

where the group law $\mathrm{BCH}$ is ~given by the Baker--Campbell--Hausdorff formula. The inverse of this bijection is given by the logarithm map. 
\end{proposition}

\begin{proof}
Let $x$ be an primitive element. Its image under the exponential map is a group-like element in $\mathrm{Grp}(\kk[G]_I^{\wedge})$:
\[
\Delta(\mathrm{exp}(x)) = \mathrm{exp}(1 \otimes x + x \otimes 1) = \mathrm{exp}(x) \otimes \mathrm{exp}(x)~.
\]
Furthermore, the exponential map is a morphism of groups when one considers the group structure on $\mathrm{Prim}(\kk[G]_I^{\wedge})$ given by the 
Baker--Campbell--Hausdorff formula, as it is the universal formula such that
\[
\mathrm{exp}(\mathrm{BCH}(x,y)) = \mathrm{exp}(x).\mathrm{exp}(y)~.
\]
Let $g$ be a group-like element. Notice that $\epsilon(g) =1$ implies that $(g-1)$ is in the augmentation ideal. Thus one can define the logarithm on group-like 
elements as 

\[
\mathrm{log}(g) = \mathrm{log}(1 - (g-1)) = \sum_{n \geq 1}\frac{(g-1)^n}{n}~.
\]

One can check that it is a morphism of groups, which is the inverse of the exponential map.
\end{proof}

This description of $\widehat{G}(\kk)$ gives a canonical morphism of groups 
\[
\begin{tikzcd}[column sep=1.5pc,row sep=0pc]
G \arrow[r]
&\widehat{G}(\kk) \\
g \arrow[r,mapsto] 
&\mathrm{exp}\big(\mathrm{log}(g)\big)~,
\end{tikzcd}
\]
since $g$ is group-like element of $\kk[G]_I^{\wedge}$. 

\begin{remark}
In particular, the element $g^\lambda$, for $\lambda$ in $\kk$, makes sense in $\widehat{G}(\kk)$. Indeed, one can define it as 
$g^\lambda \coloneqq \mathrm{exp}\big(\lambda.\mathrm{log}(g)\big)$.
\end{remark}

\begin{example}\leavevmode

\begin{enumerate}
\item One has that $\widehat{\mathbb{Z}}(\kk) \cong \kk$ and that $\widehat{\mathbb{Z}/m\mathbb{Z}}(\kk) = 0$.

\medskip

\item In general, if $G$ is an abelian group, then $\widehat{G}(\kk) \cong G \otimes_{\mathbb{Z}} \kk$. 

\medskip

\item The Malcev completion $\widehat{\mathbb{F}_n}(\kk)$ of the free group $\mathbb{F}_n$ generated by $x_1, \cdots, x_n$ is generated by the 
expressions $\{x_i^\lambda\}$ for $\lambda$ is in $\kk$ and $1 \leq i \leq n$.

\medskip

\item Using Lemma \ref{lemma: PB_3 iso au produit de libres}, one can compute that $\widehat{\mathrm{PB}_3}(\kk) \cong \widehat{\mathbb{F}_2}(\kk) \times \kk$. 
\hfill$\triangle$
\end{enumerate}
\end{example}

%%%%%%%% Appendix on operadic notions for the cocartesian monoidal structure %%%%%%%%%%

\section{Operads in cocartesian categories}\label{appendixB}

Let $\mathcal C$ be a category with finite colimits. We view is as a symmetric monoidal category: the monoidal product is the coproduct $\amalg$ and the unit is the initial object $\emptyset$. The main goal of this appendix is to show how to construct
operadic-like objects in $(\mathcal C,\amalg,\emptyset)$ from presheaves on certain categories of finite sets with values in $\mathcal C$. This formalism of presheaves just encodes a family of objects in $\mathcal{C}$ with types of insertion-coproduct morphisms. We deal with three cases: operads, operadic modules, and moperads. 

\subsection{The case of operads}\label{Appendix B1: operads}
Let $\mathrm{Fin}_*$ be the category of finite pointed sets, together with maps that respect the base point. This category is canonically equivalent to the category of finite sets with \textit{partially defined maps}. We consider two distinguished classes of morphisms in $\mathrm{Fin}_*$ (viewed as finite sets with partially defined maps):
\begin{itemize}
\item \emph{Active} morphisms are the totally defined map. 

\item \emph{Inert} morphisms are the partially defined bijections.  
\end{itemize}

\begin{example}
Let $J\subset K$ be an inclusion of finite sets. The partially defined map $\mathrm{id}_{J}^K: K \longrightarrow J$ which is given by the identity on $J$ is a partially defined bijection. We will refer to this class of maps as \textit{partially defined equalities}. \hfill$\triangle$
\end{example}

The pair (\emph{Inert}, \emph{Active}) is an orthogonal factorization system, meaning in particular that every morphism factors (uniquely up to a unique isomorphism) as an inert morphism followed by an active one. If, moreover, we require that the inert map is a partially defined equality, then the factorization becomes strictly unique. Now observe that active morphisms are generated by:
\begin{itemize}
\item bijections;
\item and contracting maps: if $I,J$ are set and $i\in I$, the contracting map $c_i:I\backslash\{i\}\sqcup J\to I$ is defined 
as $c_i(j)=j$ if $j\notin J$ and $c_i(j)=i$ if $j\in J$. Note that when $J$ is empty, the contracting map is injective. 
\end{itemize}

\begin{theorem}\label{thm: presheaves induce operads}
Any functor $\mathrm{Fin}_*^{op} \longrightarrow \mathcal C$ induces an operad in $(\mathcal C,\amalg,\emptyset)$.
\end{theorem}

\begin{proof}[Sketch of proof]
Let $\mathcal O:\mathrm{Fin}_*^{op}\to\mathcal C$ be a functor. We first restrict it to $\mathrm{Fin}_{bij}$, the category of finite sets with (totally defined) bijections; we thus get a specie of structure. 

\medskip

We define the partial compositions as 
\[
\circ_i:=\mathcal O(c_i)\amalg \mathcal O(\mathrm{id}_{J}^K): \mathcal O(I)\amalg \mathcal O(J)\to \mathcal O\big((I\backslash\{i\})\sqcup J\big)~,
\]

where $K=(I\backslash\{i\})\sqcup J$. One can check that equivariance, associativity and unitality relations defining an operad are implied by relations among generating morphisms of $\mathrm{Fin}_*^{op}$ (these generating morphisms being bijections, contracting maps, and partial equalities). 
\end{proof}

Suppose that $\mathcal C$ carries a symmetric monoidal structure $\otimes$ such that the identity functor is a symmetric colax monoidal functor $(\mathcal C,\amalg,\emptyset)\to(\mathcal C,\otimes,\mathbb{1})$. It might not send operads to operads in general. Nevertheless, if each $\circ_i$ factors through $\mathcal O(I)\otimes\mathcal O(J)$, then $\mathcal O$ becomes an operad in $(\mathcal C,\otimes,\mathbb{1})$. 

\begin{example}\label{example: symmetric monoidal cocartesian descent}
Here are two examples of such a symmetric monoidal category: 
\begin{itemize}
\item The category of unitial associative $\kk$-algebras, with monoidal product being the tensor product; 
\item The category of Lie $\kk$-algebras, with monoidal product being the direct sum (that is the cartesian monoidal structure). 
\end{itemize}
In both cases, the condition for a map $A_1\amalg A_2\to B$ to factor through the monoidal product of $A_1$ and $A_2$ is that the images 
of $A_1\to B$ and $A_2\to B$ commute. \hfill$\triangle$
\end{example}
This explains why the insertion-coproduct morphisms we consider in the main body of the paper define operads in 
Lie/associative $\kk$-algebras: 
\begin{itemize}
\item Insertion-coproduct morphisms define operads for the cocartesian monoidal structure; 
\item The image of insertion morphisms commutes with the image of coproduct morphisms, making the operad structure descend to the 
monoidal structure of interest. 
\end{itemize}

\subsection{The case of operadic modules}\label{Appendix B2: modules}
Let $\mathrm{Fin}_{*,\not*}$ be the following category: 
\begin{itemize}
\item Objects are finite sets with or without pointing; 
\item Morphisms are maps preserving the pointing. 
\end{itemize}
In other words, every object of $\mathrm{Fin}_{*,\not*}$ is isomorphic to $\{1,\dots,n\}$ or $\{*,1,\dots,n\}$, where $*$ is the base point, for some $n\geq0$. For every finite set $I$ there are two objects, denoted by $I$ and $I^+$, which are given by the finite set without or with a base point. There are three types of morphisms: pointed maps between pointed sets, maps between sets without pointing, and maps from a non-pointed set to a pointed one. Concretely, we have that:
\begin{itemize}
\item Morphisms from $I$ to $J$ are totally defined maps $I\to J$;
\item Morphisms from $I^+$ to $J^+$ are partially defined maps $I\to J$; 
\item Morphisms from $I$ to $J^+$ are partially defined maps $I\to J$;
\item There are no morphisms from $I^+$ to $J$. 
\end{itemize}

As we explained in Appendix \ref{Appendix B1: operads}, partially defined maps admit an orthogonal factorization system made of inert maps and active maps. Active maps are precisely totally defined maps, and these are generated by bijections and contracting maps. 

\begin{remark}\label{remarque-pedante}
There is a more pedantic (but useful!) way of defining $\mathrm{Fin}_{*,\not*}$. Consider the functor $F:\Delta^1\to \mathrm{Cat}$ representing the functor $(-)^+:\mathrm{Fin}\to\mathrm{Fin}_*$ sending $I$ to $I^+$, where $\mathrm{Fin}$ is the category of finite sets with totally defined maps. Then $\mathrm{Fin}_{*,\not*}\to \Delta^1$ is the cocartesian fibration associated to $F$. 
\end{remark}

\begin{theorem}
Any functor $\mathrm{Fin}_{*,\not*}^{op} \longrightarrow \mathcal C$ induces an operad together with a right operadic module over it in $(\mathcal C,\amalg,\emptyset)$.
\end{theorem}

\begin{proof}[Sketch of proof]
It follows from Remark \ref{remarque-pedante} that a functor $\mathrm{Fin}_{*,\not*}^{op}\to\mathcal C$ amounts to the data of a triple $(\mathcal M,\mathcal O,f)$ where 
\begin{itemize}
\item $\mathcal O$ is a functor $\mathrm{Fin}_*^{op}\to\mathcal C$, which by Theorem \ref{thm: presheaves induce operads} induces an operad in $(\mathcal C,\amalg,\emptyset)$;

\item $\mathcal M$ is a functor $\mathrm{Fin}^{op}\to\mathcal C$;

\item $f:\mathcal O((-)^+)\Rightarrow \mathcal M$ is a natural transformation.
\end{itemize}

Since totally defined maps are generated by bijections and contracting maps, the functor $\mathcal M$ induces a species of structure when we restrict it to bijections. Furthermore, it comes equipped with the following maps 
\[
\mathcal M(c_i): \mathcal M(I) \longrightarrow \mathcal M\big((I\backslash\{i\})\sqcup J\big)~,
\]
induced by the contracting maps. On the other hand, we consider the following composition of morphisms:
\[
\begin{tikzcd}
\mathcal O(J)  \arrow[r,"O(\mathrm{id}_{J}^K)"]
&\mathcal O\big((I\backslash\{i\})\sqcup J\big) \arrow[r,"f"]
&\mathcal M\big((I\backslash\{i\})\sqcup J\big)~,
\end{tikzcd}
\]
where $K=(I\backslash\{i\})\sqcup J$. These two aforementioned maps assemble into the following structure maps 
\[
\circ_i: \mathcal M(I)\amalg \mathcal O(J)\to \mathcal M\big((I\backslash\{i\})\sqcup J\big)~.
\]
Once again, one can check that all the axioms of a right module over an operad are satisfied, by using all the relations amongst the generation morphisms of the source category.
\end{proof}

\begin{remark}
Like in Appendix \ref{Appendix B1: operads}, the structure defined on the cocartesian monoidal structure can descend to other types of monoidal structures under certain conditions. In particular, the elliptic insertion-coproduct morphisms that we consider in the main body of the paper define a right module structure for the cocartesian monoidal structure in Lie/associative $\kk$-algebras. Since the image of insertion morphisms commutes with the image of coproduct morphisms, they induce a right module structure in the monoidal structure of interest. 
\end{remark}

\subsection{The case of moperads}\label{Appendix B3: moperads}
Let $\mathrm{bFin}_{*,*}$ be the following category: 
\begin{itemize}
\item Objects are finite sets with one or two pointings,
\item Morphisms are maps preserving the pointings. 
\end{itemize}
In other words, every object of $\mathrm{bFin}_{*,*}$ is isomorphic to either a pointed set $\{0=*,1,\dots,n\}$ or a double-pointed set $\{*,0,1,\dots,m\}$, where $*$ is the first base point and $0$ is the second base point. For every finite set $I$ there are two objects, denoted by $I^+$ and $I^{++}$, which are given by the finite set with a base point or with two base points. There are three types of morphisms: pointed maps between pointed sets, maps between sets without pointing, and maps from a non-pointed set to a pointed one. Concretely, we have that:
\begin{itemize}
\item Morphisms from $I^+$ to $J^+$ are partially defined maps $I\to J$; 
\item Morphisms from $I^{++}$ to $J^{+}$ are partially defined maps $I\to J$;
\item Morphisms from $I^{++}$ to $J^{++}$ are partially defined pointed maps $I^+ \to J^+$;
\item There are no morphisms from $I^+$ to $J^{++}$. 
\end{itemize}

\begin{example}
If $I,J$ are finite sets, there are pointed-contraction maps $c_0^1:I^+ \sqcup J^+\to I^+$ and $c_0^2:I^+ \sqcup J^+\to J^+$. The first map $c_0^1$ is defined by sending $I^+$ to itself and $J^+$ to the base point of $I^+$; the second map $c_0^2$ is defined by sending $J^+$ to itself and $I^+$ to the base point of $J^+$.
\end{example}

\begin{theorem}
Any functor $\mathrm{bFin}_{*,*}^{op} \longrightarrow \mathcal C$ induces an operad together with a moperad over it in $(\mathcal C,\amalg,\emptyset)$.
\end{theorem}

\begin{proof}[Sketch of a proof]
A functor $\mathrm{bFin}_{*,*}^{op}\to\mathcal C$ amounts to the data of a triple $(\mathcal M,\mathcal O,f)$ where 
\begin{itemize}
\item $\mathcal O$ is a functor $\mathrm{Fin}_*^{op}\to\mathcal C$, which by Theorem \ref{thm: presheaves induce operads} induces an operad in $(\mathcal C,\amalg,\emptyset)$;

\item $\mathcal M$ is a functor $\mathrm{pFin}_{*}^{op} \to \mathcal C$ from the opposite category of pointed sets with partially defined pointed maps to $\mathcal C$;

\item $g:\mathcal O(-)\Rightarrow \mathcal M((-)^+)$ is a natural transformation.
\end{itemize}

The functor $\mathcal M$ induces a species of structure by considering its precomposition with the functor $(-)^+$ and restricting it to bijections. 

\medskip

The right modules structure over $\mathcal O(-)$ is constructed as follows. On one hand, we consider the image of the contracting maps $c_i$ by the functor $(-)^+$:
\[
\mathcal M((c_i)^+): \mathcal M(I^+) \longrightarrow \mathcal M\big(\big((I\backslash\{i\})\sqcup J\big)^+\big)~,
\]
induced by the contracting maps. On the other hand, we consider the following composition of morphisms:
\[
\begin{tikzcd}
\mathcal O(J)  \arrow[r,"O(\mathrm{id}_{J}^K)"]
&\mathcal O\big((I\backslash\{i\})\sqcup J\big) \arrow[r,"g"]
&M\big(\big((I\backslash\{i\})\sqcup J\big)^+\big)~,
\end{tikzcd}
\]
where $K=(I\backslash\{i\})\sqcup J$. These two aforementioned maps assemble into the following structure maps 
\[
\circ_i: \mathcal M(I^+) \amalg \mathcal O(J) \longrightarrow \mathcal M\big(\big((I\backslash\{i\})\sqcup J\big)^+\big)~.
\]
One can check that all the axioms of a right module over an operad are satisfied, by using all the relations amongst the generation morphisms of the source category.

\medskip

The monoidal structure of $\mathcal{M}$ is constructed using the pointed-contraction maps. It is given by 
\[
\circ_0\coloneqq \mathcal M(c_0^1) \amalg \mathcal{M}(c_0^2): \mathcal M(I^+) \amalg \mathcal M(J^+) \longrightarrow \mathcal M\big(I^+ \sqcup J^+\big)~.
\]
Using the relations that the pointed-contraction maps satisfy, one can show that they define a monoid structure and that furthermore it is compatible with the right module structure over $\mathcal O$, thus forming a moperad. 
\end{proof}

\begin{remark}
Like in Appendix \ref{Appendix B1: operads}, the structure defined on the cocartesian monoidal structure can descend to other types of monoidal structures under certain conditions. In particular, the cyclotomic insertion-coproduct morphisms that we consider in the main body of the paper define a right module structure for the cocartesian monoidal structure in Lie/associative $\kk$-algebras. Since the image of insertion morphisms commutes with the image of coproduct morphisms, they induce a right module structure in the monoidal structure of interest. 
\end{remark}

\newpage

\bibliographystyle{halpha}
\bibliography{zebribs}
\end{document}